\documentclass[%
 aip,
 cha,
 amsmath,amssymb,
 preprint,
]{revtex4-1}

\usepackage{graphicx}
\usepackage{dcolumn}
\usepackage{bm}
\usepackage{amsthm}
\usepackage{booktabs}
\usepackage{subcaption}
\usepackage{xcolor}
\usepackage{dsfont}
\usepackage[normalem]{ulem}

\usepackage[utf8]{inputenc}
\usepackage[T1]{fontenc}
\usepackage{mathptmx}
\usepackage{etoolbox}
\usepackage[table]{xcolor}

\makeatletter
\def\@email#1#2{%
 \endgroup
 \patchcmd{\titleblock@produce}
  {\frontmatter@RRAPformat}
  {\frontmatter@RRAPformat{\produce@RRAP{*#1\href{mailto:#2}{#2}}}\frontmatter@RRAPformat}
  {}{}
}%
\makeatother

\newcommand{\DD}[2]{\frac{d #1}{d #2}}
\newcommand{\dd}[2]{\frac{\partial #1}{\partial #2}}
\newcommand{\norm}[1]{\left\lVert#1\right\rVert}

\makeatletter
\newcommand*{\addFileDependency}[1]{
  \typeout{(#1)}
  \@addtofilelist{#1}
  \IfFileExists{#1}{}{\typeout{No file #1.}}
}
\makeatother

\newcommand{\snor}[1]{\left| #1 \right|} 
\newcommand{\LRp}[1]{\left( #1 \right)} 
\newcommand{\LRs}[1]{\left[ #1 \right]} 
\newcommand{\LRc}[1]{\left\{ #1 \right\}} 

\newcommand{\mc}[1]{\mathcal{#1}} 
\newcommand{\half}{\frac{1}{2}}

\newcommand{\bs}[1]{\boldsymbol{#1}}

\newcommand{\R}{{\bs{\mathbb{R}}}}

\newcommand{\pb}{\bs{u}}
\newcommand{\ub}{\pb}

\newcommand{\figlab}[1]{\label{fig:#1}}
\newcommand{\eqnlab}[1]{\label{eq:#1}}
\newcommand{\theolab}[1]{\label{theo:#1}}

\newcommand{\tablab}[1]{\label{tab:#1}}

\newcommand{\figref}[1]{\ref{fig:#1}}
\newcommand{\theoref}[1]{\ref{theo:#1}}

\newcommand{\eqnref}[1]{\eqref{eq:#1}}

\newcommand{\seclab}[1]{\label{sect:#1}}
\newcommand{\secref}[1]{\ref{sect:#1}}
\newcommand{\tabref}[1]{\ref{tab:#1}}

\newcommand{\eval}[2][\right]{\relax \ifx#1\right\relax \left.\fi#2#1\rvert}

\renewcommand{\epsilon}{\varepsilon}
\newcommand{\vb}{\bs{v}}

\newcommand{\epsb}{\bs{\epsilon}}

\newcommand{\fb}{\bs{f}}

\newcommand{\ubhat}{\widehat{\ub}}

\newcommand{\fbhat}{\widehat{\fb}}
\newcommand{\fhat}{\widehat{f}}

\newtheorem{theorem}{Theorem}
\newtheorem{remark}{Remark}

\definecolor{heat}{RGB}{180,215,240}
\definecolor{timeheat}{RGB}{205,205,205}
\definecolor{lossblue}{RGB}{55,126,184}

\newcommand{\E}{\mathbb{E}}
\newcommand{\vbhat}{{\widehat{\vb}}}

\newcommand{\Jb}{{\bf{J}}}

\newcommand{\Jbhat}{\widehat{\bf{J}}}

\newcommand{\Hcal}{\mathcal{H}}
\newcommand{\Hcalhat}{\widehat{\mathcal{H}}}
\newcommand{\Frob}{\mathrm{F}}
\newcommand{\veps}{\bs{\epsilon}}
\newcommand{\eps}{{\epsilon}}
\newcommand{\Lcurv}{{ L_{\mathrm{curv}}}}

\begin{document}


\title[Second-order consistency for learning chaotic dynamics]{Second-order consistency for learning chaotic dynamics via randomized Jacobian matching}

\author{Shinhoo Kang}
 \email{shinkang@korea.ac.kr}
\affiliation{Department of Computer Science and Software Engineering, Korea University, Sejong 30019, Republic of Korea}

\author{Hai V. Nguyen}
\affiliation{Department of Aerospace Engineering and Engineering Mechanics, The Oden Institute for Computational Engineering and Sciences, The University of Texas at Austin, Austin, TX, USA}

\author{Tan Bui-Thanh}
\affiliation{Department of Aerospace Engineering and Engineering Mechanics, The Oden Institute for Computational Engineering and Sciences, The University of Texas at Austin, Austin, TX, USA}

\date{\today}

\begin{abstract}
Short-horizon accuracy does not ensure that a learned chaotic system has correct long-time dynamics. Trajectory (zeroth-order) matching constrains vector-field values, and Jacobian (first-order) matching constrains local tangent dynamics, but neither determines how the Jacobian varies away from supervised states. A model may be locally accurate while drifting toward spurious attractors and distorting long-time statistics. We show that second-order supervision mitigates these failures. Because forming full Hessian tensors is computationally prohibitive in high dimensions, we propose model-constrained randomized Jacobian matching, which compares the Jacobians of the true and learned vector fields at randomly perturbed inputs. A Taylor expansion shows that the expected randomized Jacobian loss decomposes into the Jacobian mismatch plus a Hessian mismatch scaled by the noise variance, implicitly enforcing second-order consistency at $\mathcal{O}(d^2)$ memory cost without forming the $\mathcal{O}(d^3)$ Hessian tensor. In Lorenz~63 with minimal temporal supervision, second-order supervision reduces invariant-measure error and Lyapunov-spectrum MSE in both paired comparisons and recovers the constant Hessian norm of the true bilinear field. Across five training seeds, explicit Hessian matching produces catastrophic Lyapunov outliers for four of five seeds, whereas randomized Jacobian matching produces none among 5,000 on-attractor rollouts. It also has the largest threshold in a directional capture scan, possibly reflecting the neighborhood coverage of its perturbed inputs. In coupled Lorenz~96, first-order methods enter spurious high-amplitude regimes as forcing increases, while second-order methods retain accurate marginals. Randomized Jacobian matching costs about the same as explicit Hessian matching on Lorenz~63 and 40\% less on Lorenz~96, with no reference Hessian evaluations during training.
\end{abstract}


\maketitle

\begin{quotation}
Learned surrogate models can match short trajectories yet converge to the wrong attractor over long integrations. Trajectory and Jacobian matching constrain the vector field and its local tangent dynamics but leave its curvature underconstrained. We introduce randomized Jacobian matching, which provides implicit second-order supervision using Jacobian evaluations at perturbed states. On Lorenz~63 and coupled Lorenz~96, the method recovers long-time statistics and derivative accuracy where lower-order methods fail. Across five training seeds on Lorenz~63, it is the only method with no observed catastrophic Lyapunov outlier. The comparison with explicit Hessian matching suggests an additional benefit from supervising a neighborhood of the training data.
\end{quotation}

\section{Introduction}
\seclab{intro}

Learning dynamical systems from time-series data has a long history in science and engineering~\citep{crutchfield1987equations,kantz2003nonlinear,brunton2016discovering,chen2018neural,karniadakis2021physics}.
Most existing methods train by minimizing discrepancies between observed and predicted trajectories~\citep{chen2018neural}, or by enforcing governing equations through residual minimization as in Physics-Informed Neural Networks (PINN)~\citep{karniadakis2021physics,raissi2019physics}.
Although these methods often achieve accurate short-term prediction, they do not explicitly constrain the long-time qualitative behavior of the learned dynamics~\citep{vlachas_data-driven_2018,park2024dynamical}.
In the context of dynamical systems, derivative-informed approaches, such as Jacobian matching~\citep{park2024dynamical,tian_jacobian-enforced_2024}, combine trajectory supervision with constraints on the local tangent sensitivities of the learned dynamics. Related ideas appear in gradient-enhanced PINN training~\citep{mohammadian2023gradient} and in knowledge distillation via input-Jacobian matching~\citep{zagoruyko2016paying,srinivas2018knowledge}.
\citet{park2024dynamical} found that Jacobian matching can recover the Lorenz~63 attractor, its invariant measure, and the associated Lyapunov spectrum. \citet{tian_jacobian-enforced_2024} also showed that enforcing Jacobian consistency during training improves the accuracy of the tangent linear and adjoint models of a neural-network-emulated Lorenz~96 system, making it suitable for data assimilation.

Trajectory and Jacobian matching share a structural limitation. Trajectory supervision constrains the value of the learned vector field, while Jacobian supervision constrains its local tangent behavior, but neither constrains how the Jacobian varies away from the supervised states. A learned model can be locally accurate at those states while still having incorrect second-order structure between them. In Lorenz 63, this discrepancy is particularly easy to diagnose because the true Hessian is spatially constant. This motivates enforcing second-order consistency in addition to first-order matching.

The most direct way to enforce curvature consistency is explicit Hessian matching, which penalizes discrepancies in second-order derivatives. In practice, however, this approach is often computationally prohibitive; explicit construction of the full Hessian tensor requires $\mathcal{O}(d^3)$ memory and at least $\mc{O}(d^2)$ automatic-differentiation (AD) passes through the vector field for a $d$-dimensional state space. We therefore seek a way to enforce second-order consistency without explicit Hessian construction.

The model-constrained (MC) framework in \cite{van2025model,nguyen2022model} motivates our approach. By evaluating the training loss at randomly perturbed inputs, the MC approach promotes Jacobian consistency without requiring derivative computations: a Taylor expansion shows that the expected squared error between true and learned trajectories at perturbed inputs penalizes the Jacobian mismatch to leading order.

We consider differentiable surrogate modeling of a known reference model, rather than equation discovery from trajectory observations alone. During offline training, we assume that the reference model can be initialized at perturbed states and queried for trajectories and vector-field Jacobians, or equivalent Jacobian actions. This access model arises when a high-fidelity simulator is available offline but too expensive for repeated use in ensemble forecasting, uncertainty quantification, parameter estimation, or control~\cite{benner2015survey}. Examples include numerical weather and ocean models equipped with tangent-linear or adjoint codes~\cite{errico1997adjoint,heimbach2005efficient}, ODE and PDE simulators implemented with AD, and closure or universal differential equation (UDE)~\cite{rackauckas2020universal} models in which only part of the dynamics is learned.

We propose model-constrained randomized Jacobian matching, which matches Jacobians at perturbed states and thereby induces implicit second-order supervision. Section~\secref{hessian-matching} shows that the expected randomized Jacobian loss contains a leading-order Hessian-mismatch term, together with a third-derivative cross term and a higher-order remainder. When the nominal Jacobian mismatch is controlled and the perturbation scale is sufficiently small, minimizing this loss promotes Hessian consistency without requiring the Hessian tensor to be formed. This makes the approach practical in high-dimensional settings where explicit Hessian matching is infeasible. Because the Jacobians are evaluated at perturbed states, the loss also extends derivative supervision from sampled states to a finite neighborhood of the training data.

We test our framework on two chaotic systems: the Lorenz~63 system (low-dimensional chaos with a single positive Lyapunov exponent) and the coupled Lorenz~96 system. The coupled Lorenz~96 system produces high-dimensional chaos in a $396$-dimensional multiscale slow--fast state space with multiple positive Lyapunov exponents. We focus learning on the quadratic advection term of the $36$ slow variables, using a structured architecture that encodes spatial locality and translation equivariance.

The main contributions of this paper are summarized as follows:
\begin{itemize}
    \item
We identify a structural limitation of trajectory and Jacobian matching: neither constrains the curvature of the vector field. This missing second-order constraint can lead to incorrect long-time dynamics even when the learned model reproduces local first-order behavior.
    \item
    We propose model-constrained randomized Jacobian matching, which provides implicit second-order supervision using only Jacobian evaluations. A Taylor expansion shows that its expected loss contains a leading-order Hessian-mismatch term. Each perturbation requires $\mathcal{O}(d)$ AD passes and $\mathcal{O}(d^2)$ memory without explicit Hessian construction.
    \item
On Lorenz~63, we show that randomized Jacobian matching improves long-time statistics and curvature recovery under minimal temporal supervision. Multi-seed and directional capture analyses characterize catastrophic Lyapunov outliers and suggest that neighborhood supervision provides robustness beyond curvature matching at sampled states.
    \item
For the coupled Lorenz~96 system, we demonstrate that randomized Jacobian matching extends second-order supervision without reference Hessians to a $396$-dimensional multiscale UDE setting and improves long-time fidelity under out-of-distribution forcing.
\end{itemize}

The remainder of the paper is organized as follows. Section~\secref{modelproblems} introduces the two benchmark systems, Lorenz~63 and coupled Lorenz~96. Section~\secref{hessian-matching} develops the randomized Jacobian matching framework and derives the leading-order Hessian-mismatch term in the expected loss. Section~\secref{numerical-results} reports the numerical experiments and compares the five supervision methods on both systems. Section~\secref{conclusion} concludes.

\section{Model Problems}
\seclab{modelproblems}
We evaluate our approach on two canonical chaotic dynamical systems that span low- and high-dimensional regimes: the Lorenz~63 system \citep{lorenz1963deterministic} and the coupled Lorenz~96 system \citep{lorenz1996predictability}.


\subsection{Lorenz~63 Model}

The Lorenz~63 model is a simplified dynamical system originally introduced to model atmospheric convection.
The system is governed by 
\begin{subequations}
\eqnlab{lorenz63}
\begin{align}
    \DD{x}{t} &= \sigma\LRp{y-x}, \\
    \DD{y}{t} &= x\LRp{\rho-z}-y, \\
    \DD{z}{t} &= xy - \beta z,
\end{align}
\end{subequations}
where $x$, $y$, and $z$ denote the state variables. 

Under the standard parameter set $\sigma=10$, $\rho=28$, and $\beta=8/3$, the Lorenz~63 system is chaotic and dissipative. Although the system is highly sensitive to initial conditions, making accurate long-term trajectory prediction impractical, its ergodicity ensures that time averages along a trajectory converge to statistical averages over the attractor \citep{eckmann1985ergodic}. 
Using this property, we evaluate the learned models with the invariant measure, the Wasserstein-1 distance~\citep{ruschendorf1985wasserstein,villani2009optimal}, and the Lyapunov spectrum ($\lambda_1>0$, $\lambda_2=0$, $\lambda_3<0$). The Lyapunov spectrum characterizes the tangent-space dynamics averaged over the attractor: $\lambda_1$ governs exponential sensitivity to initial conditions, while $\lambda_3$ reflects dissipation and contraction toward the attractor.

\subsection{Coupled Lorenz 96 Model}

The two-scale coupled Lorenz~96 model represents multiscale atmospheric dynamics, with nonlinear advection, dissipation, scale coupling, and external forcing \citep{lorenz1996predictability}. The model consists of two interacting physical scales: the slow variables $X_k$ ($k=1,2,\cdots,K$), representing large-scale atmospheric waves, and the fast variables $Y_{j,k}$ ($j=1,2,\cdots,J$), modeling rapidly varying small-scale dynamics such as convection. The system contains $N_{\mathrm{dof}}=K(1+J)$ degrees of freedom and is governed by the coupled equations:
\begin{subequations}
  \eqnlab{lorenz96}
\begin{align}
  \eqnlab{lorenz96-slow}
  \DD{X_k}{t} & = \underbrace{
  - X_{k-1} \LRp{X_{k-2}-X_{k+1} }
  }_{\text{nonlinear advection}}
  - X_k + F - h \overline{Y}_k,\\
  \eqnlab{lorenz96-fast}
  \frac{1}{c}\DD{Y_{j,k}}{t} &= - J Y_{j+1,k} \LRp{Y_{j+2,k} - Y_{j-1,k} } - Y_{j,k} + \frac{h}{J} X_k,
\end{align}
\end{subequations}
where $- X_{k-1} \LRp{X_{k-2}-X_{k+1} }$ is the nonlinear advection, $-X_k$ provides linear dissipation, 
$F$ is the external forcing parameter, and  
$-h\overline{Y}_k$ couples the slow variables to the mean of the fast variables, $\overline{Y}_k :=\frac{1}{J}\sum_{j=1}^J Y_{j,k}$. The parameter $c$ controls the time-scale separation, $J$ scales the amplitude of the fast variables relative to the slow variables, and $h$ determines the coupling strength between the two scales \citep{carlu2019lyapunov}.
Following \citet{lorenz1996predictability}, we adopt $h=1$, $K=36$, and $c=J=10$, yielding $N_{\text{dof}}=396$. Under this configuration, the fast variables evolve approximately ten times faster and with smaller characteristic amplitude than the slow variables.

Lorenz~63 is a prototype of low-dimensional chaos with a single positive Lyapunov exponent. Coupled Lorenz~96, by contrast, has high-dimensional chaos with multiple positive Lyapunov exponents. This multiscale interaction structure resembles key features of geophysical turbulence. The system is therefore a more challenging benchmark for evaluating long-term statistical fidelity, attractor preservation, and out-of-distribution generalization.



\section{Model-Constrained Randomized Jacobian Matching}
\seclab{hessian-matching}

The dynamical systems introduced in Section~\secref{modelproblems} can be written in the compact form
\begin{align}
    \eqnlab{ode}
    \DD{\ub}{t} = \fb (\ub),
\end{align}
where $\ub \in \R^d$ denotes the state vector and $\fb:\R^d \to \R^d$ is the corresponding vector field. For example, for the Lorenz~63 system in~\eqnref{lorenz63}, $\ub=(x,y,z)^T$ and $$\fb (\ub)=\LRp{\sigma(y-x),\; x(\rho-z)-y,\; xy - \beta z}^T.$$

\subsection{Neural ODE}
\seclab{node}

We use Neural Ordinary Differential Equations (NODEs)~\citep{chen2018neural} to approximate the continuous-time dynamics of~\eqnref{ode}. The true vector field $\fb$ is replaced by a neural network $\fbhat(\cdot;\theta):\R^d\to\R^d$, yielding
\begin{align}
    \eqnlab{node}
    \DD{\ub}{t} = \fbhat(\ub; \theta).
\end{align}
Given an initial condition $\ub_0$, the true system~\eqnref{ode} and the NODE~\eqnref{node} are numerically integrated to generate the reference trajectory $\{\ub_j\}_{j=1}^{m}$ and the predicted trajectory $\{\ubhat_j\}_{j=1}^{m}$, respectively.
Here, $j=1,\cdots,m$ denotes the discrete time index within the trajectory segment, with $u_j\approx u(t_0+j\Delta t)$, where $\Delta t$ is the integration time step. 

The network parameters $\theta$ are trained using batches of $n_b$ trajectory segments. Each segment $i$ is initialized by randomly selecting a training trajectory and a starting state $\ub_0^{(i)}$ along that trajectory, from which the NODE generates $m$ predictions $\ubhat_1^{(i)},\ldots,\ubhat_m^{(i)}$. The standard trajectory-matching loss is
\begin{equation}
\eqnlab{loss-node-data}
    \mc{L}_{data} = \frac{1}{n_b m} \sum_{i=1}^{n_b} \sum_{j=1}^{m} \norm{\ubhat_j^{(i)} -\ub_j^{(i)}}^2.
\end{equation}

Equation~\eqnref{loss-node-data} also defines a trajectory-only
baseline that is applicable when only trajectory observations are
available. The derivative-supervised methods considered here assume
additional offline access to a differentiable reference model that can
be initialized at queried states and evaluated along the resulting
trajectories. Standard Jacobian matching requires reference Jacobians
along nominal trajectories, while explicit Hessian matching additionally
requires reference Hessians. The proposed randomized Jacobian-matching
method requires reference Jacobians along trajectories initialized from
perturbed states, but never queries or forms reference Hessians.

\subsection{Jacobian matching and Hessian matching}

Throughout this paper, derivatives are taken with respect to the vector field $\fb$, which the NODE directly learns. We define the Jacobian matrices of the true and learned vector fields by
\begin{equation}
\Jb(\ub) := \dd{\fb}{\ub} (\ub) \in \R^{d\times d},
\qquad
\Jbhat(\ub) := \dd{\fbhat}{\ub} (\ub; \theta) \in \R^{d\times d},
\eqnlab{Jdef}
\end{equation}
and the corresponding Hessian tensors by
\begin{equation}
\Hcal_{ijk}(\ub) := \dd{^2 f_i}{u_j\,\partial u_k}(\ub),
\qquad
\Hcalhat_{ijk}(\ub) := \dd{^2 \fhat_i}{u_j\,\partial u_k}(\ub;\theta).
\eqnlab{Hdef}
\end{equation}
Here, $\Jb,\Jbhat\in\R^{d\times d}$ characterize the local linear sensitivity of the vector field, while $\Hcal,\Hcalhat\in\R^{d\times d\times d}$ characterize its local curvature.


The Jacobian-matching loss penalizes discrepancies between the true and learned Jacobians evaluated along observed trajectories:
\begin{equation}
\eqnlab{loss-node-jac}
    \mc{L}_{jac} = \frac{1}{n_b (m+1)} \sum_{i=1}^{n_b} \sum_{j=0}^{m} \norm{\Jbhat (\ub_j^{(i)}) - \Jb  (\ub_j^{(i)}) }^2_\Frob.
\end{equation}
Analogously, explicit Hessian matching penalizes discrepancies between the true and learned Hessian tensors:
\begin{equation}
\eqnlab{loss-node-hes}
    \mc{L}_{hes} = \frac{1}{n_b (m+1)} \sum_{i=1}^{n_b} \sum_{j=0}^{m} \norm{\Hcalhat (\ub_j^{(i)}) - \Hcal (\ub_j^{(i)})}_\Frob^2, 
\end{equation}
where $\|\Hcalhat-\Hcal\|_\Frob^2 := \sum_{i,j,k}\bigl(\Hcalhat_{ijk}-\Hcal_{ijk}\bigr)^2$.
Although $\mc{L}_{hes}$ provides direct second-order supervision, explicit construction of the full Hessian tensor requires $\mc{O}(d^3)$ memory and $\mc{O}(d^2)$ AD passes. 
Even when the reference simulator is differentiable and provides Jacobian information through tangent-linear, adjoint, or AD operations, constructing and storing the full Hessian tensor may be computationally impractical. We therefore develop an implicit second-order approach that requires reference Jacobians but does not form reference Hessians.

\subsection{Model-Constrained approach: Implicit Derivative Supervision}
\seclab{mc-framework}

A NODE defines a dynamical system, not merely a sequence predictor. If the learned vector field deviates from the true dynamics, even small perturbations in initial conditions can amplify over time, leading to large prediction errors and qualitative failures of the long-term dynamics. To improve robustness, we adopt the MC framework, in which losses are evaluated at randomly perturbed states. This framework induces \emph{implicit} derivative supervision: trajectory matching at perturbed states penalizes first-order (Jacobian) mismatch, and Jacobian matching at perturbed states penalizes second-order (Hessian) mismatch.

\subsubsection{Randomized Trajectory Matching}

We perturb each initial condition by Gaussian noise:
\begin{equation}
\eqnlab{randu}
    \vb_0^{(i)} = \ub_0^{(i)} + \veps^{(i)}, \qquad \veps^{(i)} \sim \mc{N}\LRp{0, \eta^2 \mathbf{I}_d},
\end{equation}
where $\bs{\epsilon}^{(i)}$ is a Gaussian perturbation with zero mean and standard deviation (noise level) $\eta > 0$.
Here, $i$ denotes the same trajectory-segment index as in \eqnref{loss-node-data}, and each segment receives an independently sampled perturbation $\veps^{(i)}$. The perturbations are resampled at every gradient step. 
We use Gaussian perturbations throughout this work. However, Gaussianity is not essential to the leading-order analysis. What matters is that the perturbations have zero mean and are isotropic. Deterministic or quasi-Monte Carlo perturbations with these properties could also be considered.

Integrating both the true and learned dynamics from $\vb_0^{(i)}$ produces ground-truth and predicted trajectories $\bigl\{\vb_j^{(i)}\bigr\}_{j=1}^{m}$ and $\bigl\{\vbhat_j^{(i)}\bigr\}_{j=1}^{m}$.
The model-constraint loss is then
\begin{equation}
    \eqnlab{loss-node-mc}
    \mc{L}_{mc} = \frac{1}{n_b m} \sum_{i=1}^{n_b} \sum_{j=1}^{m} \norm{\vbhat_j^{(i)} - \vb_j^{(i)}}^2.
\end{equation}

Consider a single perturbed state $\ub + \veps$ and let $\vb_1=\Phi(\ub + \veps)$ and $\vbhat_1=\widehat{\Phi}(\ub + \veps)$ where $\Phi$ and $\widehat{\Phi}$ denote the true and learned one-step flow maps, respectively.
Expanding the one-step flow maps $\Phi$ and $\hat{\Phi}$ in Taylor series about $\ub$ gives
\begin{align}
v_{1,i}
& =
\Phi_i(\ub) 
+ \LRp{\dd{\Phi_i}{\ub}} \epsb 
+ \half \epsb^T \dd{^2 \Phi_i}{\ub^2} \epsb + \mc{O}(\norm{\epsb}^3),\\
\hat{v}_{1,i}
& =
\hat{\Phi}_i(\ub) 
+ \LRp{\dd{\hat{\Phi}_i}{\ub}} \epsb 
+ \half \epsb^T \dd{^2 \hat{\Phi}_i}{\ub^2}\, \epsb + \mc{O}(\norm{\epsb}^3),
\eqnlab{flowmaps}
\end{align}
for $i=1,\cdots, d$. Following~\cite{van2025model}, a Taylor expansion shows that, for small $\eta$ (small random perturbation $\epsb$ accordingly),
\begin{multline*}
    {\E}_{\epsb}\LRs{\norm{\vbhat_1 - \vb_1}^2} 
    = \underbrace{\norm{\widehat{\Phi}(\ub) - \Phi(\ub)}^2}_{\text{Trajectory mismatch}} 
+ \eta^2 \underbrace{
\norm{\dd{\widehat{\Phi}}{\ub}(\ub) - \dd{\Phi}{\ub}(\ub) }_\Frob^2 
}_{\text{Jacobian mismatch of the flow maps}} \\
+ 
\eta^2 \sum_{i=1}^d \LRp{\widehat{\Phi}_i(\ub) - \Phi_i(\ub) } 
\text{tr}\LRp{
\dd{^2 \hat{\Phi}_i}{\ub^2} - 
\dd{^2 \Phi_i}{\ub^2}}
+ \mc{O}(\eta^4).
\end{multline*}
The third term arises from second-order cross terms in the Taylor expansion and becomes negligible as training drives the trajectory mismatch to zero.
Thus $\mc{L}_{mc}$ implicitly penalizes the Jacobian mismatch of the flow maps in addition to the trajectory mismatch, without computing flow-map Jacobians.

Expanding $\ub(t+\Delta t)=\Phi(\ub(t))$ in powers of $\Delta t$ using $\dot{\ub}=\fb(\ub)$ and $\ddot{\ub}=\Jb(\ub)\fb(\ub)$ gives
\begin{align*}
\Phi(\ub) & =
\ub  
+ \Delta t \fb(\ub) 
+ \frac{(\Delta t)^2}{2} \Jb (\ub) \fb(\ub)
+ \mc{O}(\Delta t^3),\\
\hat{\Phi}(\ub) & =
\ub  
+ \Delta t \fbhat(\ub) 
+ \frac{(\Delta t)^2}{2} \Jbhat (\ub) \fbhat(\ub)
+ \mc{O}(\Delta t^3).
\end{align*}
Subtracting yields the flow-map discrepancy:
\begin{align*}
    \widehat{\Phi}(\ub) - \Phi(\ub) &=
    \Delta t \LRp{\fbhat(\ub) - \fb(\ub)} 
    + \frac{(\Delta t)^2}{2} \LRp{
    \Jbhat (\ub) \fbhat(\ub)
    - \Jb (\ub) \fb(\ub)
    }
    + \mc{O}(\Delta t^3).
\end{align*}
Taking derivatives with respect to $\ub$, we obtain 
\begin{align*}
    \dd{\widehat{\Phi}}{\ub}(\ub) - \dd{\Phi}{\ub}(\ub) &= \Delta t \LRp{\Jbhat(\ub) - \Jb(\ub)} + \mc{O}(\Delta t^2).
\end{align*}
Therefore, for small $\Delta t$, the flow-map Jacobian mismatch coincides with the vector-field Jacobian mismatch up to the factor $\Delta t$. Hence $\mathcal{L}_{mc}$ implicitly penalizes $\|\Jbhat(\ub) - \Jb(\ub)\|_F^2$ without Jacobian evaluations during training.
At the same time, the flow-map Jacobian term induced by the perturbation satisfies
\[
\eta^2
\left\|
\frac{\partial\widehat{\Phi}}{\partial u}
-
\frac{\partial\Phi}{\partial u}
\right\|_F^2
=
\eta^2\Delta t^2
\|\widehat{\mathbf J}-\mathbf J\|_F^2
+O(\eta^2\Delta t^3).
\]
Thus, reducing $\Delta t$ weakens the implicit vector-field Jacobian supervision per step.

\begin{table}[t]
\centering
\caption{Verification of the small-$\Delta t$ approximation for randomized
trajectory matching. The reported value is
$\Delta t^2\|\widehat{\mathbf J}-\mathbf J\|_F^2/
\|\partial_u\widehat{\Phi}-\partial_u\Phi\|_F^2$,
averaged over $20{,}000$ held-out on-attractor states for Lorenz~63 model.}
\label{tab:dt-expansion}
\begin{tabular}{l|cccc}
\hline
$\Delta t$ & $0.001$ & $0.01$ (used) & $0.05$ & $0.1$ \\
\hline
\texttt{mc} & $1.021$ & $1.464$ & $0.466$ & $0.002$ \\
\hline
\end{tabular}
\end{table}

Table~\ref{tab:dt-expansion} evaluates this approximation for the \texttt{mc} model using $20{,}000$ held-out on-attractor states from the Lorenz~63 system. 
As \(\Delta t\) decreases, the ratio approaches unity, reaching \(1.021\) at \(\Delta t=0.001\), consistent with the small-\(\Delta t\) expansion. At the timestep used in our experiments, \(\Delta t=0.01\), the ratio is \(1.464\), so the leading-order approximation retains the correct order of magnitude. For larger \(\Delta t\), no monotonic behavior is expected because higher-order terms in the flow-map Jacobian mismatch are no longer negligible. At \(\Delta t=0.1\), the ratio drops to \(0.002\), indicating that the leading-order vector-field approximation no longer tracks the numerical one-step flow-map Jacobian mismatch.

\subsubsection{Randomized Jacobian Matching} 

We propose to match Jacobians at perturbed states:
\begin{equation}
    \eqnlab{loss-node-mcjac}
    \mc{L}_{mcjac} = \frac{1}{n_b (m+1)} \sum_{i=1}^{n_b} \sum_{j=0}^{m} \norm{\Jbhat(\vb_j^{(i)}) - \Jb(\vb_j^{(i)}) }_\Frob^2.
\end{equation}
This loss implicitly penalizes the full Hessian tensor mismatch.

\subsection{Model-Constrained Jacobian Matching: Implicit Hessian Matching} 

The following theorem formalizes the implicit Hessian penalty induced by $\mc{L}_{mcjac}$.

\begin{theorem}[Implicit Hessian Matching]
\theolab{implicit_hessian}
Let $\fb,\fbhat:\R^d\to\R^d$ be sufficiently smooth vector fields, with Jacobians $\Jb,\Jbhat$ and Hessian tensors $\Hcal,\Hcalhat$ defined in~\eqnref{Jdef}--\eqnref{Hdef}. 
Define the third-derivative tensors
$$ \mc{T}_{i\ell j k}(\ub):=\dd{^3 f_{i}}{u_{\ell}\,\partial u_j\,\partial u_k}(\ub),
\qquad 
\widehat{\mc{T}}_{i\ell j k}(\ub) := \dd{^3 \hat{f}_{i}}{u_\ell\, \partial u_j\,\partial u_k}(\ub;\theta),
$$ and let 
$\Delta \Jb(\ub) := \Jbhat(\ub) - \Jb(\ub)$, 
$\Delta \Hcal(\ub) := \Hcalhat(\ub) - \Hcal(\ub)$, and $\Delta \mc{T}(\ub):= \widehat{\mc{T}}(\ub)
- \mc{T}(\ub)$.
For $\veps\sim\mc{N}(0,\eta^2 I_d)$, the randomized Jacobian mismatch satisfies
\begin{multline*}
    {\E}_{\veps}\LRs{ \norm{\Delta \Jb (\ub+\veps)}_\Frob^2 }
\;=\;
\underbrace{\norm{\Delta \Jb(\ub)}_\Frob^2 }_{\substack{\text{Jacobian} \\ \text{mismatch}}}
\;+\; 
\eta^2\,\underbrace{\norm{\Delta \Hcal(\ub)}_\Frob^2}_{\substack{\text{implicit Hessian}  \text{mismatch}}} 
\;+\; \\
\eta^2 \sum_{i,\ell}
\Delta J_{i \ell}(\ub)\,
\operatorname{tr}\LRp{\Delta\mc{T}_{i \ell} (\ub)}
\;+\;
\mathcal{R}_\eta, 
\end{multline*}
where 
$\mathcal{R}_\eta$ collects higher-order terms in the perturbation scale. 
\end{theorem}

\begin{proof}
Expand each component of $\fb$ in a Taylor series around $\ub$:
$$
f_i(\ub + \veps)
= f_i(\ub) + \sum_j J_{ij}(\ub)\,\eps_j + \tfrac{1}{2}\sum_{j,k} \Hcal_{ijk}(\ub)\,\eps_j\eps_k + \mc{O}(\|\veps\|^3).
$$

Differentiating component $i$ with respect to $u_\ell$,
the first- and second-order terms produce $\Hcal_{i\ell j}$ and $\mc{T}_{i\ell j k}$, respectively; symmetry of mixed partials gives
$\Hcal_{i\ell j}=\Hcal_{ij\ell}$ and full symmetry of $\mc{T}_{i\ell j k }$ in ($\ell,j,k$).
This yields
$$
J_{i\ell}(\ub+\veps)
= J_{i\ell}(\ub) + \sum_{j} \Hcal_{i\ell j}(\ub)\,\eps_j 
+ \half\sum_{j,k} \mc{T}_{i\ell j k} (\ub) \eps_j \eps_k 
+ \mathcal{O}(\|\veps\|^3).
$$
Similarly, we obtain 
$$
\hat{J}_{i\ell}(\ub+\veps)
= \hat{J}_{i\ell}(\ub) + \sum_{j} \Hcalhat_{i\ell j}(\ub)\,\eps_j 
+ \half\sum_{j,k} \widehat{\mc{T}}_{i\ell j k} (\ub) \eps_j \eps_k 
+ \mathcal{O}(\|\veps\|^3).
$$
Subtracting the two expansions gives
\begin{equation*}
\Delta J_{i\ell}(\ub+\veps)
= \Delta J_{i\ell}(\ub)
+ \sum_{j} \Delta \Hcal_{i\ell j}(\ub)\,\eps_j 
+ \half\sum_{j,k} \Delta \mc{T}_{i\ell j k} (\ub) \eps_j \eps_k 
+ \mathcal{O}(\|\veps\|^3).
\end{equation*}
%
Squaring and summing over $i, \ell$ result in
\begin{multline*}
\norm{
\Delta \Jb(\ub+\veps)
}_\Frob^2
= 
\norm{
\Delta \Jb(\ub)
}_\Frob^2
+ 2\sum_{i,\ell,j} \Delta J_{i\ell}(\ub)\,\Delta \Hcal_{i\ell j}(\ub)\,\eps_j \\
 + \sum_{i,\ell,j,k} \Delta \Hcal_{i\ell j}(\ub)\,\Delta \Hcal_{i\ell k}(\ub)\,\eps_j\eps_k
+ 
\sum_{i,\ell,j,k} 
\Delta J_{i\ell } (\ub)
\Delta \mc{T}_{i\ell j k} (\ub)\,\eps_j\eps_k
+ \mc{O}(\|\veps\|^3).
\end{multline*}
Taking expectation with respect to the zero-mean Gaussian perturbation, using 
$\E[\veps] = 0$
and $\E[\veps\veps^T] = \eta^2 \mathbf{I}$, the linear term vanishes and the quadratic term becomes
$\eta^2 \norm{ \Delta \Hcal(\ub) }_\Frob^2$. 
The cross term 
$\sum_{i,\ell,j,k} 
\Delta J_{i\ell} (\ub)
\Delta \mc{T}_{i\ell j k} (\ub)\,\eps_j\eps_k$ becomes 
$
\eta^2 \sum_{i,\ell}
\Delta J_{i \ell}(\ub)\,
\text{tr}\LRp{\Delta\mc{T}_{i \ell}},
$
where $\text{tr}\LRp{\Delta\mc{T}_{i \ell}}:=\sum_j \Delta \mc{T}_{i \ell jj}$ denotes the partial trace of the third-derivative tensor over the last two indices.
All remaining higher-order contributions are collected in the remainder $\mathcal{R}_\eta$.
  This completes the proof. 
\end{proof}

\begin{remark}[]
In Theorem~\theoref{implicit_hessian}, the second term is non-negative and directly penalizes the Hessian mismatch. The third term is sign-indefinite and bilinear in the Jacobian mismatch $\Delta \Jb$ and the third-derivative mismatch $\Delta \mc{T}$. Because the first term directly penalizes $\Delta\Jb$, the cross term becomes smaller as the nominal Jacobian mismatch is reduced, provided that the third-derivative mismatch remains bounded. When the cross term is small
relative to the Hessian-mismatch term, and $\eta$ is small enough for the higher-order remainder to be negligible, minimizing the randomized Jacobian
loss promotes Hessian consistency on average.
In this sense, Theorem~\theoref{implicit_hessian} identifies an implicit second-order penalty without requiring construction of the Hessian tensor $\Hcal$. Each sampled perturbation requires only Jacobian evaluations at perturbed states, costing $\mc{O}(d)$ AD passes and $\mc{O}(d^2)$ memory, rather than the
$\mc{O}(d^2)$ AD passes and $\mc{O}(d^3)$ memory required for the explicit Hessian tensor. The framework is practical for high-dimensional
systems where Hessian computation is infeasible.
\end{remark}

\subsection{Total Loss} 

The NODE is trained by minimizing the composite loss
\begin{equation*}
    \mc{L} = \mc{L}_{data}
    + \alpha_{mc} \mc{L}_{mc}
    + \alpha_{mcjac}\mc{L}_{mcjac}
    + \alpha_{jac} \mc{L}_{jac}
    + \alpha_{hes} \mc{L}_{hes},
\end{equation*}
where $\alpha_{mc}$, $\alpha_{mcjac}$, $\alpha_{jac}$, and $\alpha_{hes}$ are non-negative weights controlling the relative strength of each supervision term. Table~\tabref{lossmethod} and Figure~\figref{method-schematic} summarizes the five training strategies considered in this work. For the MC-based variants, the perturbation scale $\eta$ controls the strength of the implicit derivative supervision and is tuned separately for each benchmark problem.

Numerical experiments with the five methods are presented in Section \secref{numerical-results}.
 
\begin{table}[ht]
\centering
\caption{
   Training strategies with increasing derivative supervision. All methods include the data loss $\mc{L}_{\text{data}}$ with unit weight. A ``$+$'' indicates that the corresponding supervision term is included.
    }
    \vspace{0.05in}
\begin{tabular}{l | c c c c c}
\toprule
Method & $\alpha_{mc}$ &  $\alpha_{mcjac}$ &  $\alpha_{jac}$ & $\alpha_{hes}$ & Derivative supervision\\
\midrule
\texttt{mc}    & $+$ & 0 & 0 & 0 & Implicit 1st-order\\
\texttt{mcjac} & $+$ & $+$ & 0 & 0 & Explicit 1st \& Implicit 2nd-order \\
\texttt{naive} & 0 & 0 & 0 & 0  & None \\
\texttt{jac}   & 0 & 0 & $+$ & 0 & Explicit 1st-order \\
\texttt{hes}   & 0 & 0 & $+$ & $+$ & Explicit 1st \& Explicit 2nd-order\\
\bottomrule
\end{tabular}
\tablab{lossmethod}
\end{table}

\begin{figure}[h!t!b!]
    \centering
        \includegraphics[width=\linewidth]{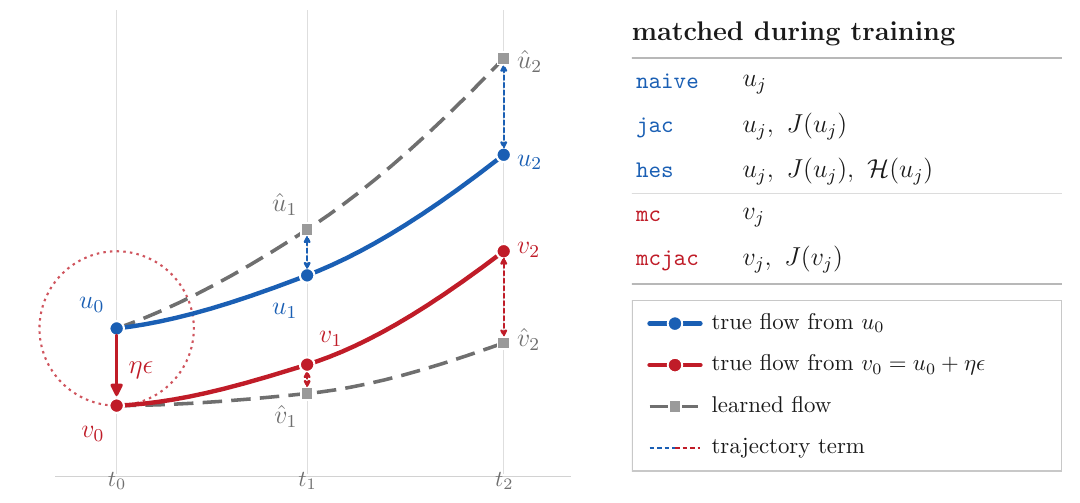}
        \caption{Schematic of the five training strategies. The blue and red branches denote the true trajectories initialized at $u_0$ and at the perturbed state $v_0=u_0+\eta\epsilon$, respectively, while gray dashed curves denote the corresponding learned trajectories. The \texttt{naive}, \texttt{jac}, and \texttt{hes} methods supervise the unperturbed branch, whereas \texttt{mc} and \texttt{mcjac} apply supervision along the perturbed branch. The quantities matched by each method are summarized on the right.}
        \figlab{method-schematic}
\end{figure}

\section{Numerical Experiments}
\seclab{numerical-results}

We evaluate the five methods of Table~\tabref{lossmethod} on the Lorenz~63 system and the coupled Lorenz~96 system. These methods span a hierarchy of derivative supervision. The \texttt{naive} method performs only zeroth-order trajectory matching. The \texttt{mc} method introduces implicit first-order information through the model-constrained approach. The \texttt{jac} method adds explicit Jacobian supervision, while \texttt{hes} further incorporates explicit Hessian supervision. Finally, \texttt{mcjac} combines Jacobian matching with the model-constrained approach, which by Theorem~\theoref{implicit_hessian} induces implicit second-order supervision.
  
The methods naturally form two paired comparisons for assessing the role of second-order supervision.
The first, \texttt{mc} versus \texttt{mcjac}, examines the effect of implicit second-order supervision in the model-constrained setting. The second, \texttt{jac} versus \texttt{hes}, examines the effect of explicit second-order supervision.
 The \texttt{naive} method serves as a baseline, illustrating the limitations of trajectory-only supervision for chaotic systems.

\subsection{Lorenz~63 Model}
\seclab{lorenz63}

We first evaluate the learned models from off-attractor initial conditions. We then repeat the evaluation using initial conditions sampled from the validation attractor across multiple training seeds.

The training data were generated by integrating the true Lorenz~63 system over $t\in[0,500]$ with time step $\Delta t = 0.01$. 
The system was initialized from 32 initial conditions sampled uniformly from a unit cube centered at $(x,y,z) = (-15,-15,5)$.
Among these trajectories, 16 were used for training and the remaining 16 for validation. 
An initial transient of $1{,}000$ time steps was discarded so that all trajectories lie on the attractor.
To approximate the vector field $\fbhat$ of the Lorenz~63 system in \eqnref{node}, we employed a fully connected neural network $\mathbb{R}^3\to\mathbb{R}^3$ with three hidden layers, each containing $512$ neurons, and \texttt{tanh} activations to ensure smooth derivatives. The neural network was optimized using the Adam optimizer with a constant learning rate of $10^{-3}$ for $10^4$ epochs, batch size $n_b=512$, and rollout length $m=1$.

A single-step segment provides the least amount of temporal information and therefore forces the model to rely almost entirely on the local geometric supervision provided by the Jacobian and Hessian terms.
Under ergodicity, a sufficiently large collection of post-transient single-step
pairs $(\ub_0,\ub_1)$ samples the invariant measure on the attractor~\citep{eckmann1985ergodic}.
The associated Jacobians $(\Jb_0,\Jb_1)$ and Hessians $(\Hcal_0,\Hcal_1)$, being state-dependent quantities, then characterize the local tangent and curvature structure over statistically representative regions of the attractor. Thus, in the ergodic limit, local matching can constrain global dynamical invariants. The $m=1$ regime therefore tests whether local geometric constraints alone can reproduce invariant measures and Lyapunov exponents.

\begin{figure}[h!t!b!]
    \centering
    \begin{subfigure}[b]{0.32\linewidth}
        \centering
        \includegraphics[width=\linewidth]{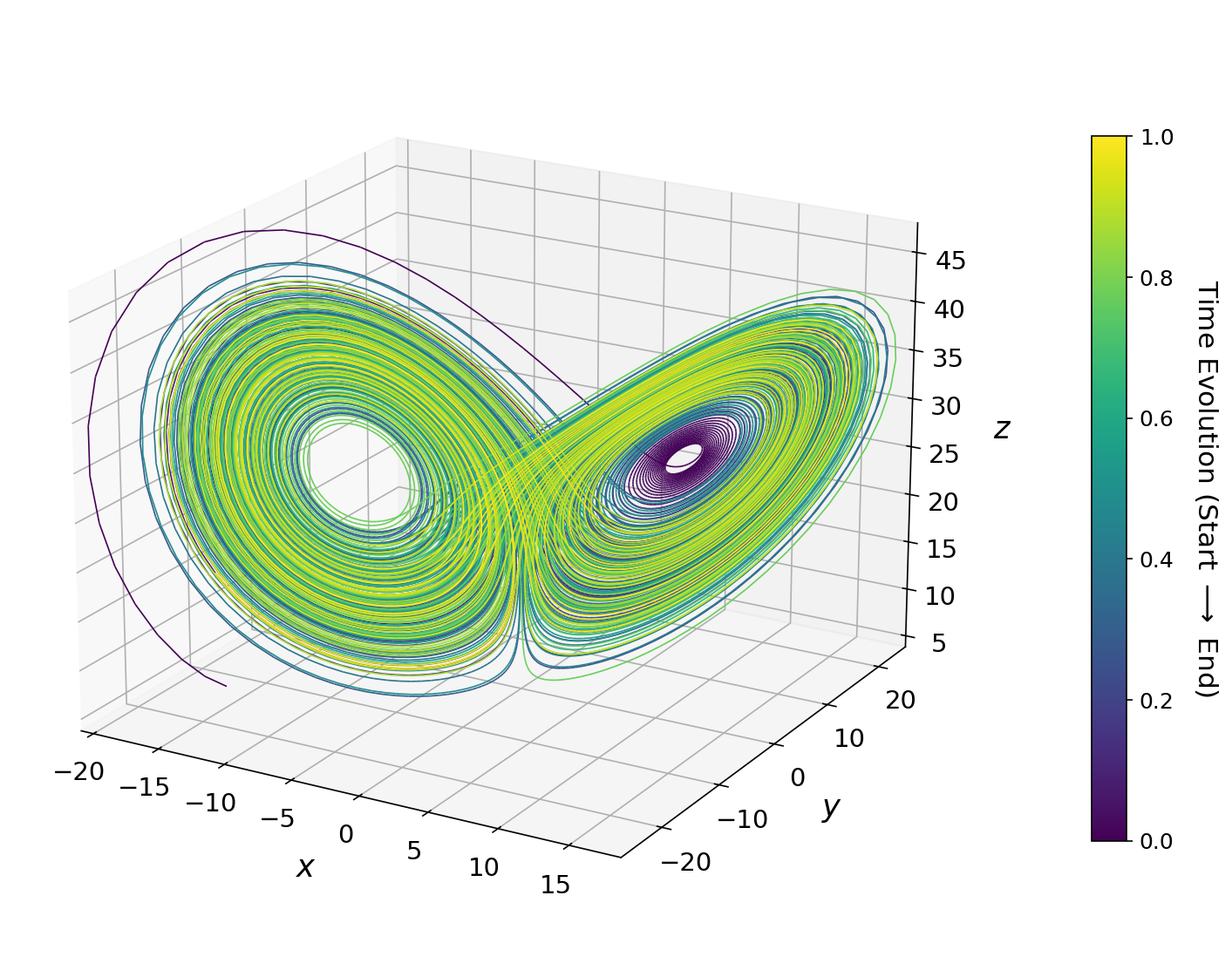}
        \caption{Full trajectory (true)}
        \figlab{full-true}
    \end{subfigure}
    \hfill
    \begin{subfigure}[b]{0.32\linewidth}
        \centering
        \includegraphics[width=\linewidth]{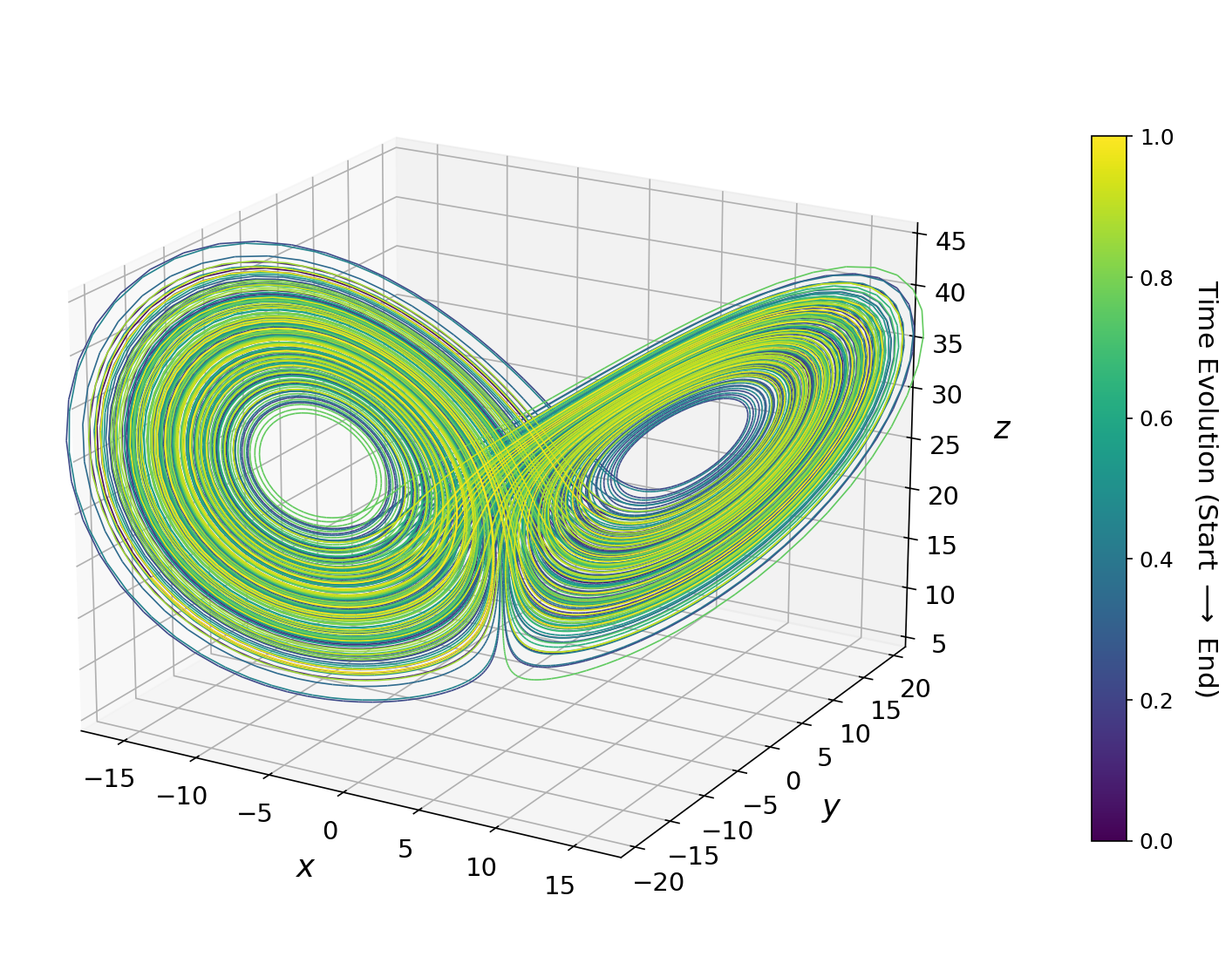}
        \caption{Truncated trajectory (true)}
        \figlab{trunc-true}
    \end{subfigure}
    \hfill
    \begin{subfigure}[b]{0.32\linewidth}
        \centering
        \includegraphics[width=\linewidth]{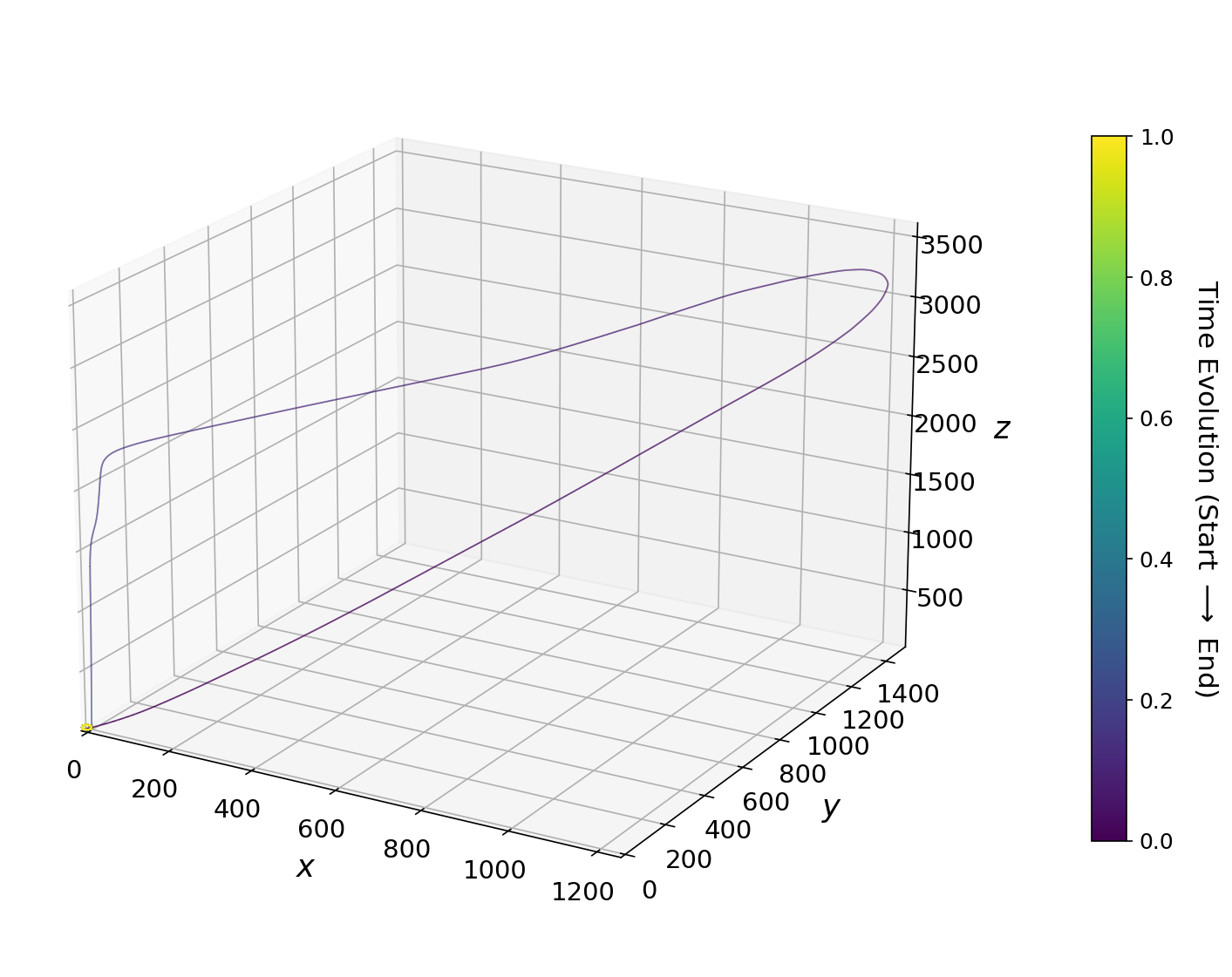}
        \caption{\texttt{naive}}
        \figlab{fail-naive}
    \end{subfigure}
    \caption{Phase space trajectories of the Lorenz~63 system under the minimal temporal setting ($m=1$).
    (a) Full true trajectory ($t\in[0,320]$), 
    (b) truncated true trajectory ($t\in[20,320]$), and 
    (c) learned trajectory of \texttt{naive} model.
    }
    \figlab{phase-trajectories}
\end{figure}

Figure \figref{phase-trajectories} contrasts the true butterfly attractor with the diverging \texttt{naive} model. Without derivative supervision, the learned dynamics escape to $\snor{x}, \snor{y} \sim 10^{3}$ and fail to remain on the attractor. 
Figure \figref{phase-trajectories2} shows that adding any form of derivative supervision---\texttt{mc}, \texttt{mcjac}, \texttt{jac}, or \texttt{hes}---recovers a bounded butterfly attractor. The differences among these methods are examined below.

\begin{figure}[h!t!b!]
    \centering
    \begin{subfigure}[b]{0.23\linewidth}
        \centering        
        \includegraphics[trim=0cm 1cm 4.2cm 0cm,clip=true,width=\linewidth]{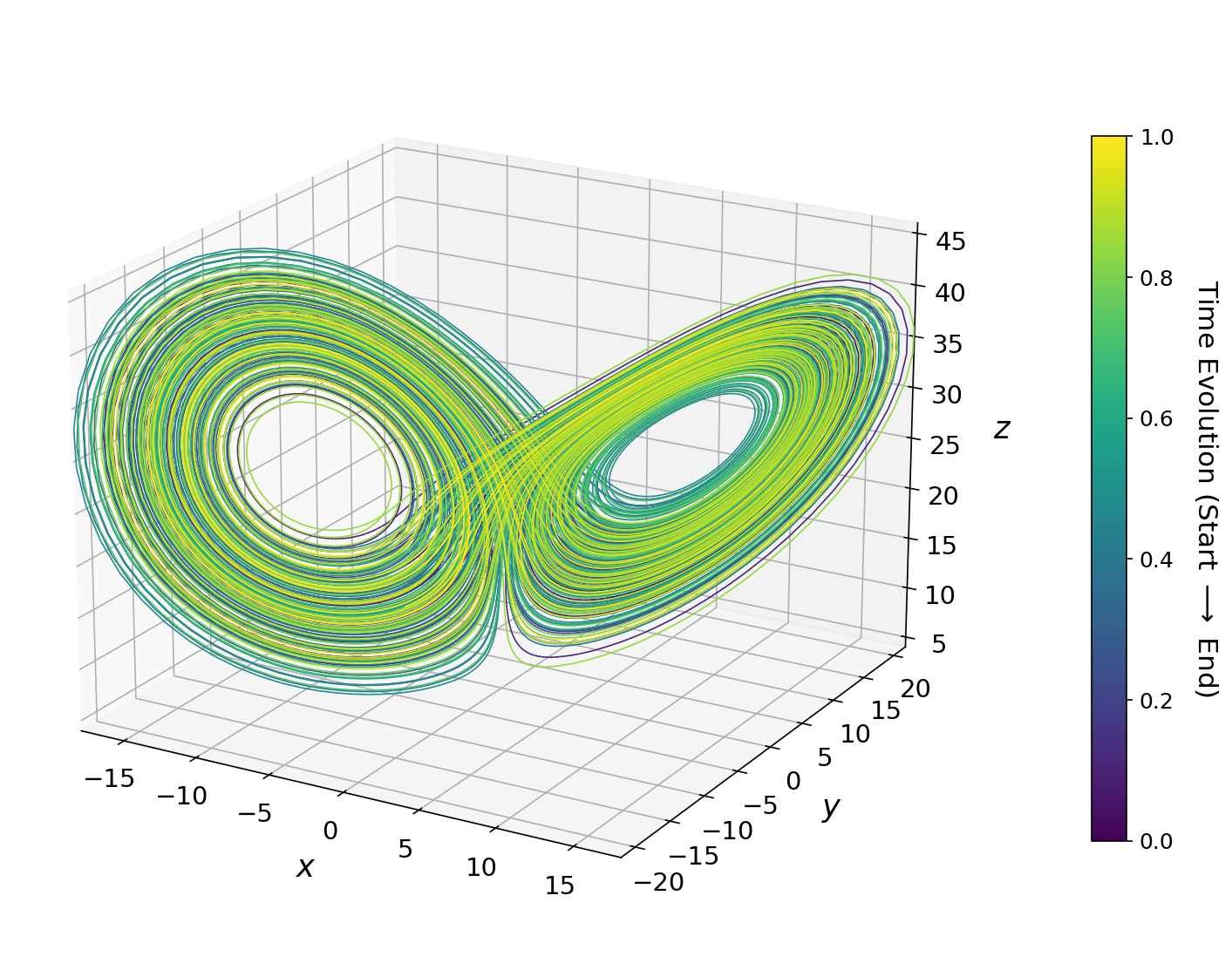}
        \caption{\texttt{mc}}
        \figlab{mc-trunc}
    \end{subfigure}
    \begin{subfigure}[b]{0.23\linewidth}
        \centering
        \includegraphics[trim=0cm 1cm 4.2cm 0cm,clip=true,width=\linewidth]{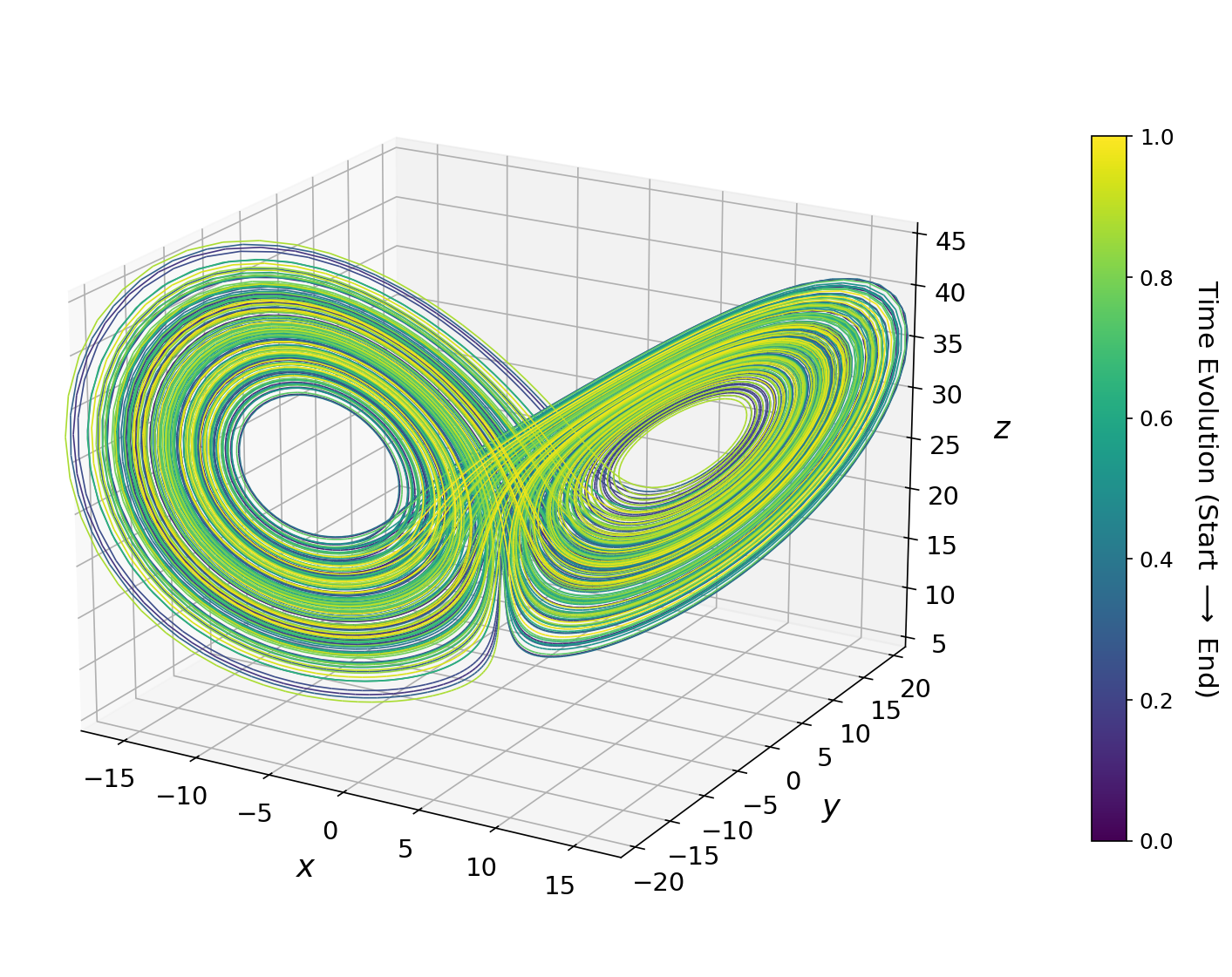}
        \caption{\texttt{mcjac}}
        \figlab{mcjac-trunc}
    \end{subfigure}
    \hfill
    \begin{subfigure}[b]{0.23\linewidth}
        \centering
        \includegraphics[trim=0cm 1cm 4.2cm 0cm,clip=true,width=\linewidth]{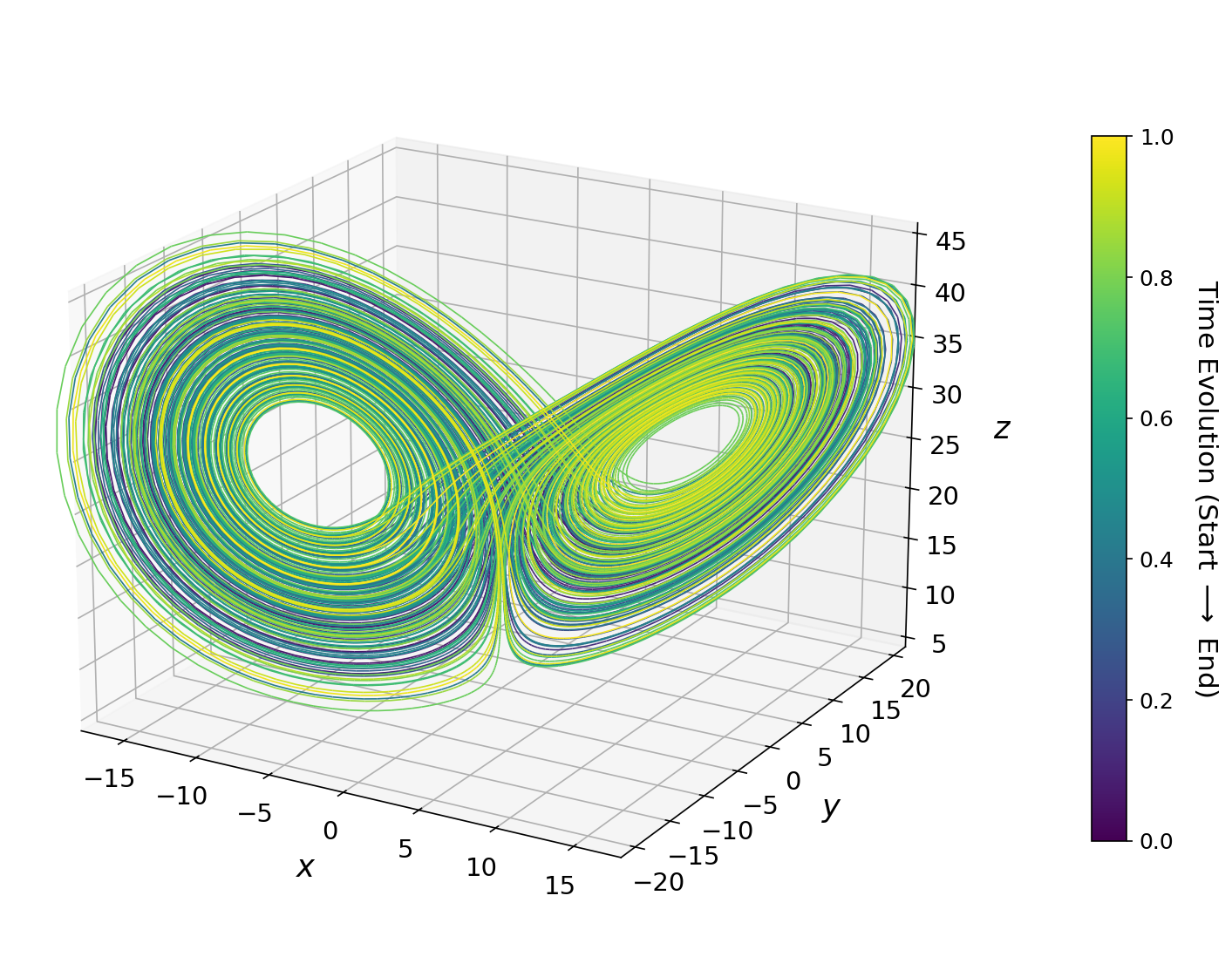}
        \caption{\texttt{jac}}
        \figlab{jac-trunc}
    \end{subfigure}
    \hfill
    \begin{subfigure}[b]{0.23\linewidth}
        \centering
        \includegraphics[trim=0cm 1cm 4.2cm 0cm,clip=true,width=\linewidth]{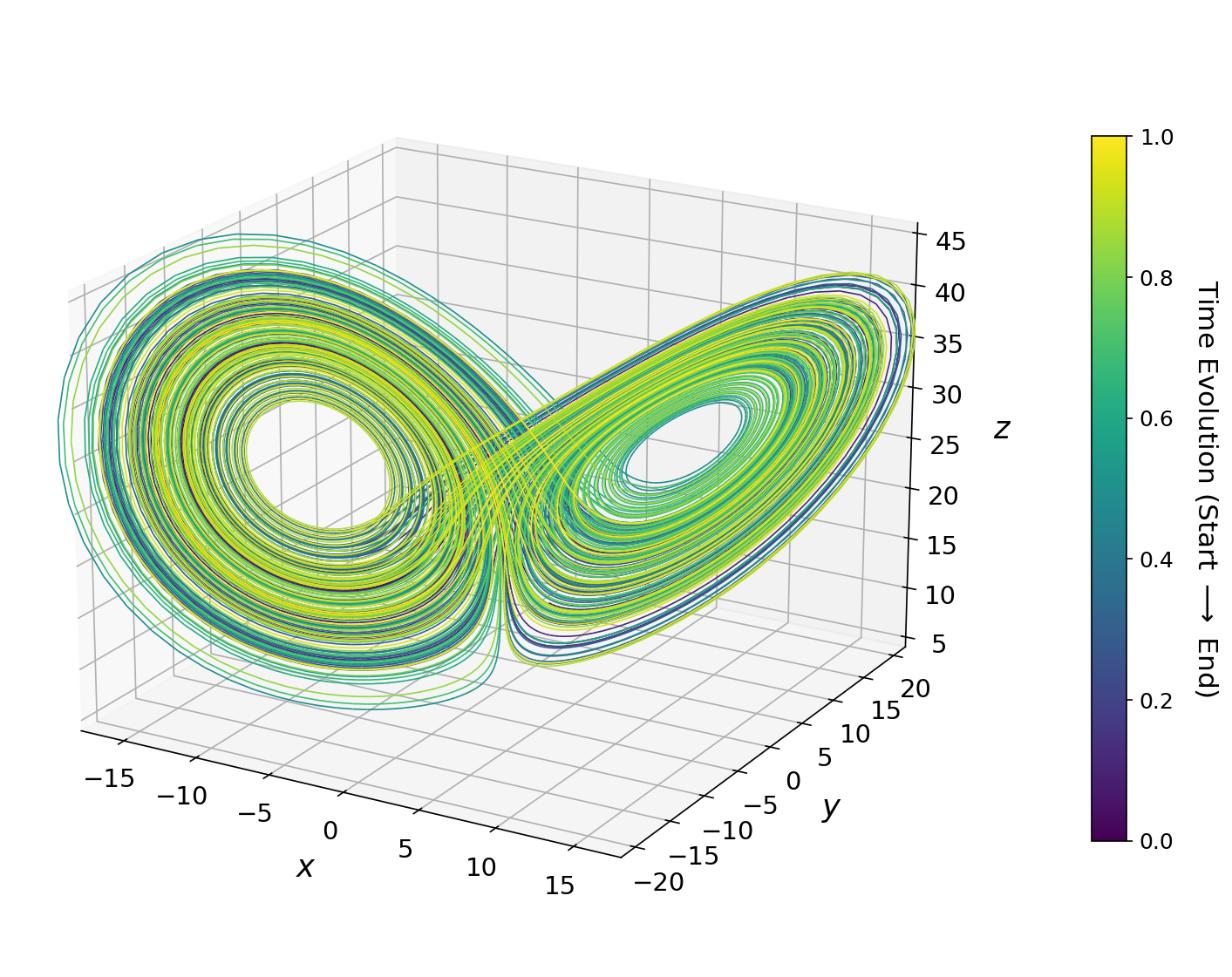}
        \caption{\texttt{hes}}
        \figlab{hes-trunc}
    \end{subfigure}

    \caption{Phase space trajectories of the learned models of (a) \texttt{mc}, 
    (b) \texttt{mcjac},
    (c) \texttt{jac}, and 
    (d) \texttt{hes} under the minimal temporal setting ($m=1$).
    }
    \figlab{phase-trajectories2}
\end{figure}

For the error metrics in 
Tables~\tabref{wd-m1},~\tabref{wd-m2}, and~\tabref{le-wc},
we use a reference Lyapunov spectrum computed directly from the simulated Lorenz~63 trajectories, averaged over the same $1{,}000$ initial conditions used for evaluation: $(\lambda_1^{\rm ref}, \lambda_2^{\rm ref}, \lambda_3^{\rm ref}) \approx (0.9041, -0.0023, -14.5685)$. 
This numerically estimated reference is close to the widely cited literature values ($0.9056, 0, -14.5723$) reported in~\cite{sprott2003chaos}, with the small discrepancy potentially arising from finite-integration time, finite-ensemble size, and temporal discretization.

We generated $1{,}000$ initial conditions sampled uniformly in a unit cube around $(x,y,z) = (-15,-15,5)$. Each trajectory was integrated for $32{,}000$ time steps using the learned NODEs, with the first $2{,}000$ steps discarded as transient. The Lyapunov exponents were computed over the remaining $30{,}000$ steps using the standard QR-based Benettin algorithm \citep{benettin1980lyapunov}. The Wasserstein-1 distance between empirical invariant measures was computed from the same long-time trajectories. The Wasserstein-1 distance quantifies how faithfully the learned dynamics reproduce the global geometry of the strange attractor.

\begin{table}[h!t!b!]
\centering
\caption{
    Statistical comparison between the ground-truth and learned Lorenz~63
    systems under the minimal temporal setting ($m=1$).
    $W^1$ denotes the Wasserstein-$1$ distance.
    Here, $\hat{\lambda}_j$ denotes the $j$-th Lyapunov exponent estimated
    from the learned NODE, while $\lambda_j^{\rm ref}$ denotes the
    corresponding reference exponent computed from the simulated Lorenz~63
    trajectories using the same evaluation interval $t\in[20,320]$.
    For \texttt{naive}, N/A indicates that all evaluated rollouts leave the
    target attractor, making the long-time statistical comparison not
    meaningful. Cell shading indicates relative performance within each
    column on a logarithmic scale, with darker shading corresponding to
    smaller values.
}
\vspace{0.05in}

\begin{tabular}{l | c c c c}
\hline
Method
& $W^1$
& $|\hat{\lambda}_1-\lambda_1^{\rm ref}|$
& $|\hat{\lambda}_2-\lambda_2^{\rm ref}|$
& $|\hat{\lambda}_3-\lambda_3^{\rm ref}|$ \\
\hline

\texttt{mc}
& \cellcolor{heat!43} 1.0534
& \cellcolor{heat!15} 4.5242E-02
& \cellcolor{heat!24} 5.9054E-03
& \cellcolor{heat!70} {\bf 5.5523E-03} \\

\texttt{mcjac}
& \cellcolor{heat!70} {\bf 0.9268}
& \cellcolor{heat!70} {\bf 9.6428E-04}
& \cellcolor{heat!42} 7.3800E-04
& \cellcolor{heat!63} 8.6746E-03 \\

\texttt{naive}
& \cellcolor{gray!12} N/A
& \cellcolor{gray!12} N/A
& \cellcolor{gray!12} N/A
& \cellcolor{gray!12} N/A \\

\texttt{jac}
& \cellcolor{heat!15} 1.1977
& \cellcolor{heat!38} 9.0963E-03
& \cellcolor{heat!15} 1.7947E-02
& \cellcolor{heat!15} 1.5970E-01 \\

\texttt{hes}
& \cellcolor{heat!54} 1.0003
& \cellcolor{heat!55} 2.7298E-03
& \cellcolor{heat!70} {\bf 2.7080E-05}
& \cellcolor{heat!52} 1.6469E-02 \\

\hline
\end{tabular}

\tablab{wd-m1}
\end{table}


Table~\tabref{wd-m1} summarizes the results at $m=1$. Because \texttt{naive} fails to reconstruct the attractor, its long-time statistical metrics are not comparable with those of the other methods and are reported as N/A; for reference, its rollouts give $W^1 = 379.16$ and an order-unity error in the leading Lyapunov exponent. Although nearby trajectories separate exponentially with $\lambda_1^{\rm ref}=0.9041$, the single-step interval $\Delta t=0.01$ yields an expected separation of only $e^{\lambda_1^{\rm ref}\Delta t}\approx 1.009$ per training step, which provides little constraint on the learned vector field. Trajectory-only supervision at $m=1$ therefore leaves the chaotic instability of the learned dynamics essentially unconstrained.

In contrast, all derivative-informed methods produce bounded attractors and yield comparable Wasserstein distances ($W^1 \in [0.93, 1.20]$). The paired comparisons show improved invariant-measure accuracy under second-order supervision. In the Jacobian-based pair, explicit Hessian matching improves upon Jacobian matching, reducing $W^1$ from $1.1977$ (\texttt{jac}) to $1.0003$ (\texttt{hes}). Similarly, in the MC pair, randomized Jacobian matching improves upon randomized trajectory matching, reducing $W^1$ from $1.0534$ (\texttt{mc}) to $0.9268$ (\texttt{mcjac}). In both pairs, the additional supervision yields a more accurate joint $(x,y,z)$ invariant measure. 
%
For the Lyapunov spectrum, \texttt{mcjac} attains the smallest error in \(\lambda_1\), while \texttt{hes} is most accurate for the neutral exponent $\lambda_2$. The errors in $\lambda_3$, however, are not independent of the other exponents. For the true Lorenz~63 system, the exponent associated with the flow direction is zero, $\lambda_2=0$, and the divergence is the state-independent $\nabla\cdot\fb=-(\sigma+1+\beta)=-13.6667$. Because the sum of the Lyapunov exponents equals the time-averaged divergence, the true spectrum satisfies $\lambda_3=-(\sigma+1+\beta)-\lambda_1$. Although the Benettin algorithm computes all three exponents separately, its full-spectrum calculation tracks the contraction of an infinitesimal volume, so their sum obeys this divergence constraint. The clustering of $\lambda_3$ reflects the structure of the Lorenz~63 spectrum, rather than unusually accurate estimation of the most contracting direction~\cite{benettin1980lyapunov,eckmann1985ergodic}.

The marginal distributions show the same paired improvement. Figure~\figref{x-marginals} compares the empirical $x$-marginal of each learned model against the true Lorenz~63 distribution. In the MC pair, adding implicit second-order supervision improves agreement with the reference distribution, reducing the Wasserstein distance of the $x$-marginal from $1.10$ (\texttt{mc}) to $0.37$ (\texttt{mcjac}). The \texttt{mc} histogram exhibits a distributional bias that is mitigated by \texttt{mcjac}. In the Jacobian-based pair, the improvement is smaller but remains consistent, with the Wasserstein distance decreasing from $0.23$ (\texttt{jac}) to $0.14$ (\texttt{hes}).

\begin{figure}[h!t!b!]
    \centering
    \begin{subfigure}[b]{0.23\linewidth}
        \centering
        \includegraphics[width=\linewidth]{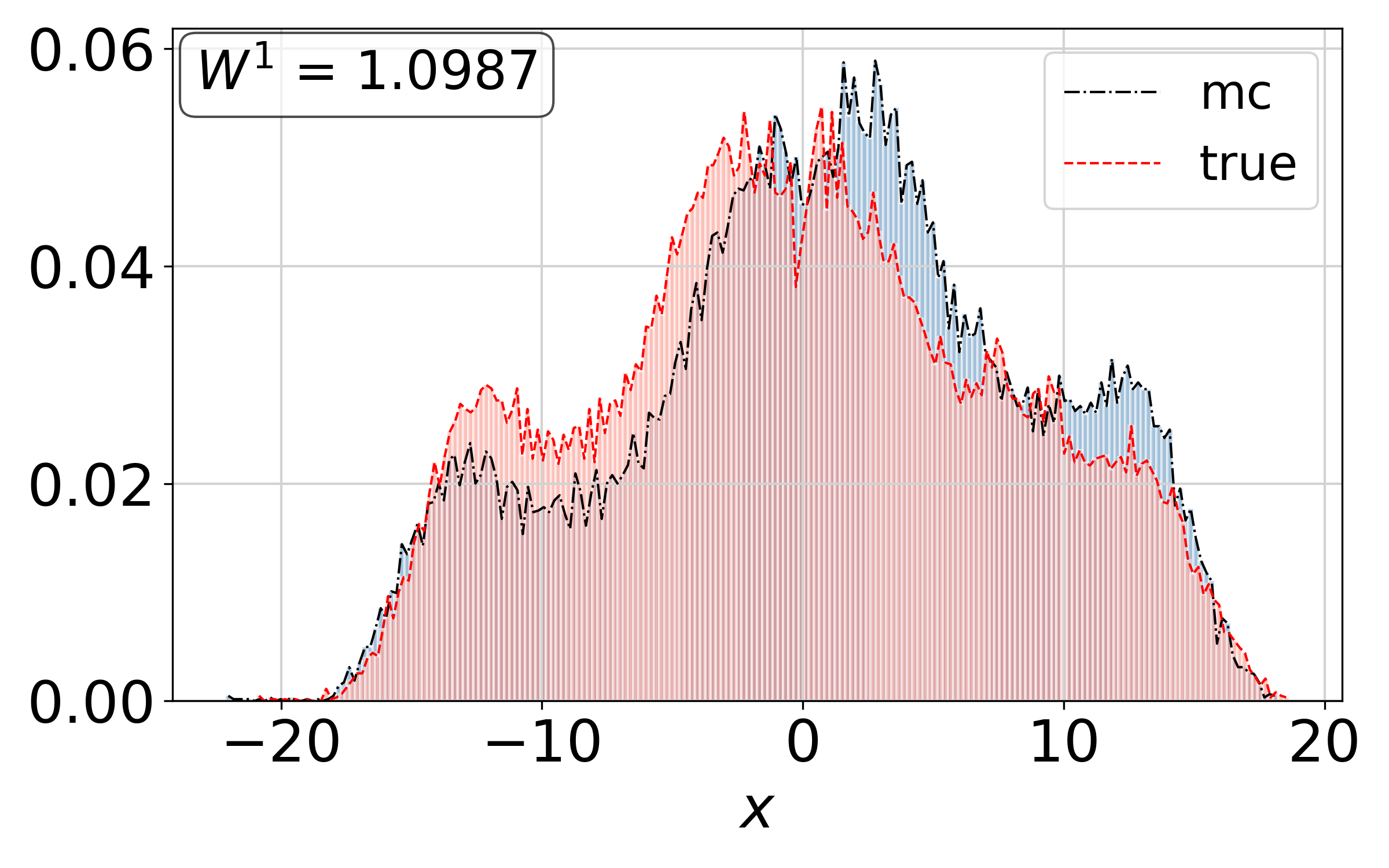}
        \caption{\texttt{mc}}
        \figlab{x-density-mc}
    \end{subfigure}
    \hfill
    \begin{subfigure}[b]{0.23\linewidth}
        \centering
        \includegraphics[width=\linewidth]{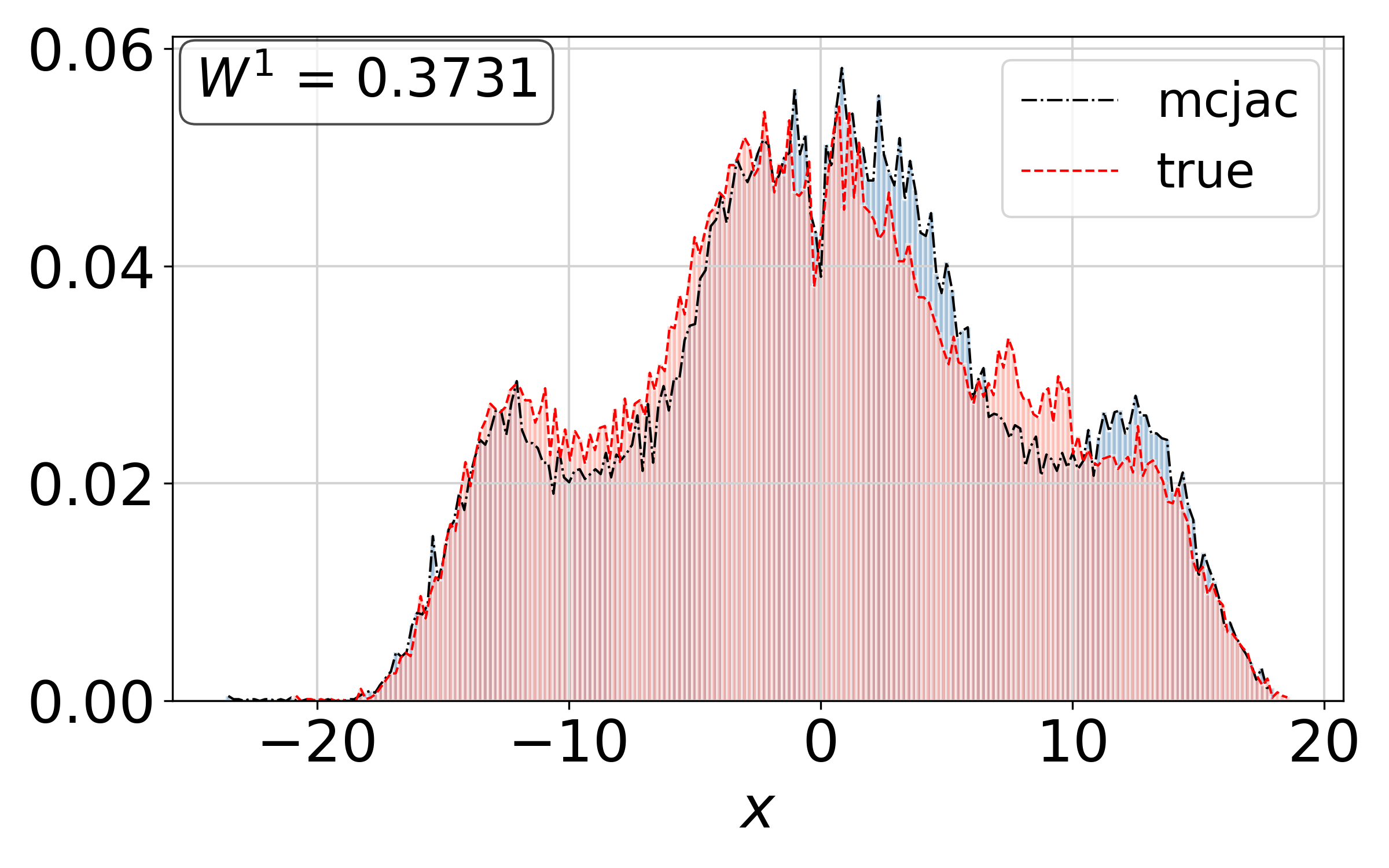}
        \caption{\texttt{mcjac}}
        \figlab{x-density-mcjac}
    \end{subfigure}
    \hfill
    \begin{subfigure}[b]{0.23\linewidth}
        \centering
        \includegraphics[width=\linewidth]{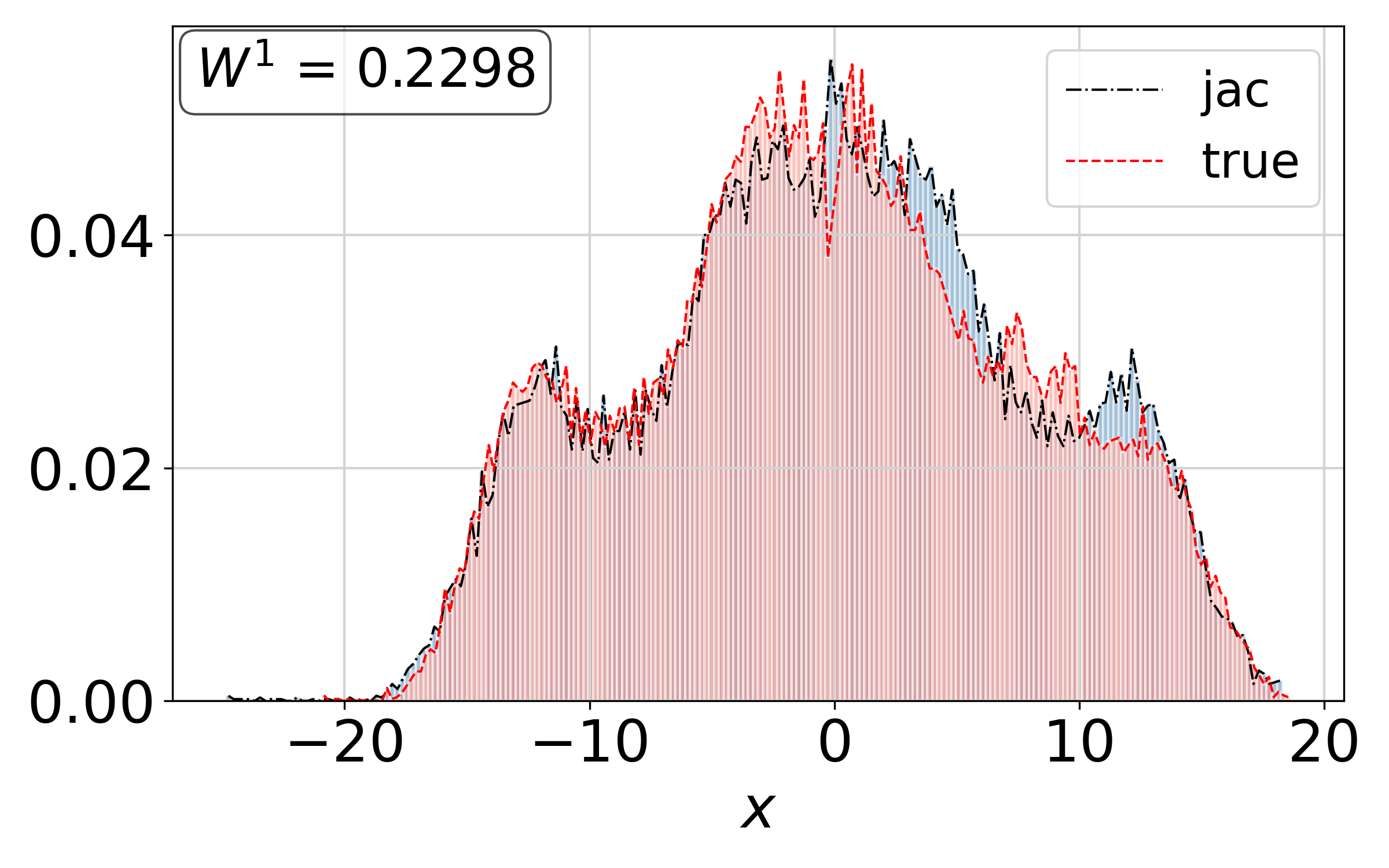}
        \caption{\texttt{jac}}
        \figlab{x-density-jac}
    \end{subfigure}
    \hfill
    \begin{subfigure}[b]{0.23\linewidth}
        \centering
        \includegraphics[width=\linewidth]{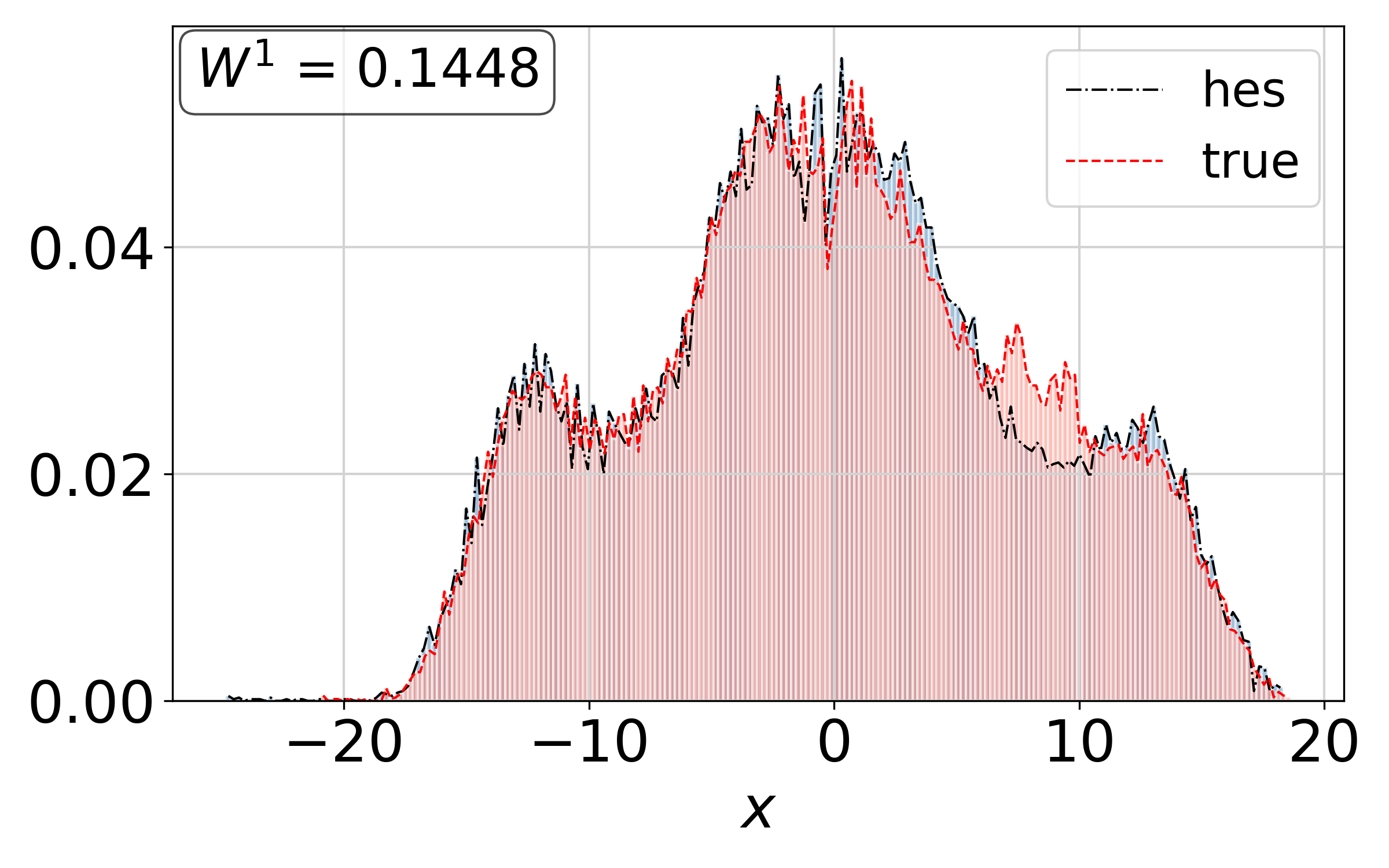}
        \caption{\texttt{hes}}
        \figlab{x-density-hes}
    \end{subfigure}
    \caption{Empirical marginal distribution of the $x$-component for each learned model compared to the true Lorenz~63 distribution ($m=1$).
    The Wasserstein-1 distance ($W^1$) between the learned and true distributions is shown in each panel.
    }
    \figlab{x-marginals}
\end{figure}


\begin{figure}[h!t!b!]
    \centering
    \begin{subfigure}[b]{0.48\linewidth}
        \centering
        \includegraphics[trim=1cm 1.0cm 4.0cm 1.0cm,clip=true,width=\linewidth]{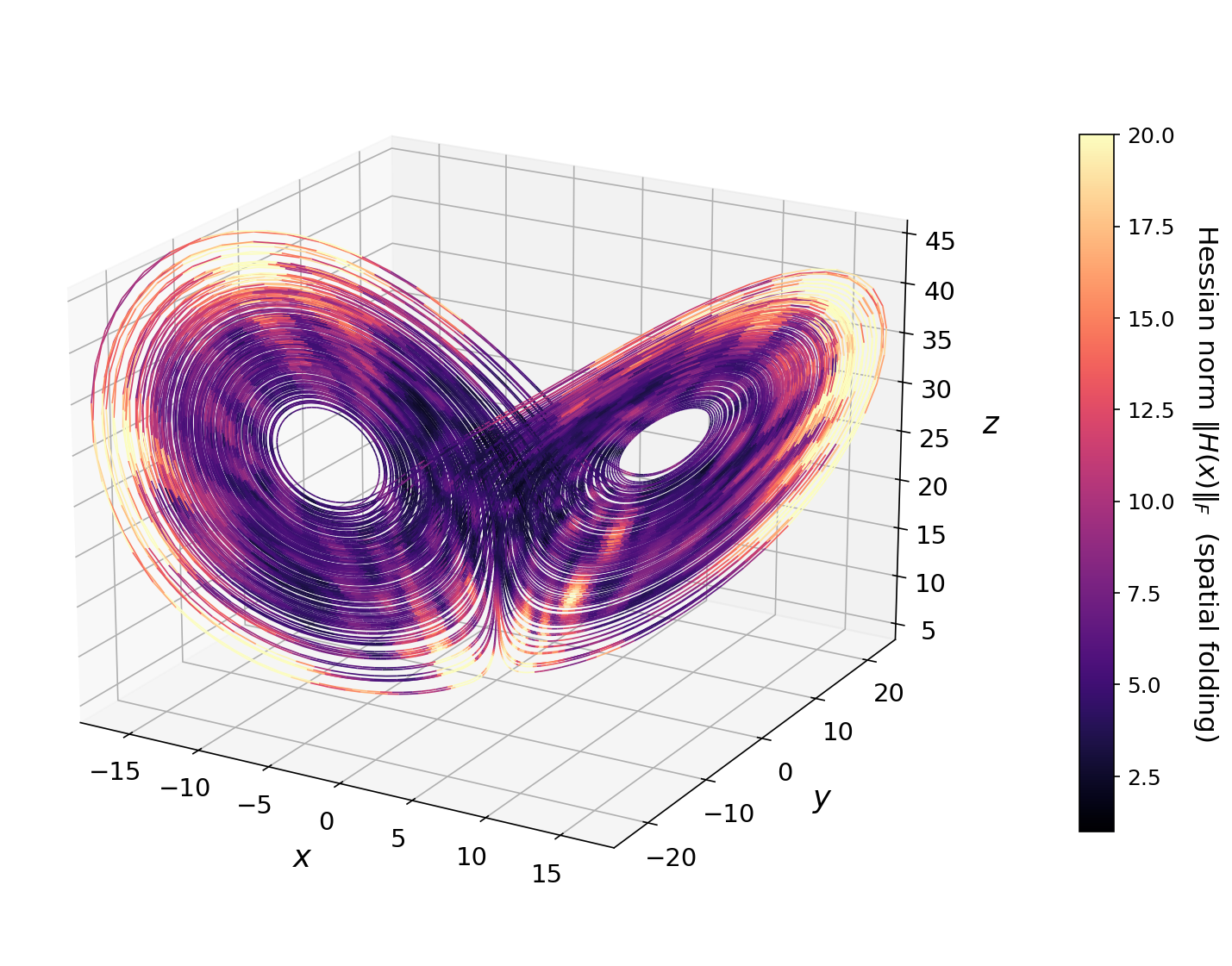}
        \caption{\texttt{mc}}
        \figlab{hnorm-mc}
    \end{subfigure}
    \begin{subfigure}[b]{0.48\linewidth}
        \centering
        \includegraphics[trim=1cm 1.0cm 4.0cm 1.0cm,clip=true,width=\linewidth]{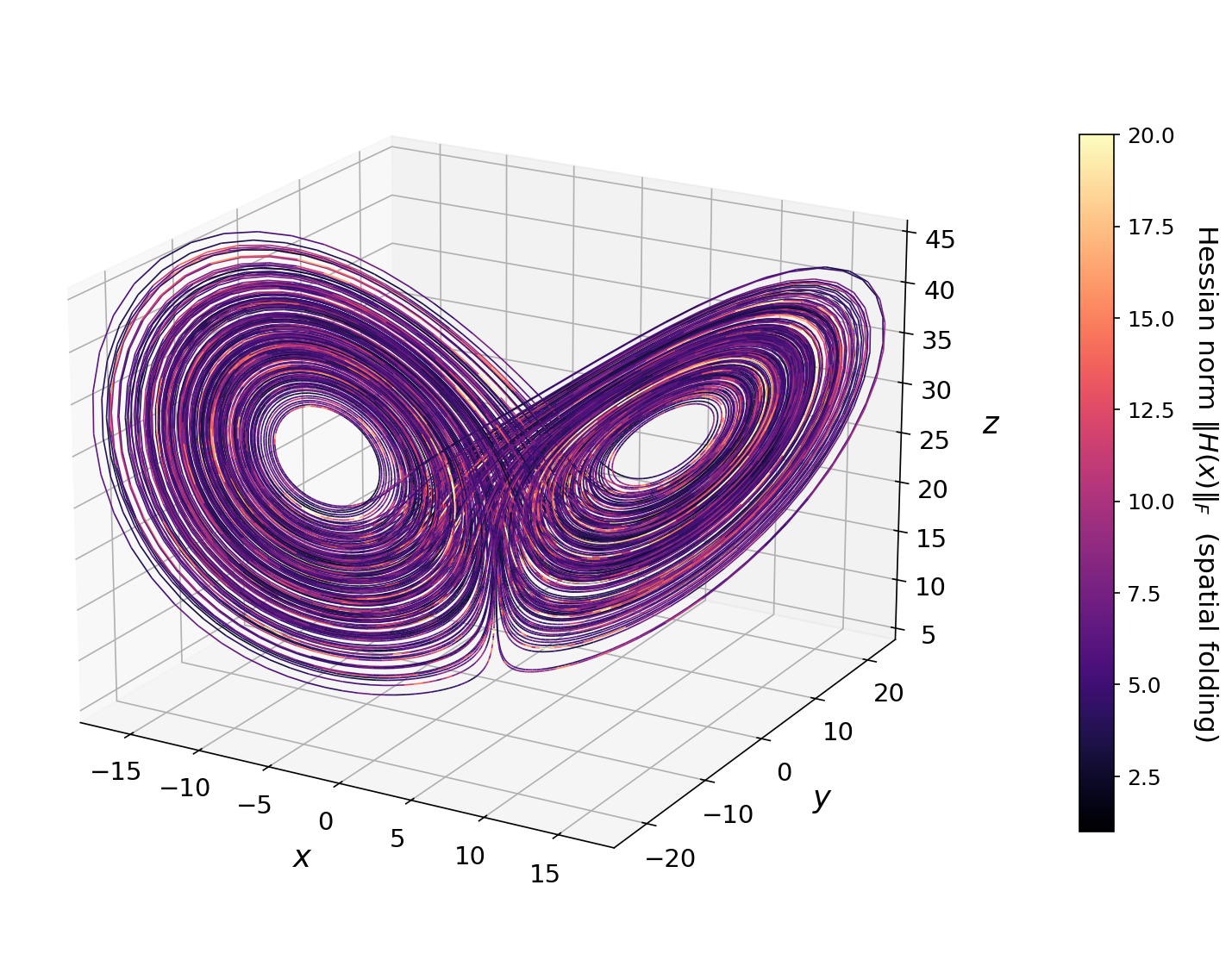}
        \caption{\texttt{mcjac}}
        \figlab{hnorm-mcjac}
    \end{subfigure} \\
    \begin{subfigure}[b]{0.48\linewidth}
        \centering
        \includegraphics[trim=1cm 1.cm 4.cm 1.0cm,clip=true,width=\linewidth]{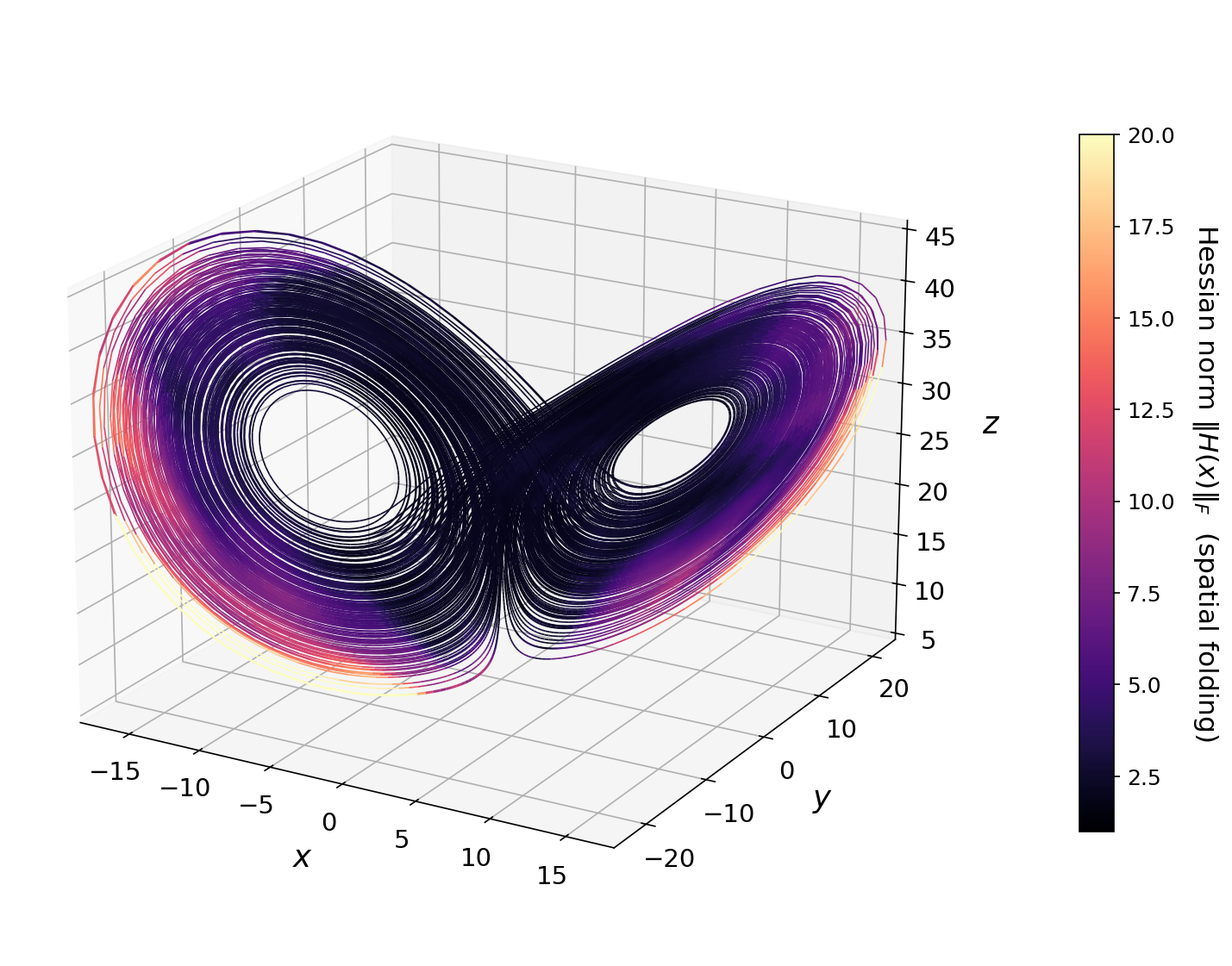}
        \caption{\texttt{jac}}
        \figlab{hnorm-jac}
    \end{subfigure}
    \begin{subfigure}[b]{0.48\linewidth}
        \centering
        \includegraphics[trim=1cm 1.cm 4.cm 1.0cm,clip=true,width=\linewidth]{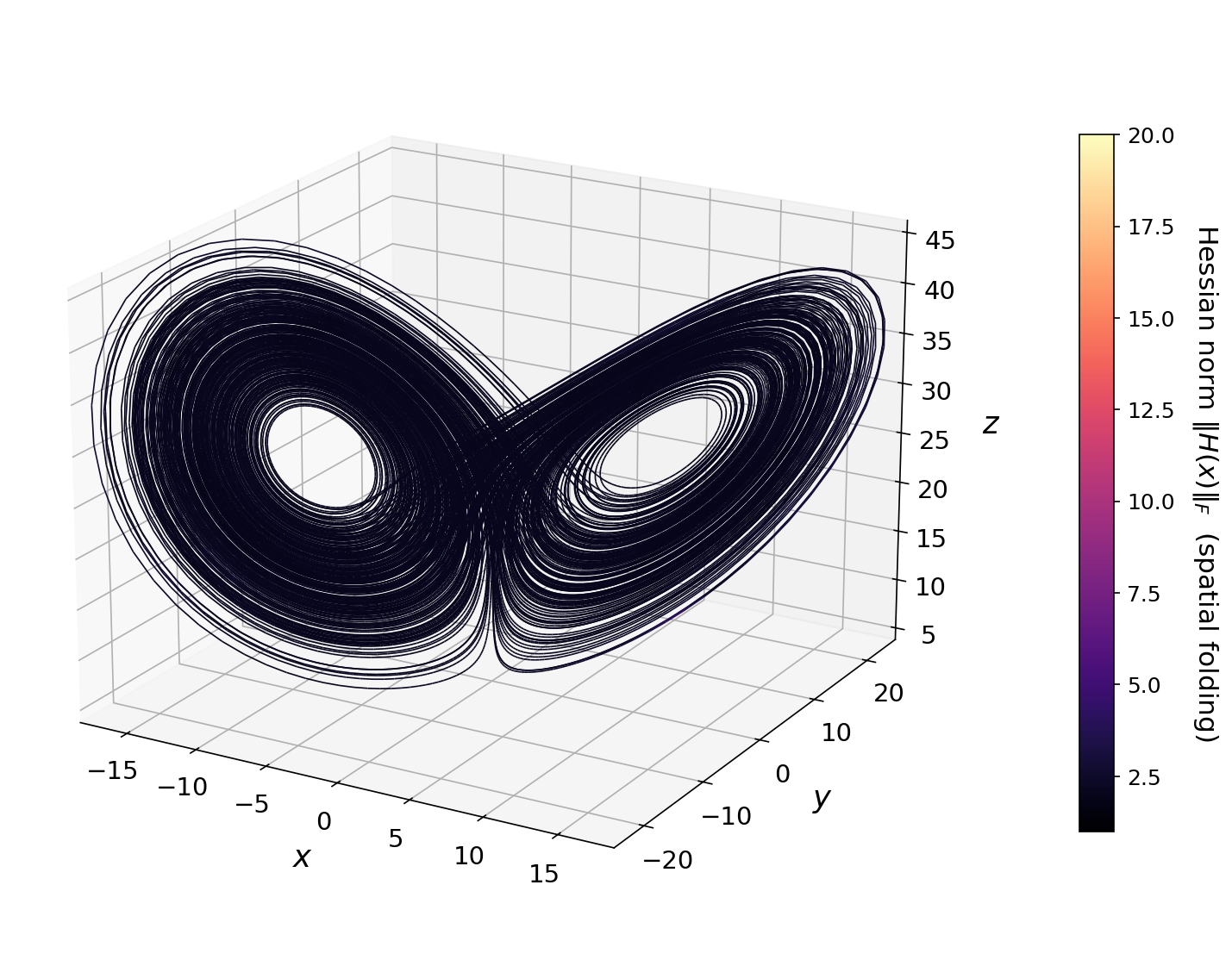}
        \caption{\texttt{hes}}
        \figlab{hnorm-hes}
    \end{subfigure}
    \caption{Hessian Frobenius norm $\| \Hcalhat \|_\Frob$ of each learned model evaluated along Lorenz~63 trajectories at $m=1$: (a)~\texttt{mc}, (b)~\texttt{mcjac}, (c)~\texttt{jac}, (d)~\texttt{hes}. The colormap ranges from $1$ (black) to $20$ (yellow); values above $20$ are clipped. The true Lorenz~63 vector field is bilinear, so its Hessian is spatially constant ($\|\Hcal\|_\Frob=2$, nearly black) at every point of phase space. Mean (std) across the trajectory is $6.76\,(3.97)$,  $2.06\,(0.17)$, $3.43\,(2.35)$, $2.01\,(0.08)$, respectively (Table~\tabref{hnorm}).
    }
    \figlab{hnorm}
\end{figure}

\begin{table}[h!t!b!]
\centering
\caption{Statistics of the Hessian Frobenius norm $\|\Hcalhat \|_\Frob$ along trajectories for each training method.}
\begin{tabular}{l|cccc}
\hline
Method & mean & std & min & max \\
\hline
\texttt{mc}    & 6.76  & 3.97   & 1.56 & 44.93  \\
\texttt{mcjac} & 2.06  & 0.17   & 1.58 & 6.33  \\
\texttt{naive} & 46.81 & 67.69  & 7.48E-04 & 231.7 \\
\texttt{jac}   & 3.43  & 2.35   & 1.51 & 38.37 \\
\texttt{hes}   & {\bf 2.01}  & 0.08   & 1.54 & 4.84  \\
\hline
\end{tabular}
\tablab{hnorm}
\end{table}

The learned models also differ in the second-order structure of their vector fields. Since the Lorenz~63 system is bilinear, the Hessian norm remains constant throughout phase space, with $\norm{\Hcal}_\Frob=2$. Figure~\figref{hnorm} visualizes $\|\Hcalhat\|_\Frob$ along trajectories of the learned models under the $m=1$ setting, with summary statistics reported in Table~\tabref{hnorm}. 
Second-order supervision yields Hessian norms that are close to the true constant: \texttt{hes} (mean $2.01$, std $0.08$) and \texttt{mcjac} (mean $2.06$, std $0.17$). Both remain within $3\%$ of the true value and exhibit only weak spatial variation. The first-order methods deviate far more: \texttt{mc} (mean $6.76$, std $3.97$) and \texttt{jac} (mean $3.43$, std $2.35$) reach peaks exceeding the true norm by more than an order of magnitude. The trajectory-only baseline \texttt{naive} exhibits even larger deviations (mean $46.8$, std $67.7$), consistent with its complete failure to reproduce the attractor.

\begin{table}[h!t!b!]
\centering
\caption{
Cross-loss evaluation of the five supervision strategies at $m=1$. Rows denote the training objective and columns denote the evaluation objective.
Entries are the mean $\pm$ standard error over 16 validation trajectories, evaluated on $100{,}000$ held-out on-attractor windows.
}
\vspace{0.05in}

\resizebox{\linewidth}{!}{%
\begin{tabular}{l | c c c c c}
\hline
Method
& $\mathcal{L}_{\mathrm{naive}}$
& $\mathcal{L}_{\mathrm{jac}}$
& $\mathcal{L}_{\mathrm{hes}}$
& $\mathcal{L}_{\mathrm{mc}}$
& $\mathcal{L}_{\mathrm{mcjac}}$ \\
\hline

\texttt{mc}
& \cellcolor{heat!60} 2.06E-5 \(\pm\) 1.40E-6
& \cellcolor{heat!48} 4.09E-1 \(\pm\) 6.20E-3
& \cellcolor{heat!43} 1.43E+0 \(\pm\) 9.80E-3
& \cellcolor{heat!61} \textbf{4.27E-3} \(\pm\) 3.90E-4
& \cellcolor{heat!49} 2.35E+1 \(\pm\) 3.10E-1 \\

\texttt{mcjac}
& \cellcolor{heat!70} 8.27E-6 \(\pm\) 5.10E-8
& \cellcolor{heat!68} 5.84E-3 \(\pm\) 1.90E-4
& \cellcolor{heat!66} 1.11E-2 \(\pm\) 3.80E-4
& \cellcolor{heat!70} 1.10E-3 \(\pm\) 4.50E-5
& \cellcolor{heat!70} \textbf{3.49E-1} \(\pm\) 1.10E-2 \\

\texttt{naive}
& \cellcolor{heat!15} \textbf{1.04E-3} \(\pm\) 3.30E-5
& \cellcolor{heat!15} 3.88E+2 \(\pm\) 2.00E+0
& \cellcolor{heat!15} 5.22E+2 \(\pm\) 3.60E+0
& \cellcolor{heat!15} 5.08E+0 \(\pm\) 3.60E-2
& \cellcolor{heat!15} 1.91E+4 \(\pm\) 1.10E+2 \\

\texttt{jac}
& \cellcolor{heat!64} 1.38E-5 \(\pm\) 4.70E-8
& \cellcolor{heat!67} \textbf{6.71E-3} \(\pm\) 4.20E-4
& \cellcolor{heat!51} 2.53E-1 \(\pm\) 3.20E-3
& \cellcolor{heat!48} 3.20E-2 \(\pm\) 1.20E-3
& \cellcolor{heat!47} 3.44E+1 \(\pm\) 8.10E-1 \\

\texttt{hes}
& \cellcolor{heat!62} 1.72E-5 \(\pm\) 1.10E-7
& \cellcolor{heat!70} 3.97E-3 \(\pm\) 3.40E-4
& \cellcolor{heat!70} \textbf{5.26E-3} \(\pm\) 4.40E-4
& \cellcolor{heat!54} 1.31E-2 \(\pm\) 7.50E-4
& \cellcolor{heat!56} 5.65E+0 \(\pm\) 4.10E-1 \\

\hline
\end{tabular}%
}
\tablab{lorenz63-crossloss-m1}
\end{table}

To examine whether derivative supervision also improves the learned vector field, we evaluated every model under each loss on \(100{,}000\) held-out on-attractor states. The evaluation set contains \(6{,}250\) states from each of the \(16\) validation trajectories, sampled every \(0.07\) time units over \(t \in [10,447.43]\); none was used for training. Table~\tabref{lorenz63-crossloss-m1} reports trajectory-matching losses of $8.27\times10^{-6}$ for \texttt{mcjac} and $1.04\times10^{-3}$ for \texttt{naive}, corresponding to a 126-fold reduction. The lower loss for \texttt{mcjac} reflects improved trajectory accuracy in addition to derivative consistency.



\begin{table}[h!t!b!]
\centering
\caption{
Statistical comparison as in Table~\tabref{wd-m1}, except for the $m=2$ setting. Cell shading indicates relative performance within each column on a logarithmic scale, with darker shading corresponding to smaller values.
}
\vspace{0.05in}

\begin{tabular}{l | c c c c}
\hline
Method
& $W^1$
& $|\hat{\lambda}_1-\lambda_1^{\rm ref}|$
& $|\hat{\lambda}_2-\lambda_2^{\rm ref}|$
& $|\hat{\lambda}_3-\lambda_3^{\rm ref}|$ \\
\hline

\texttt{mc}
& \cellcolor{heat!70} {\bf 0.9217}
& \cellcolor{heat!15} 1.0901E-02
& \cellcolor{heat!15} 1.1439E-02
& \cellcolor{heat!15} 1.5818E-01 \\

\texttt{mcjac}
& \cellcolor{heat!42} 0.9916
& \cellcolor{heat!45} 4.6015E-03
& \cellcolor{heat!48} 3.0487E-04
& \cellcolor{heat!55} 1.4677E-02 \\

\texttt{naive}
& \cellcolor{gray!12} N/A
& \cellcolor{gray!12} N/A
& \cellcolor{gray!12} N/A
& \cellcolor{gray!12} N/A \\

\texttt{jac}
& \cellcolor{heat!42} 0.9933
& \cellcolor{heat!70} {\bf 2.0638E-03}
& \cellcolor{heat!40} 7.2673E-04
& \cellcolor{heat!70} {\bf 6.1131E-03} \\

\texttt{hes}
& \cellcolor{heat!15} 1.0393
& \cellcolor{heat!44} 4.6415E-03
& \cellcolor{heat!70} {\bf 2.7578E-05}
& \cellcolor{heat!49} 1.9238E-02 \\

\hline
\end{tabular}
\tablab{wd-m2}
\end{table}

The \texttt{mcjac} method achieves Hessian accuracy comparable to explicit Hessian supervision (\texttt{hes}) despite never evaluating Hessians directly, consistent with Theorem~\theoref{implicit_hessian}.
We next examine how the benefit of second-order supervision changes when more temporal information is provided during training. Increasing the rollout length from $m=1$ to $m=2$ evaluates each loss at three consecutive states, $(\ub_0,\ub_1,\ub_2)$, rather than two.

Table~\tabref{wd-m2} reports the corresponding results. At $m=2$, the additional temporal constraints improve the \texttt{mc} and \texttt{jac} methods without second-order supervision, and the best ensemble-averaged performance is no longer concentrated among the second-order methods. The $W^1$ values of \texttt{mc} ($0.92$) and \texttt{jac} ($0.99$) become competitive with those of their second-order counterparts, and \texttt{jac} gives the smallest errors in $\lambda_1$ and $\lambda_3$. These results do not, however, establish that second-order supervision is redundant, because ensemble-averaged metrics do not characterize trajectory-to-trajectory variability or rare failures.

Figure~\figref{lyap-x-box-m} shows the box plot of the leading Lyapunov exponent at rollout lengths $m=1$ and $m=2$. 
Each learned model estimates the leading Lyapunov exponent by integrating $1{,}000$ initial conditions over $T = 300$ time units, yielding an ensemble of $\hat{\lambda}_1$ estimates. 
This is a global stability statistic governed by the first-order tangent dynamics along the attractor.
%
Most trajectories from all derivative-informed methods yield leading Lyapunov exponents concentrated around the reference value, $\lambda_1^{\rm ref}=0.9041$. 
However, the two methods lacking second-order supervision (\texttt{mc} and \texttt{jac}) produce rare but catastrophic outliers at $m=1$. The \texttt{mc} method yields two trajectories with $\hat{\lambda}_1 \approx 26$ and $29.1$, while \texttt{jac} produces one trajectory with $\hat{\lambda}_1 \approx -16.7$. In contrast, the second-order methods \texttt{mcjac} and \texttt{hes} exhibit no such failures, with all estimated exponents remaining concentrated near the reference value. 
At $m=2$, \texttt{jac} yields the ensemble-mean \(\hat{\lambda}_1\) closest to the reference value, whereas the estimates from \texttt{mcjac} and \texttt{hes} are tightly concentrated around the reference. Additional temporal information can improve average accuracy, but second-order supervision may still be important for robustness.



\begin{figure}[h!t!b!]
    \centering
    \begin{subfigure}[b]{0.48\linewidth}
        \centering
        \includegraphics[width=\linewidth]{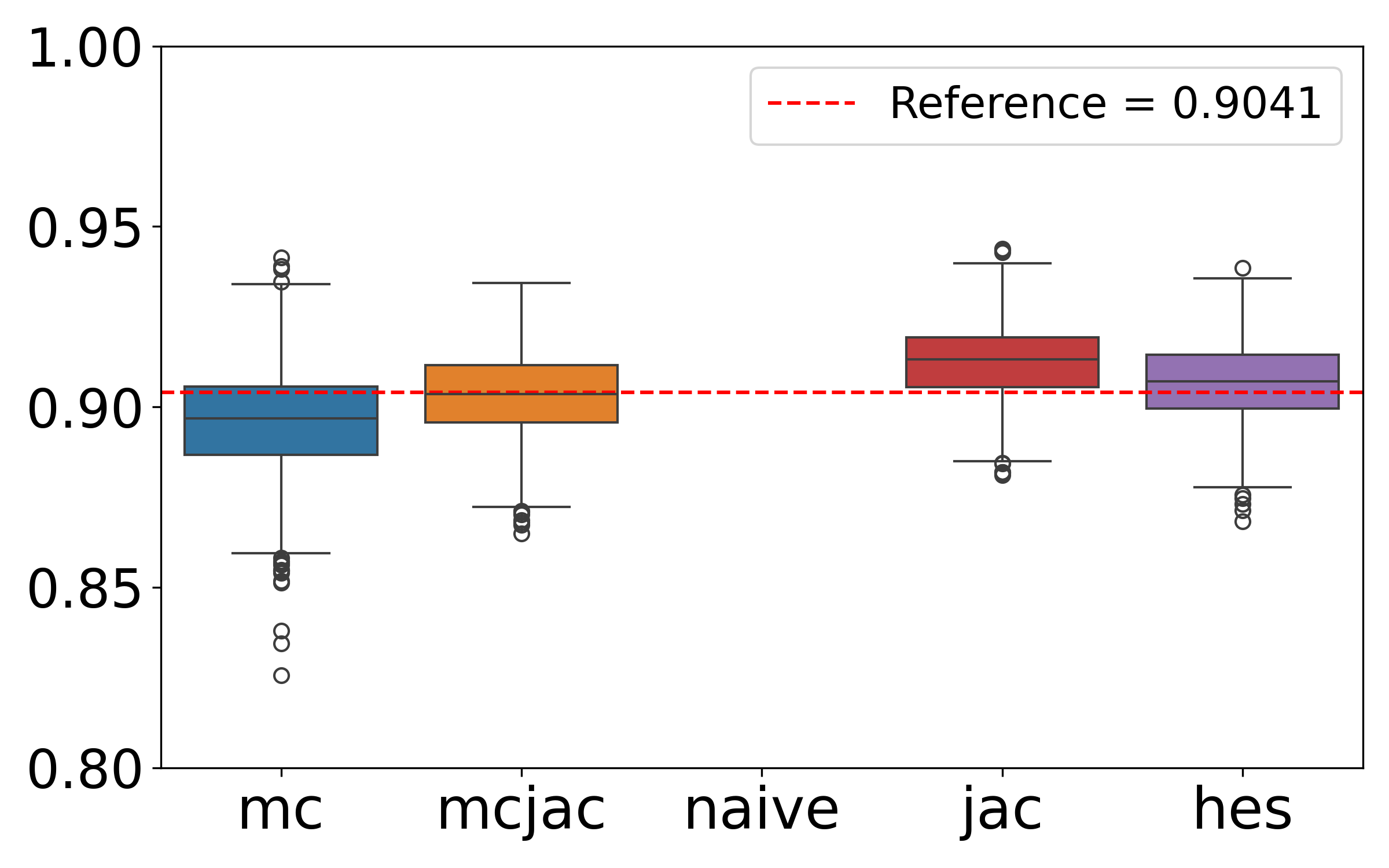}
        \caption{$m=1$}
        \figlab{lyap-x-box-lim}
    \end{subfigure}
    \hfill 
    \begin{subfigure}[b]{0.48\linewidth}
        \centering
        \includegraphics[width=\linewidth]{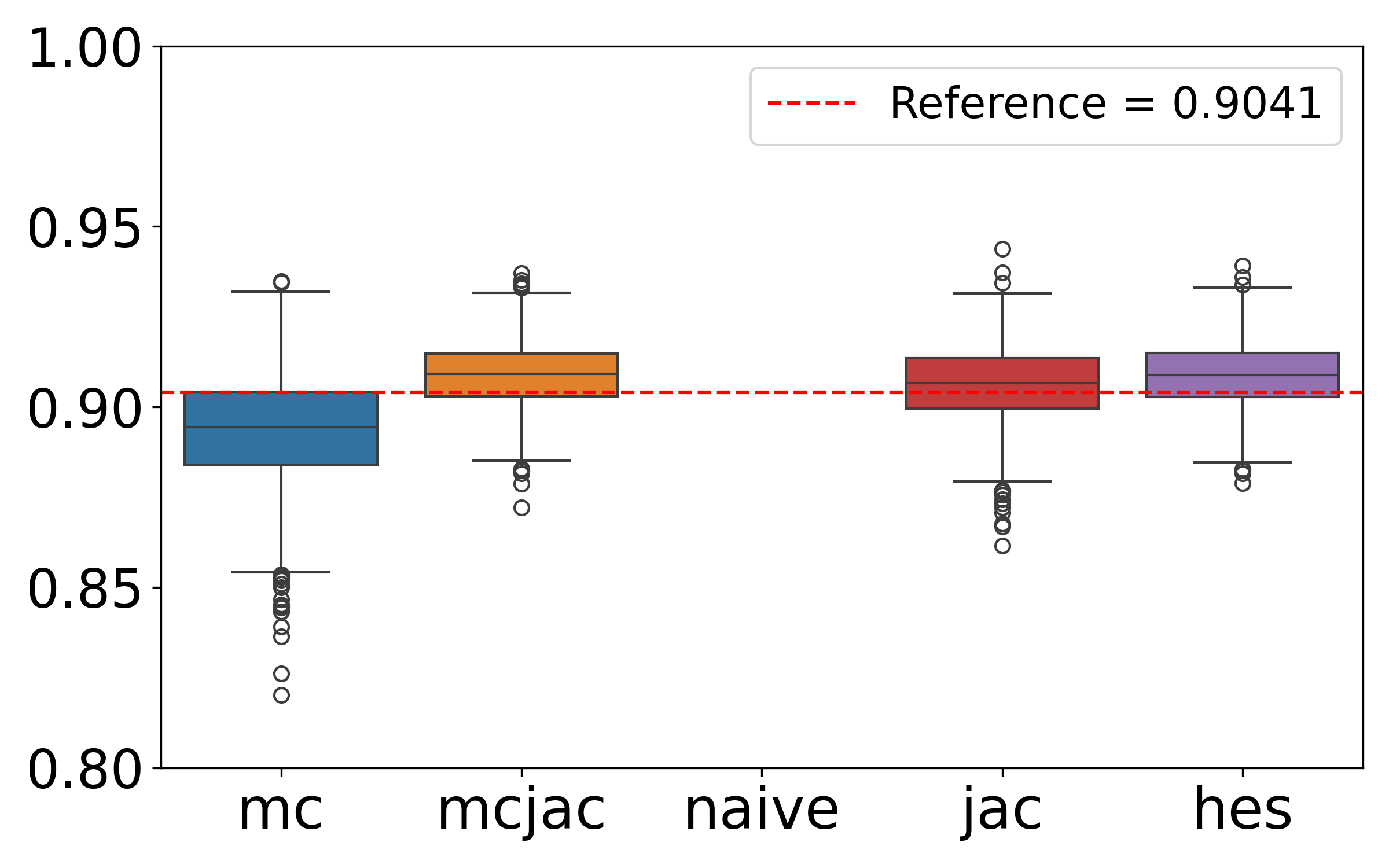}
        \caption{$m=2$}
        \figlab{lyap-m2-x-box-lim}
    \end{subfigure}
    \caption{Box plot of largest Lyapunov exponent ($\lambda_1$) estimation for the Lorenz~63 system at rollout lengths (a) $m=1$ and (b) $m=2$. The red dashed line indicates the reference value \(\lambda_1^{\mathrm{ref}}=0.9041\). The \texttt{naive} estimates lie outside the displayed range for both rollout lengths. At $m=1$, three extreme observations also lie outside the displayed range: two from \texttt{mc} ($\hat\lambda_1\approx26$ and $29.1$) and one from \texttt{jac} ($\hat\lambda_1\approx-16.7$). 
    }
    \figlab{lyap-x-box-m}
\end{figure}


Table~\tabref{le-wc} summarizes the Lyapunov spectrum errors reported in Tables~\tabref{wd-m1} and~\tabref{wd-m2} using the mean squared error (MSE),
$$
err = \frac{1}{3} \sum_{j=1}^{3} \left( \hat{\lambda}_j - \lambda_j^{\rm ref} \right)^2,
$$
together with the corresponding training wall-clock times. The second-order methods \texttt{mcjac} and \texttt{hes} achieve small errors across both rollout regimes ($2.56\times10^{-5}$ to $1.31\times10^{-4}$). The \texttt{mcjac} method attains the lowest error at $m=1$ ($2.56\times10^{-5}$), whereas the \texttt{jac} method becomes the most accurate method at $m=2$ once the catastrophic outliers disappear. Nevertheless, in the minimal temporal regime ($m=1$), \texttt{mcjac} remains more than two orders of magnitude more accurate than \texttt{jac}. Moreover, \texttt{mcjac} achieves this accuracy at a computational cost comparable to that of \texttt{hes}: $61.2$~s versus $64.6$~s at $m=1$, and $96.8$~s versus $90.8$~s at $m=2$.

\begin{table}[h!t!b!]
\centering
\caption{Mean squared error of the Lyapunov exponents ($\frac{1}{3}\sum_{j=1}^3({\hat{\lambda}_j}-\lambda_j^{\rm ref})^2$), where ${\hat{\lambda}_j}$ is the average across $1000$ initial conditions and $\lambda_j^{\rm ref}$ is the corresponding average from the reference Lorenz~63 simulation. Wall-clock times are reported in seconds.}
\begin{tabular}{c|cc|cc|cc|cc}
\hline
$m$ & \multicolumn{2}{c|}{\texttt{mc}} & \multicolumn{2}{c|}{\texttt{mcjac}} & \multicolumn{2}{c|}{\texttt{jac}} & \multicolumn{2}{c}{\texttt{hes}} \\
 & error & wc & error & wc & error & wc & error & wc \\
\hline
$1$ & 7.04E-04 & 41.0 & {\bf 2.56E-05} & 61.2 & 8.64E-03 & 51.8 & 9.29E-05 & 64.6 \\
$2$ & 8.42E-03 & 77.4 & 7.90E-05 & 96.8 & {\bf 1.40E-05} & 63.4 & 1.31E-04 & 90.8 \\
\hline
\end{tabular}
\tablab{le-wc}
\end{table}

\begin{table*}[h]
\centering
\small
\caption{Training-seed variability of the Lorenz~63 methods ($m=1$). Each method was trained with five seeds and evaluated from the same 1{,}000 on-attractor initial conditions; means and standard deviations are computed across trainings. For each seed, in that order, $n_{\mathrm{outlier}}$ counts the rollouts that either leave the attractor (divergence, collapse, or lobe collapse) or satisfy $|\hat{\lambda}_1-\lambda_1^{\mathrm{ref}}| > 0.1$, with $\lambda_1^{\mathrm{ref}} = 0.9065$, the true system's value on the same initial conditions. All other columns exclude these rollouts, separating on-attractor accuracy from failure frequency. $W^1$ denotes the invariant-measure Wasserstein distance, and $W^1_{\lambda_j}$ the Wasserstein distance between the finite-time $\lambda_j$ distributions of model and truth. 
For \texttt{naive} all rollouts fail, so its row is computed unconditionally over all $1{,}000$; it quantifies the size of the failure and is not comparable with the rows above.}
\vspace{0.05in}

\begin{tabular}{l | c | c | c | c | c | c | c | c | c}
\hline
& &
\multicolumn{2}{c|}{$W^1$} & \multicolumn{2}{c|}{$W^1_{\lambda_1}$} & \multicolumn{2}{c|}{$W^1_{\lambda_2}$} & \multicolumn{2}{c}{$W^1_{\lambda_3}$} \\
Method & $n_{\mathrm{outlier}}$
& mean & std
& mean & std
& mean & std
& mean & std \\
\hline

\texttt{mc}
& 0/0/18/38/13
& \cellcolor{heat!54} 0.8499
& \cellcolor{heat!28} 0.0931
& \cellcolor{heat!15} 1.35E-02
& \cellcolor{heat!15} 7.64E-03
& \cellcolor{heat!15} 9.48E-03
& \cellcolor{heat!15} 8.83E-04
& \cellcolor{heat!15} 8.54E-02
& \cellcolor{heat!15} 3.77E-02 \\

\texttt{mcjac}
& 0/0/0/0/0
& \cellcolor{heat!70} {\bf 0.8264}
& \cellcolor{heat!70} {\bf 0.0380}
& \cellcolor{heat!70} {\bf 3.19E-03}
& \cellcolor{heat!70} {\bf 1.93E-03}
& \cellcolor{heat!59} 1.13E-03
& \cellcolor{heat!70} {\bf 7.87E-05}
& \cellcolor{heat!70} {\bf 1.26E-02}
& \cellcolor{heat!70} {\bf 9.97E-03} \\

\texttt{naive}
& \cellcolor{gray!12} 1k/1k/1k/1k/1k
& \cellcolor{gray!12} 578.5175
& \cellcolor{gray!12} 546.8165
& \cellcolor{gray!12} 7.27E-01
& \cellcolor{gray!12} 1.88E-01
& \cellcolor{gray!12} 5.12E-01
& \cellcolor{gray!12} 3.04E-01
& \cellcolor{gray!12} 1.20E+01
& \cellcolor{gray!12} 5.18E-01 \\

\texttt{jac}
& 1/27/7/3/11
& \cellcolor{heat!15} 0.9079
& \cellcolor{heat!15} 0.1229
& \cellcolor{heat!42} 6.68E-03
& \cellcolor{heat!51} 3.09E-03
& \cellcolor{heat!60} 1.08E-03
& \cellcolor{heat!33} 4.01E-04
& \cellcolor{heat!60} 1.76E-02
& \cellcolor{heat!55} 1.45E-02 \\

\texttt{hes}
& 0/4/4/3/6
& \cellcolor{heat!39} 0.8713
& \cellcolor{heat!23} 0.1030
& \cellcolor{heat!53} 4.96E-03
& \cellcolor{heat!54} 2.88E-03
& \cellcolor{heat!70} {\bf 6.62E-04}
& \cellcolor{heat!53} 1.68E-04
& \cellcolor{heat!54} 2.19E-02
& \cellcolor{heat!54} 1.46E-02 \\

\hline
\end{tabular}
\label{tab:l63-seeds-m1}
\end{table*}

The results above use off-attractor initial conditions. To determine whether the conclusions persist under on-attractor initialization and across training seeds, we retrain all five methods with five seeds at \(m=1\). Each trained model is evaluated from the same \(1{,}000\) states drawn from the post-transient validation trajectories. All other evaluation settings, including the \(2{,}000\)-step transient discard, remain unchanged, and all rollouts remain finite. We classify a rollout as a catastrophic Lyapunov outlier if \(|\hat{\lambda}_1-\lambda_1^{\rm ref}|>0.1\). The reference spectrum is that of the true system on those same initial conditions, $(\lambda_1^{\rm ref},\lambda_2^{\rm ref},\lambda_3^{\rm ref}) = (0.9065, -0.0027, -14.5705)$
over the 1,000 trajectories. It is recomputed on the on-attractor initial conditions rather than carried over from the off-attractor set of the submitted tables, whose $\lambda_1^{\rm ref}=0.9041$ differs from it by $2.4\times10^{-3}$.

Table~\tabref{l63-seeds-m1} shows that \texttt{mcjac} produces no outlier in any of the 5,000 on-attractor rollouts and gives the best mean performance in three of the four Wasserstein metrics. Every \texttt{naive} rollout meets the outlier criterion for all five seeds. 
The remaining methods \texttt{mc}, \texttt{jac}, and \texttt{hes} produce outliers for some seeds. Within both paired comparisons, second-order supervision generally improves long-time statistical fidelity and tangent-space accuracy, although the improvement is not uniform across all Lyapunov-distribution metrics. However, the outlier counts do not follow the same pattern.
Despite explicit Hessian supervision, \texttt{hes} still produces outliers for four of the five seeds, showing that its outlier-free behavior in the previous single-seed, off-attractor experiment does not generalize across training seeds. Thus, second-order supervision improves on-attractor accuracy but does not by itself prevent catastrophic failures. We next examine the dynamical origin of these outliers.


\begin{figure}[h!t!b!]
    \centering
    \begin{subfigure}[b]{0.23\linewidth}
        \centering
        \includegraphics[trim=0cm .0cm 3.5cm 1.0cm,clip=true,width=\linewidth]{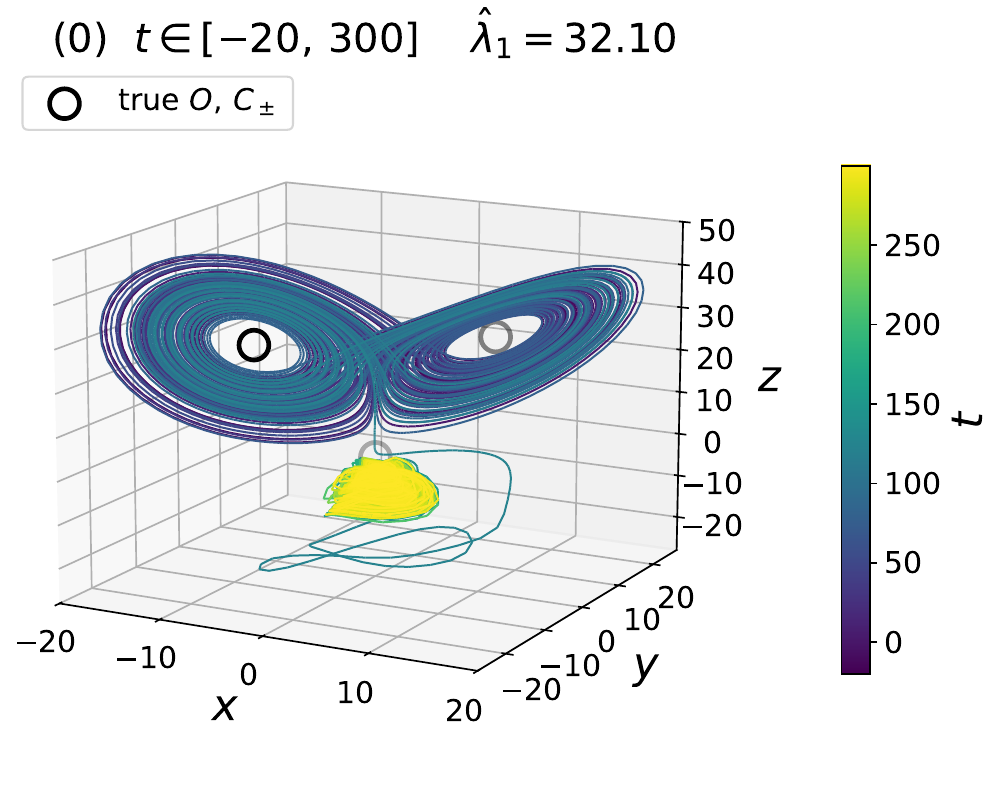}
        \caption{$t\in[0,320]$}
        \figlab{outlier-mc}
    \end{subfigure}
    \begin{subfigure}[b]{0.23\linewidth}
        \centering
        \includegraphics[trim=0cm 0.0cm 0cm 1.0cm,clip=true,width=\linewidth]{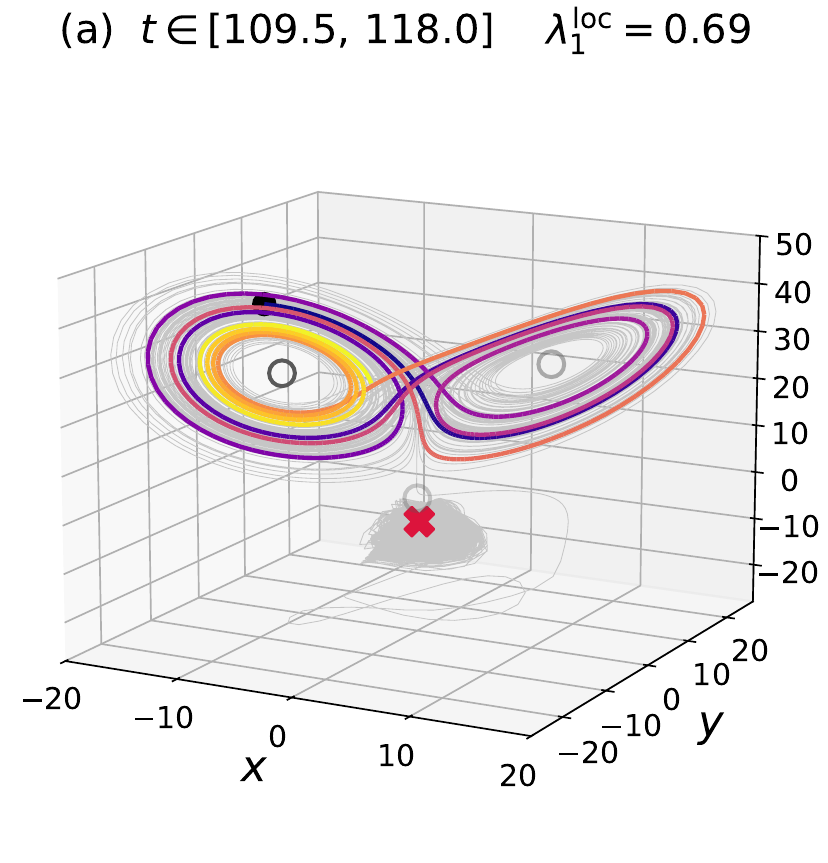}
        \caption{$t\in[129.5, 138.0]$}
        \figlab{outlier-mc-a}
    \end{subfigure} 
    \begin{subfigure}[b]{0.23\linewidth}
        \centering
        \includegraphics[trim=0cm 0.cm 0.cm 1.0cm,clip=true,width=\linewidth]{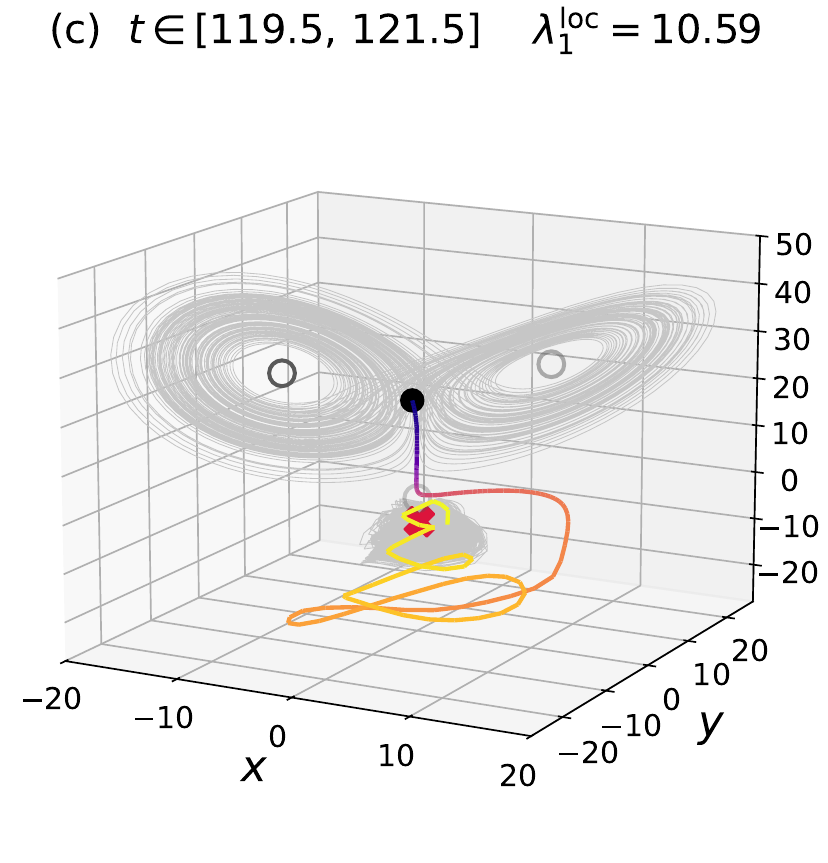}
        \caption{$t\in[139.5, 141.5]$}
        \figlab{outlier-mc-c}
    \end{subfigure}
    \begin{subfigure}[b]{0.23\linewidth}
        \centering
        \includegraphics[trim=0cm 0.cm 0.cm 1.0cm,clip=true,width=\linewidth]{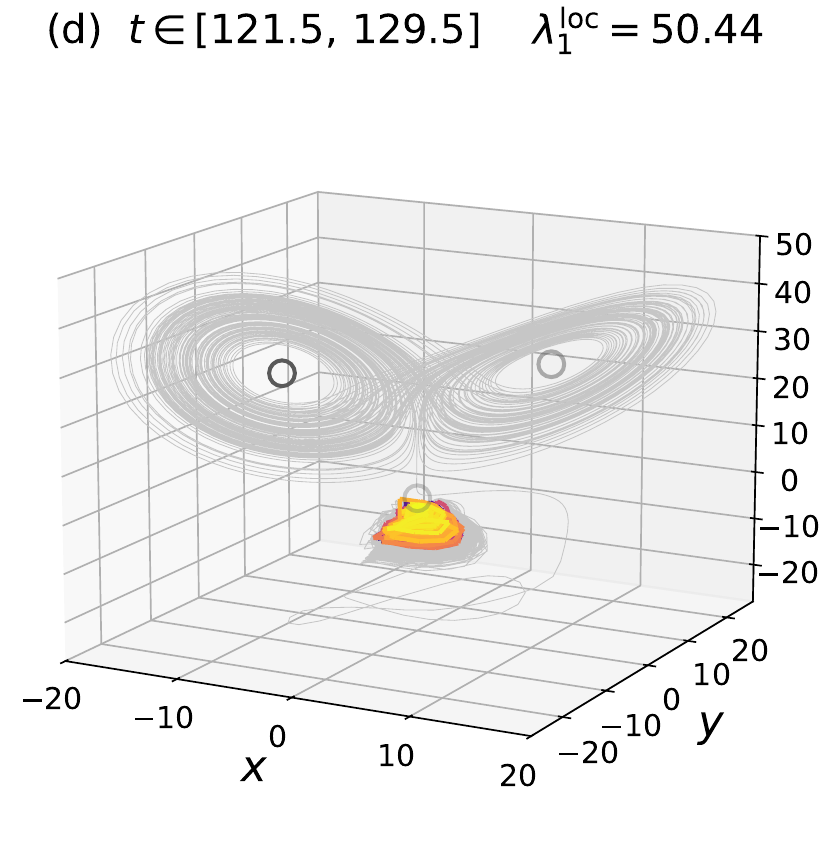}
        \caption{$t\in[141.5, 149.5]$}
        \figlab{outlier-mc-d}
    \end{subfigure}
    \caption{Transition underlying a catastrophic outlier from the five-seed on-attractor experiment for \texttt{mc}, with $\hat \lambda_1=32.1$. Colored curves show the learned trajectory, with time progressing from purple to yellow; gray curves show the reference attractor. (a) Full trajectory. (b) Ordinary lobe-switching motion before failure. The black dot marks the segment start, and the red cross marks the centroid of the spurious attractor. (c) Departure through the central transition region and rapid descent into the off-support region. (d) Capture by the compact spurious attracting set. Open circles denote the true Lorenz equilibria \(O\) and \(C_\pm\).}
    \figlab{outlier-mc}
\end{figure}

Figure~\figref{outlier-mc} traces a representative catastrophic outlier from the five-seed on-attractor experiment for \texttt{mc}, for which $\hat \lambda_1=32.1$. Starting from the on-attractor state ($1.29090,\ 0.82738,\ 19.59723$), the learned trajectory initially follows ordinary lobe-switching dynamics.
The nonzero equilibria \(C_\pm\) organize the two lobes of the attractor, while the saddle at \(O\) organizes the central transition region. 
At $t=139.5$, the trajectory enters this region and transiently aligns with the \(z\)-axis, which is invariant under the true Lorenz dynamics. The eigenvalues of the true Jacobian at the origin are \(+11.83\), \(-8/3\), and \(-22.83\); the \(z\)-axis is its weak stable direction, while the positive eigenvalue governs departure toward either lobe. Under the true dynamics, growth along this unstable direction redirects the trajectory into a lobe before it reaches $O$. The learned field instead continues to drive the trajectory downward, carrying it below $z=0$. At \(t = 140.4\), the trajectory leaves the support of the reference attractor and is eventually captured by a compact spurious attracting set. 

\begin{figure}[h!t!b!]
    \centering
    \begin{subfigure}[b]{0.54\linewidth}
        \centering
        \includegraphics[trim=0cm .0cm 0.1cm 0.8cm,clip=true,width=\linewidth]{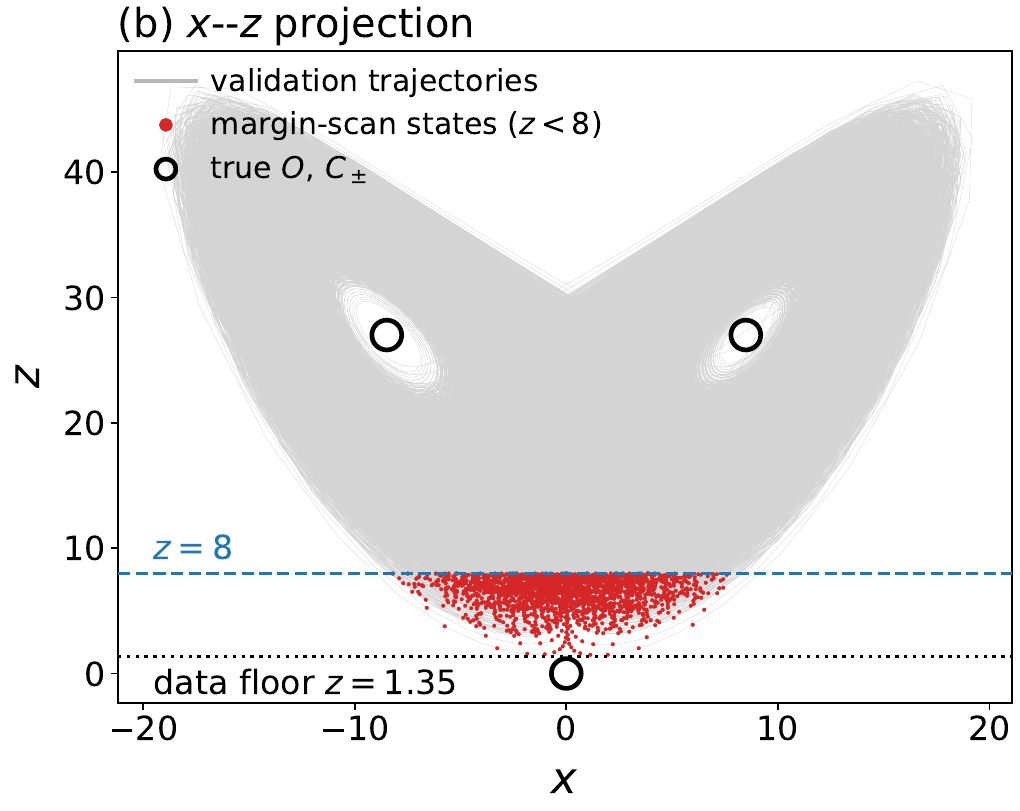}
        \caption{$xz$-plot}
        \figlab{basin-margin-scan-xz}
    \end{subfigure}
    \hfill
    \begin{subfigure}[b]{0.42\linewidth}
        \centering
        \includegraphics[trim=0.2cm .0cm 0cm 0.8cm,clip=true,width=\linewidth]{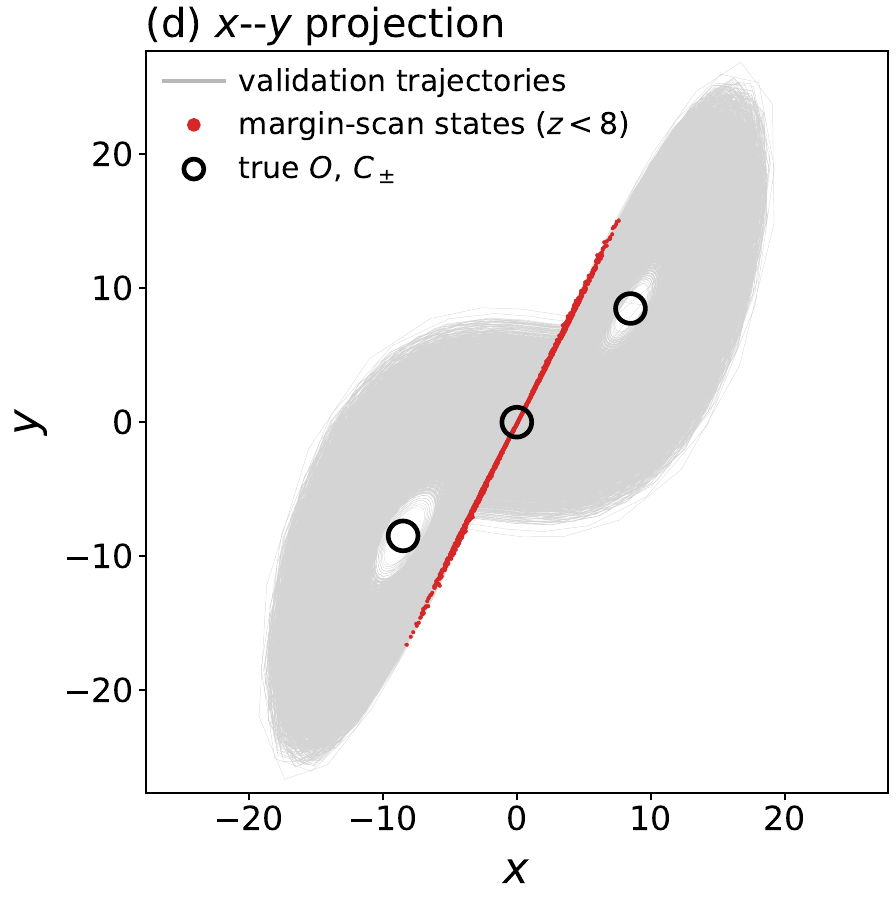}
        \caption{$xy$-plot} 
        \figlab{basin-margin-scan-xy}
    \end{subfigure}
    \caption{States selected for the directional capture scan.
    (a) Validation trajectories in the $x$--$z$ plane (gray), with $2{,}000$ margin-scan states satisfying $z<8$ (red). The blue dashed line marks the selection threshold \(z=8\), and the black dotted line at \(z=1.35\) marks the lower edge of the training-data support.
(b) In the $x$--$y$ projection, the selected states form a narrow corridor through the central transition region. Black circles denote the true equilibria $O$ and $C_\pm$.
    }
    \figlab{basin-margin-scan}
\end{figure}

To probe how readily trajectories near the true attractor are captured by spurious learned dynamics, we perform a directional capture scan. From the 16 validation trajectories, we select \(2{,}000\) states satisfying \(z<8\), as shown in Figure~\figref{basin-margin-scan}. In the \(x\)--\(y\) projection, these states form a narrow corridor near the origin.

Each state is displaced by $\delta$ in the $-z$ direction and integrated for
$T_{\rm cap}=60$. Let
$
\mathcal{T}=[0.9T_{\rm cap},T_{\rm cap}]
$
denote the final $10\%$ of the rollout. A rollout is classified as captured if
$
\max_{t\in\mathcal{T}}
\|\mathbf{x}(t)-\mathbf{x}(T_{\rm cap})\|_2 < 5
$
or
$
\min_{t\in\mathcal{T}}
d_{\mathcal D}(\mathbf{x}(t)) > 1$, 
with 
$d_{\mathcal D}(\mathbf{x})
=
\min_{\mathbf{x}_j\in\mathcal D}
\|\mathbf{x}-\mathbf{x}_j\|_2.
$
The first condition detects localized trajectories, whereas the second detects trajectories that remain outside the training-data support.
Let $n_{\rm cap}(\delta)$ be the number of rollouts satisfying this criterion. We define $1\%$ capture threshold, $\delta_{01}$, as the displacement at which the interpolated capture curve first
reaches $1\%$, corresponding to $20$ of the $2{,}000$ states.

\begin{figure}[h!t!b!]
    \centering
    \includegraphics[trim=0cm .0cm 0.0cm 0.0cm,clip=true,width=\linewidth]{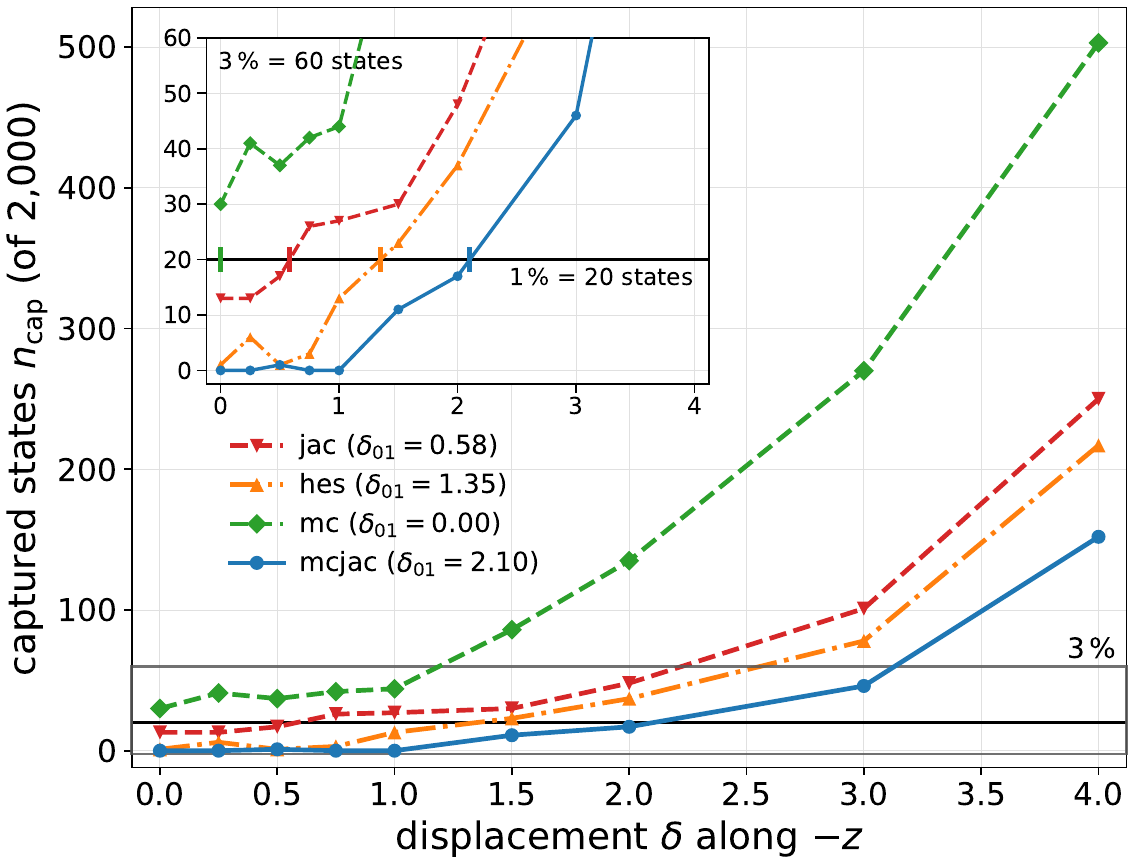}
    \caption{Capture scan for the learned Lorenz~63 models.
A set of $2{,}000$ states sampled from the true attractor is displaced by $\delta$ along the $-z$ direction and integrated forward.
The quantity $n_{\rm cap}$ denotes the number of displaced states satisfying the prescribed catastrophic-capture criterion.
}
    \figlab{capture-scan}
\end{figure}

Figure~\figref{capture-scan} shows method-dependent differences in the onset of catastrophic capture under the prescribed $-z$ perturbation. The number of captured states generally increases with $\delta$. In the zoomed-in view, the $1\%$ capture threshold increases in the order $\delta_{01}=0.00$ for \texttt{mc},
$0.58$ for \texttt{jac},
$1.35$ for \texttt{hes}, and
$2.10$ for \texttt{mcjac}.
A larger $\delta_{01}$ means that a greater displacement from the true attractor is required before $1\%$ of the probed states satisfy the capture criterion. It should not be interpreted as the geometric distance to a basin boundary. 

Within both paired comparisons, second-order supervision increases the capture threshold: from \texttt{mc} to \texttt{mcjac} and from \texttt{jac} to \texttt{hes}. This indicates lower susceptibility to spurious capture under the present scan, but does not guarantee stability across training seeds. Despite its larger threshold, \texttt{hes} produces catastrophic outliers for four of the five seeds. The still larger threshold of \texttt{mcjac} suggests an additional benefit from randomized neighborhood supervision: Jacobian matching at perturbed states constrains the learned field over a finite neighborhood of the training data rather than only on the attractor, consistent with its larger threshold and lower capture counts in the scan.
 

\subsection{Lorenz~96 Model}
\seclab{clorenz96}
The Lorenz 63 results show benefits from second-order supervision under minimal temporal supervision and across training seeds, with \texttt{mcjac} performing comparably to \texttt{hes} and showing better robustness across training seeds.
We next consider the coupled Lorenz~96 system in~\eqnref{lorenz96}, a $396$-dimensional multiscale slow--fast system with $K=36$ slow variables and $J=10$ fast variables per slow site. At the nominal forcing $F=10$ and under larger out-of-distribution forcing values, the system exhibits high-dimensional chaotic dynamics with multiple positive Lyapunov exponents.  
In this experiment, we adopt a universal differential equation framework~\citep{rackauckas2020universal} in which only a subset of the dynamics is learned. We preserve the dissipation and forcing of the slow equation \eqnref{lorenz96-slow} and replace only its quadratic nonlinear advection term with a neural network. We consider
\begin{align}
  \eqnlab{lorenz96-slow-nn}
  \DD{X_k}{t} & = \hat{f}_k\LRp{X;\theta} - X_k + F - h \overline{Y}_k,
\end{align}
where $\hat{f}_k(X;\theta)$ parameterizes the quadratic nonlinear advection operator among the slow variables. The linear damping term $-X_k$, the constant forcing $F$, and the true coupling to the fast variable $\overline{Y}_k$ are kept exact. Thus, the learning task is not to identify the full vector field, but to recover a structured nonlinear component within a strongly chaotic multiscale system.

This experiment is a controlled analogue of closure and UDE surrogate learning, rather than equation discovery from observations alone. The full two-scale model provides offline trajectory and derivative targets while only a structured nonlinear component is learned, allowing us to compare first- and second-order supervision and to test whether randomized Jacobian matching achieves second-order consistency without
reference Hessians.

The Hessian of the true quadratic interaction $-X_{k-1}(X_{k-2}-X_{k+1})$ is sparse, local, and state-independent, with nonzero entries of $\LRc{-1,+1}$. The interaction itself is translation-equivariant under cyclic shifts of the slow index $k$. Exploiting both properties, we take $\hat{f}_k$ to be a shared local multilayer perceptron (MLP), applied identically at every slow site. This locality restricts the second-derivative coupling to a sparse banded stencil, reducing the number of structurally nonzero Hessian entries from $\mc{O}(K^2)$ to $\mc{O}((2W+1)^2)$. In this study, we choose $W=2$, so that $\hat{f}_k=\hat{f}_k\LRp{X_{k-2},X_{k-1}, X_k, X_{k+1}, X_{k+2};\theta}$, where the local mapping is implemented by an MLP with layer widths $[5, 64, 64, 1]$ and \texttt{tanh} activations. This architecture directly embeds the locality and symmetry structure of Lorenz~96 into the neural parameterization, making second-order supervision computationally tractable. 

A single trajectory is generated from a random initial condition---slow variables perturbed around $F$, fast variables initialized near zero---by integrating for $200$ warmup steps ($t=1$) to drive the system onto the attractor. We verified that both the slow- and fast-variable energies reach statistically stationary levels within this interval. The resulting trajectory is used as training data. A second trajectory generated from an independently sampled initial condition with the same warmup procedure is used for validation, and all evaluation metrics are computed on this independent validation trajectory.
Training uses the TSIT5 integrator~\citep{tsitouras2011runge} with timestep $\Delta t=0.005$ over $t\in[0,500]$, batch size $n_b=128$, rollout length $m=2$, and $10{,}000$ epochs. Unlike the Lorenz~63 experiments, all Lorenz~96 rollouts and corresponding analyses were performed in double precision. Although double precision was not required for integration stability, we used it throughout the Lorenz~96 experiments for numerical consistency because single precision produced a systematic upward bias in the estimated Lyapunov spectrum.


Unlike the Lorenz~63 experiments, here the learning problem combines partial learning of a structured nonlinear component with a strongly translation-equivariant architecture. These inductive biases reduce the effective complexity of the approximation problem, so within the training regime the performance gap between supervision methods is small.
At the same time, all models are trained at $F=10$ and evaluated up to the more chaotic out-of-distribution regime $F=20$. We first compare all five methods at two representative forcing values, $F=10$ and $F=20$, including short- and long-time trajectories, Lyapunov spectra, curvature statistics, invariant measures, and Jacobian recovery. We then examine how the methods progressively separate as the forcing parameter is swept between these two regimes.


\begin{figure}
    \centering    
    \includegraphics[width=0.9\linewidth]{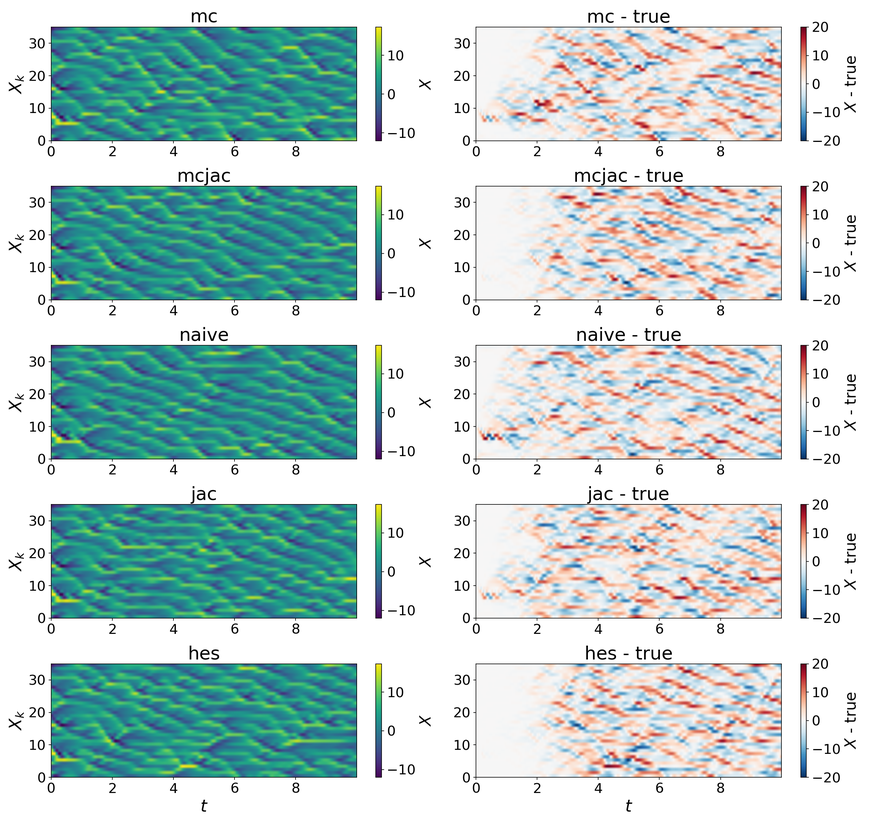}
    \caption{Hovm\"oller diagrams of the slow variables $X_k$ for $t\in[0,10]$ at $F=10$. Predicted trajectories from each learned model (left) and differences from the true trajectory (right) are shown. All five methods reproduce the characteristic wave-like spatiotemporal structures of the system with comparable short-term accuracy. The second-order supervised methods (\texttt{mcjac} and \texttt{hes}) also exhibit slightly smaller short-time residuals during the early-time interval ($t\le 2$).}
    \figlab{l96-hov-F10-t10}
\end{figure}

\begin{figure}
    \centering 
    \includegraphics[width=0.9\linewidth]{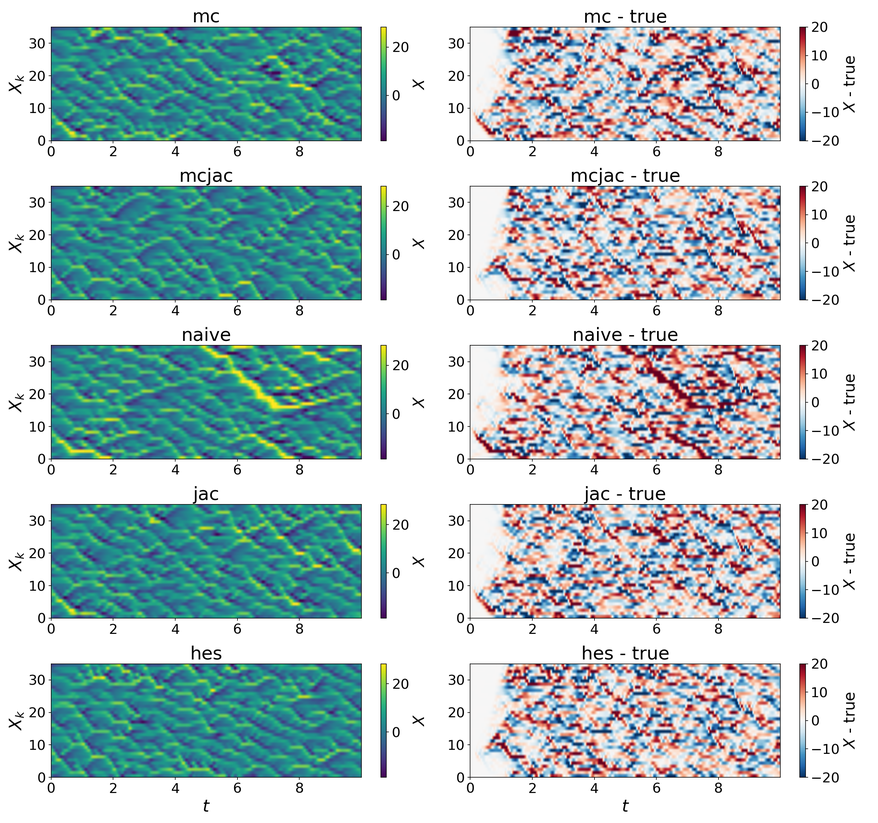}
    \caption{Hovm\"oller diagrams of the slow variables $X_k$ for $t\in[0,10]$ at $F=20$. All methods retain similar spatiotemporal structures. However, the \texttt{naive} model exhibits stronger localized high-amplitude regions.}
    \figlab{hov-F20-t10}
\end{figure}

\begin{figure}
    \centering    \includegraphics[width=0.9\linewidth]{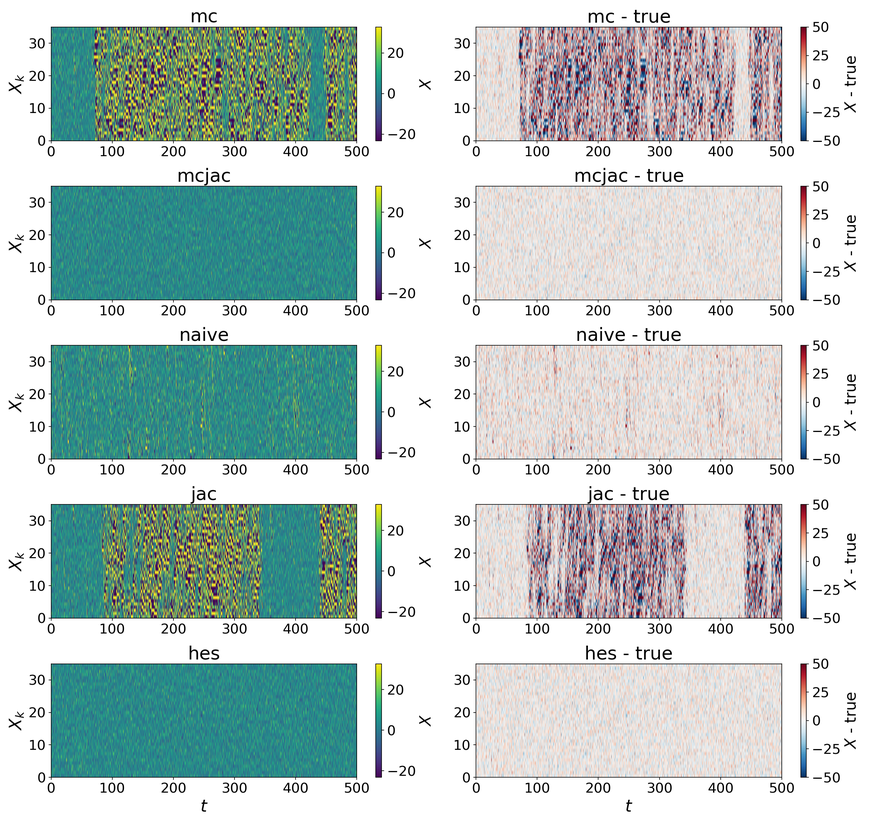}
    \caption{Long-horizon Hovm\"oller diagrams of the slow variables $X_k$ for $t\in[0,500]$ at $F=20$. \texttt{mcjac} and \texttt{hes} show very good agreement with the true dynamics over the entire rollout. The \texttt{naive} model remains bounded but shows intermittent localized high-amplitude bursts. The \texttt{mc} model transitions near $t\approx70$ into a spurious high-amplitude attractor-like regime for most of the rollout. The \texttt{jac} model similarly enters a spurious regime near $t\approx80$, briefly recovers over $t\approx350\text{--}430$, and later transitions back.}
    \figlab{hov-F20}
\end{figure}

We begin with the space--time evolution of the slow variables. Figure~\figref{l96-hov-F10-t10} shows short-window Hovm\"oller diagrams at $F=10$, while Figures~\figref{hov-F20-t10} and~\figref{hov-F20} show the corresponding short- and long-window results at $F=20$. At $F=10$, the system is already spatiotemporally chaotic with leading Lyapunov exponent $\lambda_1\approx6$ (Figure~\figref{l96-lyap-F10}), yet all five learned models remain visually similar on the short window. The second-order methods (\texttt{mcjac} and \texttt{hes}) nevertheless exhibit slightly smaller short-time residuals during the early interval $t\le2$.
At $F=20$, the true system enters a more strongly chaotic regime with $\lambda_1\approx17.5$ (Figure~\figref{l96-lyap-F20}). On the short window (Figure~\figref{hov-F20-t10}), all methods still appear reasonable, although the \texttt{naive} model develops stronger localized high-amplitude regions. The long-window rollout (Figure~\figref{hov-F20}), however, reveals a much clearer separation between the methods. The second-order methods, \texttt{mcjac} and \texttt{hes}, remain in close agreement with the true dynamics throughout the rollout. The \texttt{naive} model stays bounded but exhibits intermittent high-amplitude bursts. In contrast, \texttt{mc} transitions near $t\approx70$ into a spurious high-amplitude attractor-like regime, while \texttt{jac} undergoes a similar transition near $t\approx80$, briefly recovering over $t\approx350\text{--}430$ before transitioning again. Thus, in this
high-$F$ extrapolation regime, the two methods with second-order supervision are more robust than their first-order counterparts.

The two methods that develop sustained spurious regimes, \texttt{mc} and \texttt{jac}, lack second-order supervision. To separate the effect of forcing level from distance to the training condition, we repeated the experiment in reverse, training at $F=20$ and evaluating at lower $F$ (Appendix Figure~\figref{l96-reverse-F}). Derivative errors remained larger at larger $F$, even though $F=20$ was now the training condition. The increased difficulty is therefore associated with the forcing level rather than extrapolation distance alone. The reverse sweep did not show a systematic increase in the advantage of second-order supervision. We take the original $F=20$ results as evidence of improved robustness in this strongly chaotic extrapolation regime, but not of a growing benefit with chaoticity.

\begin{figure}
    \centering  
    \begin{subfigure}[b]{0.9\linewidth}
    \centering
    \includegraphics[width=\linewidth]{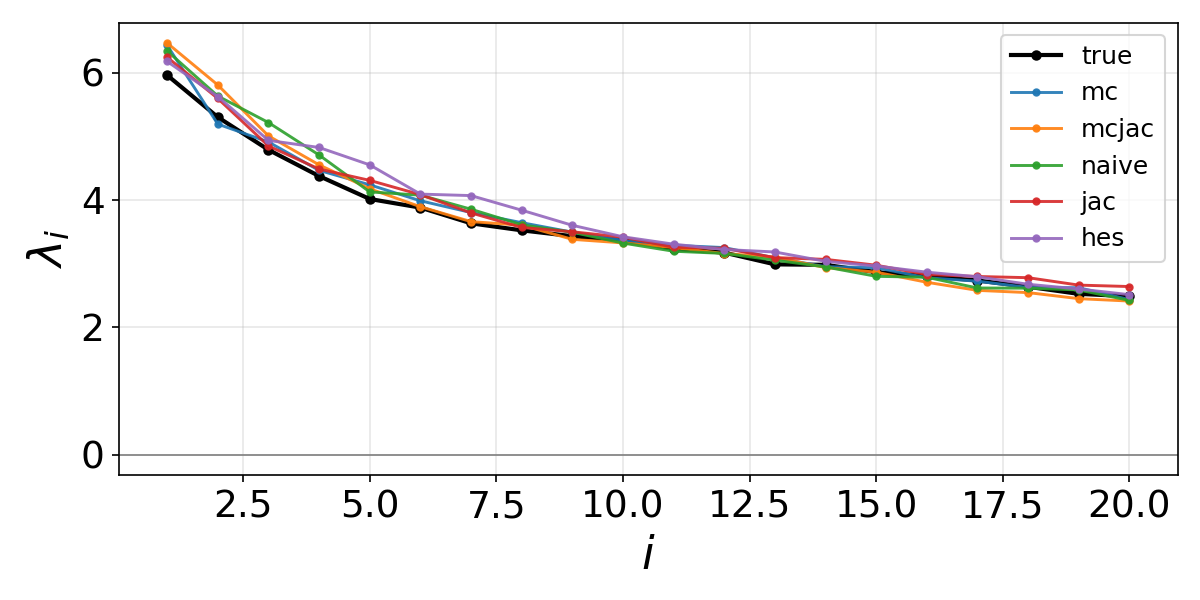}
        \caption{$F=10$}
        \figlab{l96-lyap-F10}
    \end{subfigure}
    \hfill
    \begin{subfigure}[b]{0.9\linewidth}
    \centering
    \includegraphics[width=\linewidth]{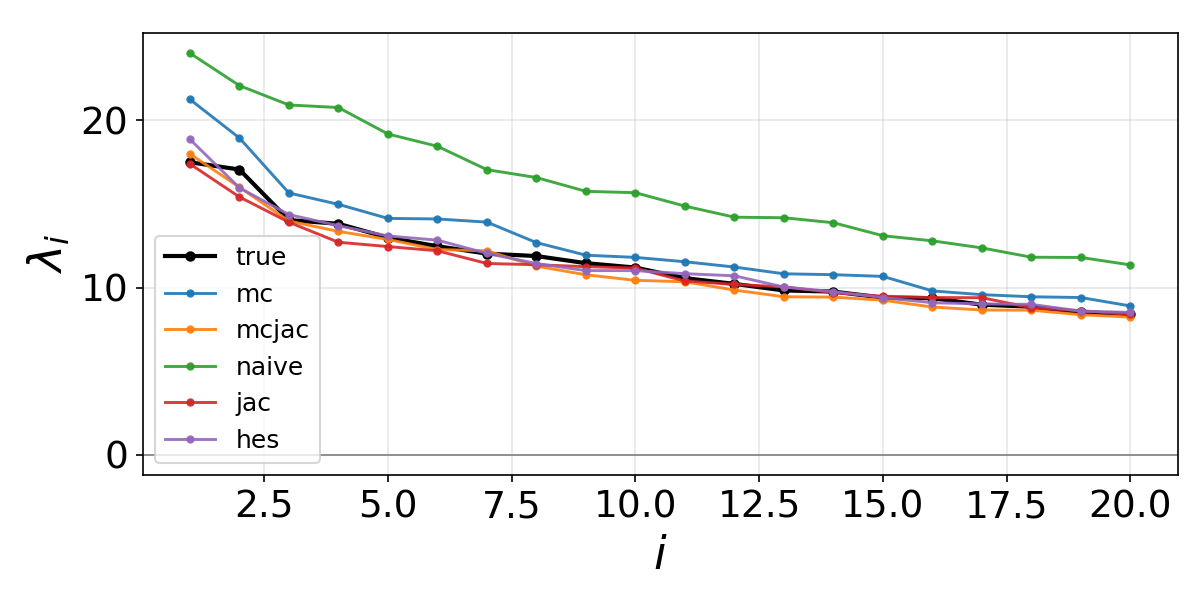}
        \caption{$F=20$}
        \figlab{l96-lyap-F20}
    \end{subfigure}
    \figlab{l96-lyapunov}
    \caption{Top-20 Lyapunov spectra of the learned models. (a) At $F=10$, all methods remain in good agreement with the true spectrum. (b) At $F=20$, clear separation emerges. \texttt{jac}, \texttt{mcjac} and \texttt{hes} remain closer to the true spectrum across the leading modes, while \texttt{mc} moderately overestimates and \texttt{naive} strongly overestimates the positive Lyapunov exponents.}
\end{figure}

The Lyapunov spectrum shows a similar contrast. At $F=10$, all methods agree closely with the true spectrum (Figure~\figref{l96-lyap-F10}). At $F=20$, the methods separate clearly. \texttt{jac}, \texttt{mcjac}, and \texttt{hes} remain closer to the true spectrum across the leading modes, whereas \texttt{mc} moderately overestimates and \texttt{naive} strongly overestimates the positive Lyapunov exponents (Figure~\figref{l96-lyap-F20}).
These spectral errors do not align with the long-time trajectory behavior in Figure~\figref{hov-F20}. Despite strongly overestimating the spectrum, \texttt{naive} remains on a bounded attractor, whereas \texttt{jac} reproduces the leading spectrum relatively accurately yet still intermittently transitions into a spurious high-amplitude regime.
Lyapunov-spectrum accuracy is therefore not a sufficient diagnostic for long-time fidelity. 

The Lyapunov spectrum characterizes tangent-space dynamics along the realized trajectory. Changes in the geometry or location of the attractor can therefore affect the spectrum by changing where in state space the local Jacobians are sampled.

We quantify this dependence directly at $F=20$. Evaluating the learned Jacobian along the true trajectory rather than its own rollout reduces the top-20 spectral error $\varepsilon_{\lambda}$ by $88.6$--$99.7\%$ across all five methods (Appendix~\secref{l96-spectrum-decomposition}). Thus, much of the spectral discrepancy is associated with the trajectory along which the local Jacobians are sampled.

We next examine whether the curvature-length diagnostic reflects differences in the learned vector field or differences in the states sampled by the
rollout. We consider
\[
\Lcurv(X,g)
=
\frac{\|\nabla g(X)\|_{\mathrm F}}
     {\|\nabla^2 g(X)\|_{\mathrm F}},
\qquad
M(g,\mu)
=
\operatorname*{median}_{X\sim\mu}
\Lcurv(X,g).
\]
Here, \(\mu_f\) and \(\mu_{\hat f}\) denote the empirical occupation measures of the true and learned rollouts over \(t\in[10,500]\) after burn-in. 
We estimate \(M(g,\mu)\) at the same 200 equally spaced times along each rollout, with consecutive samples \(2.46\) time units apart. For the true trajectory, \(\mu_f\) serves as a finite-time approximation to the invariant measure of \(f\). For each learned model, \(\mu_{\hat f}\) describes the states visited over the observation window, without assuming stationarity. Because $M(g,\mu)$ depends on both the vector field $g$ and the occupation
measure $\mu$, a change in $M$ can arise from either the learned field or the states visited by its rollout.

\begin{table}[t]
\centering
\caption{
Field-state decomposition of the curvature length scale at \(F=20\). Each entry is the median over 200 sampled states. The true-field reference is \(M(f,\mu_f)=0.3685\). 
}
\label{tab:l96-field-state-decomposition}
\begin{tabular}{c|ccc}
\hline
method & $M(\hat f,\mu_f)$ & $M(f,\mu_{\hat f})$ & $M(\hat f,\mu_{\hat f})$ \\
\hline
\texttt{mc}    & 0.3683 & 0.9339 & 0.9279 \\
\texttt{mcjac} & 0.3683 & 0.3673 & 0.3672 \\
\texttt{naive} & 0.3682 & 0.4160 & 0.4153 \\
\texttt{jac}   & 0.3683 & 0.8248 & 0.8207 \\
\texttt{hes}   & 0.3684 & 0.3692 & 0.3691 \\
\hline
\end{tabular}
\end{table}

To separate these effects, we perform a field-state decomposition at \(F=20\). For each method, we evaluate the learned field on states from the true attractor, the true field on states from the learned rollout, and the learned field on its own rollout states. The resulting medians are reported in Table~\tabref{l96-field-state-decomposition}. 
The true-field reference is $M(f,\mu_f)=0.3685$.
Evaluating the learned fields on the reference-trajectory states gives $M(\hat f,\mu_f)=0.3682\text{--}0.3684$, which agree with the reference value. Evaluating the true field on the learned-rollout states gives
$
M(f,\mu_{\hat f})
\simeq
M(\hat f,\mu_{\hat f}).
$
For \texttt{mc}, the two values are $0.9339$ and $0.9279$; for \texttt{jac}, they are $0.8248$ and $0.8207$. 
At fixed rollout states, replacing the learned field by the true field changes the median by only $0.0060$ for \texttt{mc} and $0.0041$ for \texttt{jac}. 
At fixed field $\hat f$, changing the occupation measure from $\mu_f$ to $\mu_{\hat f}$ changes the median by approximately $0.56$ and $0.45$, respectively. 
The shift in \(\Lcurv\) is driven mainly by the states occupied by the learned rollouts, not by the difference between \(f\) and \(\hat f\) at those states. 
The corresponding finite-window distributions are shown in Appendix~\secref{lcurv}. For \texttt{mc} and \texttt{jac}, the shift in state-space occupation appears as a second lobe in the distribution.

However, $\Lcurv$ is nearly insensitive to errors in
the learned slow-variable Hessian. Because the Frobenius norm sums over output components, the Hessian norm splits by output block, 
$
\|\nabla^2 f\|_F^2
=
\|\nabla^2 f_{\text{slow}}\|_F^2+
\|\nabla^2 f_{\text{fast}}\|_F^2.
$
Since the slow--fast coupling is linear, each block is determined solely by its quadratic advection term: $\|\nabla^2 f_{\rm slow}\|_F=\sqrt{4K}=12.00$ and
$\|\nabla^2 f_{\rm fast}\|_F=cJ\sqrt{4JK}=3794.73$. 
The learned slow-variable component contributes only
$12^2/3794.75^2\approx 0.001\%$ of the squared full Hessian norm. 
We use the relative Hessian error for the slow variable,
$$
\varepsilon_{H}
=\frac{1}{N}
\sum_{i=1}^{N}
\frac{
\|\nabla^2\hat f_{\mathrm{slow}}(X_i)
-\nabla^2 f_{\mathrm{slow}}(X_i)\|_F
}{
\|\nabla^2 f_{\mathrm{slow}}(X_i)\|_F
},
$$
where $X_i\sim\mu_f$ and all methods are evaluated at the same $N=50$ states sampled from the true trajectory.

\begin{table}[t]
\centering
\caption{
Mean relative Hessian error of the learned slow variable at $F=20$, averaged over the same $50$ states sampled from the true trajectory.
}
\label{tab:l96-hessian-error}
\begin{tabular}{lccccc}
\hline
method
& \texttt{hes}
& \texttt{mcjac}
& \texttt{jac}
& \texttt{mc}
& \texttt{naive}\\
\hline
$\varepsilon_{H}$
& 0.485
& 0.649
& 1.163
& 1.510
& 1.619\\
\hline
\end{tabular}
\end{table}

Table~\tabref{l96-hessian-error} shows that second-order supervision lowers the Hessian error in both paired comparisons: from $1.510$ for \texttt{mc} to $0.649$ for \texttt{mcjac}, and from $1.163$ for \texttt{jac} to $0.485$ for \texttt{hes}. The \texttt{naive} and \texttt{mc} models, neither of which receives second-order supervision, have comparable errors of $1.619$ and $1.510$. Their long-time behaviors differ: \texttt{mc} leaves the amplitude range of the true attractor at $t=69.5$ and stays outside for most of the run, whereas \texttt{naive} exceeds it only in bursts and always returns. The slow-variable Hessian error alone does not explain the high-amplitude regime switching.

We next turn from the geometric structure of the learned vector field to the statistical occupation of the attractor. The one-dimensional marginal density of $X$ captures where trajectories spend most of their time on the attractor. At $F=10$, all methods agree closely with the true density in Figure~\figref{l96-pdf-F10}. At $F=20$, \texttt{mcjac} and \texttt{hes} remain close to the true density, whereas \texttt{mc} and \texttt{jac} produce broadened distributions with heavy tails extending to values near $\pm100$. The \texttt{naive} method reproduces the marginal distributions reasonably well at the representative sites despite overestimating the Lyapunov spectrum (Figure~\figref{l96-lyap-F20}).
Marginal densities show where trajectories spend time, but they do not show whether the learned vector field is accurate. As the curvature analysis above demonstrates, \texttt{naive} distorts the underlying second-order structure of the dynamics even though its marginal density remains close to the truth. Because \texttt{naive} is trained only on state values, it can reproduce the occupation statistics while failing to recover the local derivative structure that governs perturbation growth.


\begin{figure}
    \centering  
    \begin{subfigure}[b]{0.9\linewidth}
    \centering
    \includegraphics[width=\linewidth]{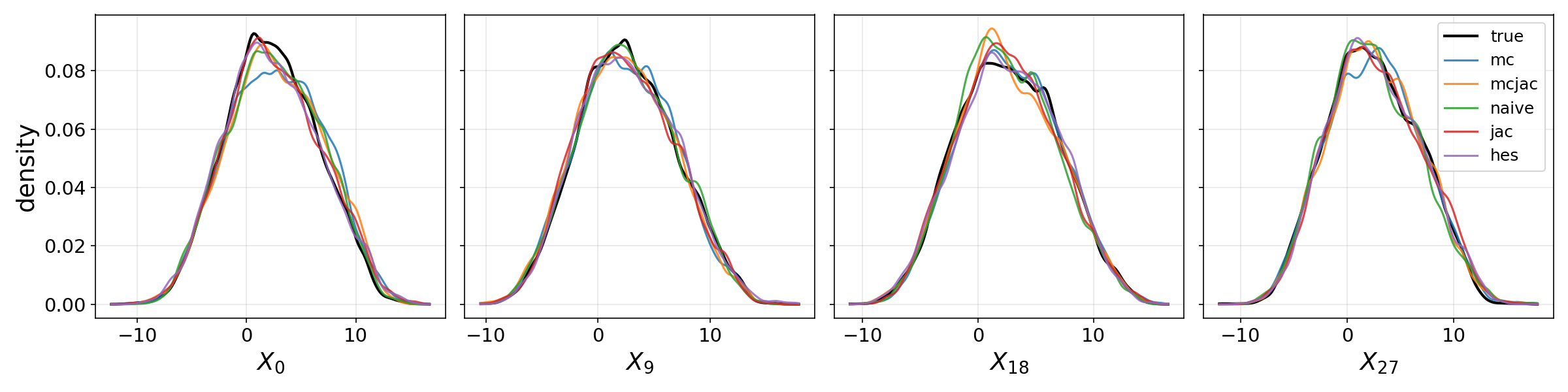}
        \caption{$F=10$}
        \figlab{l96-pdf-F10}
    \end{subfigure}
    \hfill
    \begin{subfigure}[b]{0.9\linewidth}
    \centering
    \includegraphics[width=\linewidth]{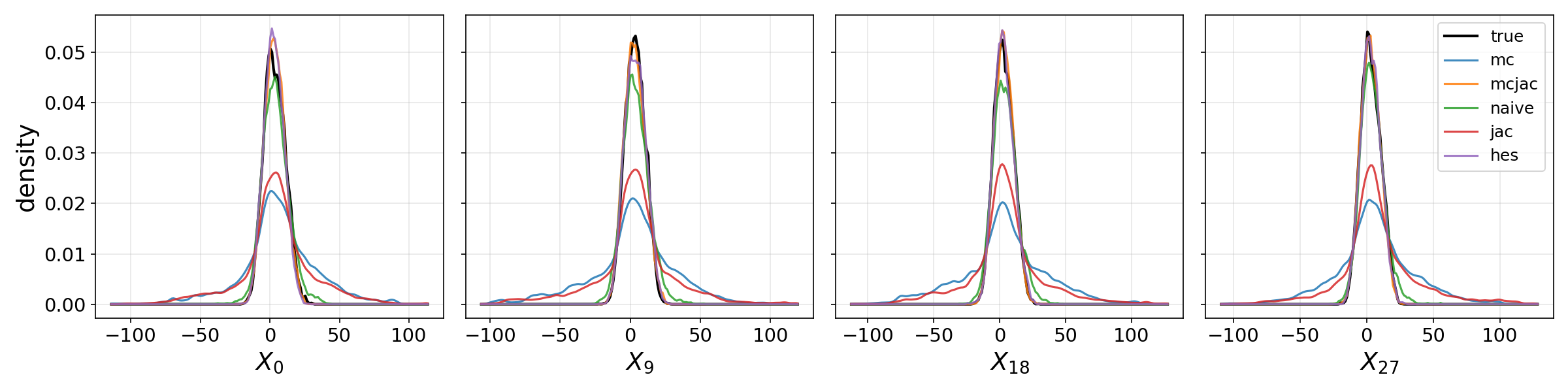}
        \caption{$F=20$}
        \figlab{l96-pdf-F20}
    \end{subfigure}
    \figlab{l96-pdf}
    \caption{Marginal densities of $X$ at four representative sites ($X_0$, $X_9$, $X_{18}$, and $X_{27}$). (a) At $F=10$, all methods overlap closely with the true density (black). (b) At $F=20$, \texttt{mcjac} and \texttt{hes} remain in close agreement with the true density, \texttt{mc} and \texttt{jac} produce broadened marginals with heavy tails, and \texttt{naive} remains comparatively close to the true density at these representative sites.     
    }
\end{figure}

\begin{figure}
    \centering  
    \begin{subfigure}[b]{0.9\linewidth}
    \centering
    \includegraphics[width=\linewidth]{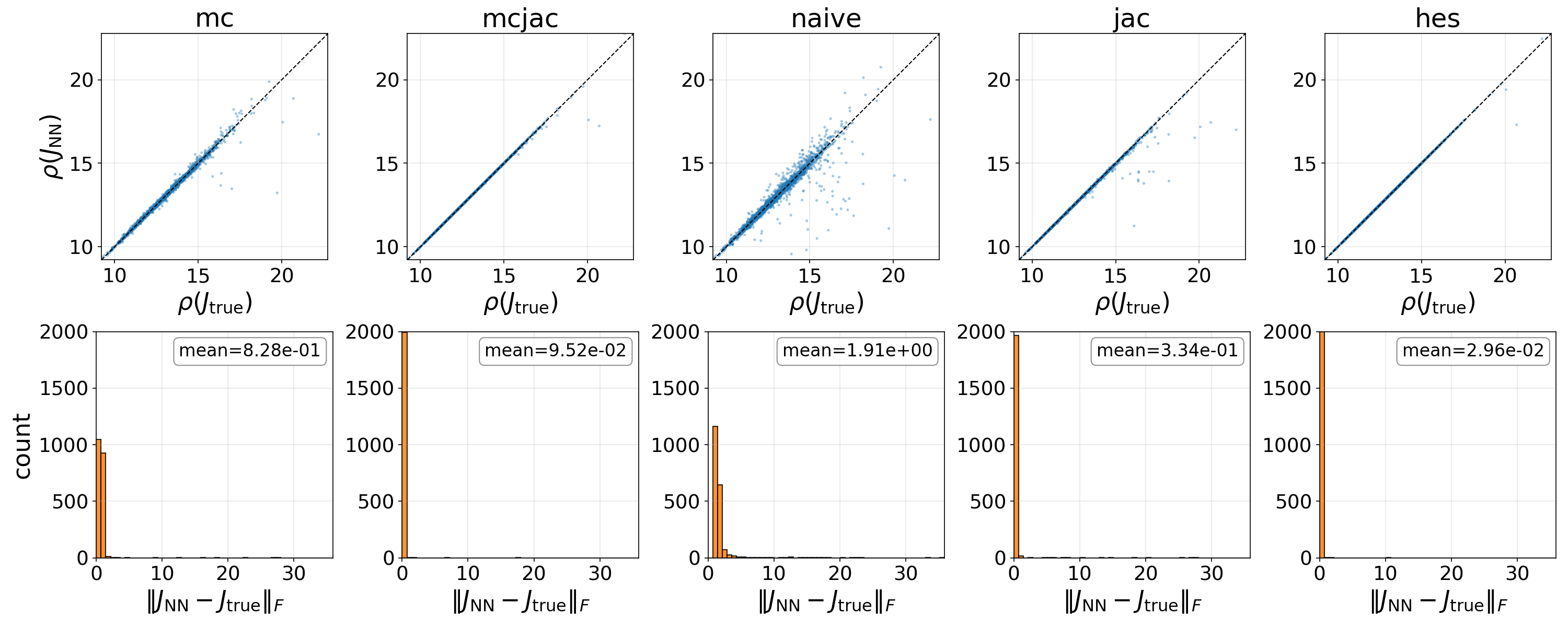}
        \caption{$F=10$}
        \figlab{l96-jac-F10}
    \end{subfigure}
    \hfill
    \begin{subfigure}[b]{0.9\linewidth}
    \centering
    \includegraphics[width=\linewidth]{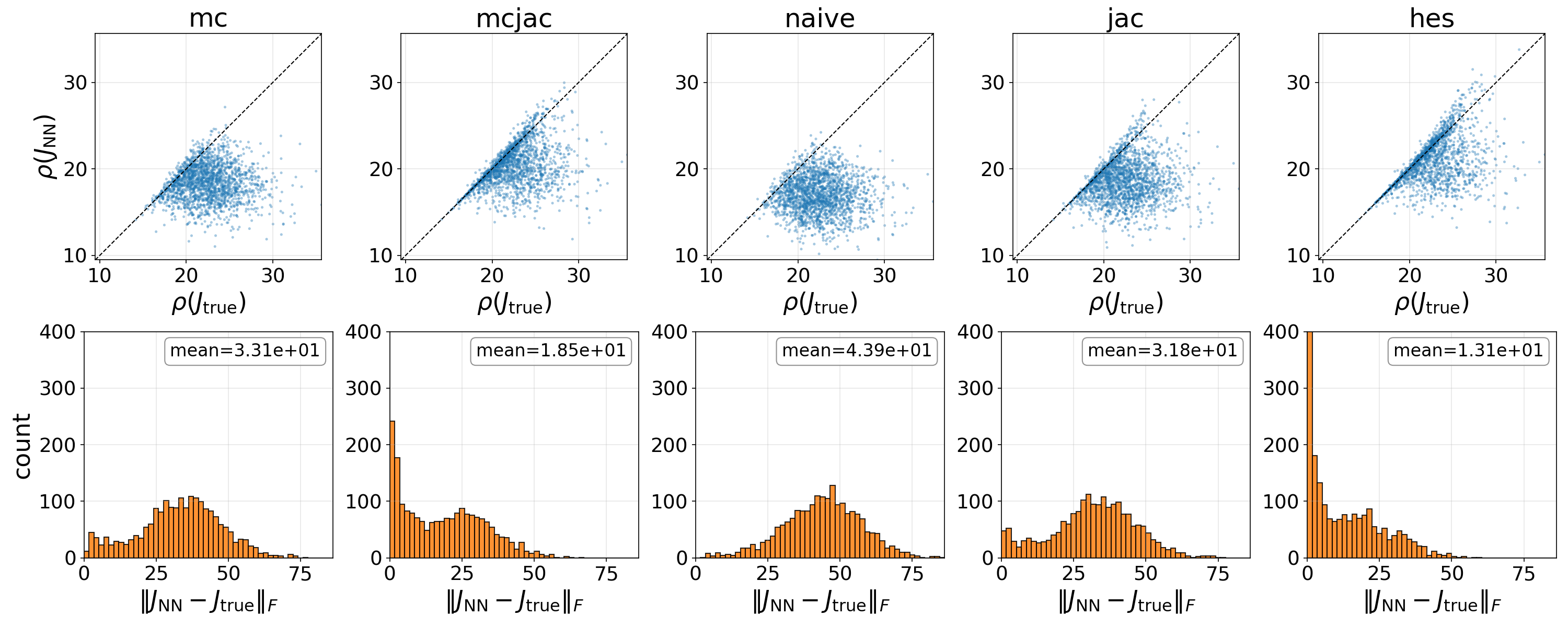}
        \caption{$F=20$}
        \figlab{l96-jac-F20}
    \end{subfigure}
    \caption{Jacobian recovery at (a) $F=10$ and (b) $F=20$. Scatter plots of the learned versus true Jacobian spectral radius (top) and histograms of the Frobenius-norm Jacobian error (bottom) are shown. At $F=10$, all methods are broadly aligned with the diagonal, but the Frobenius means differ by nearly two orders of magnitude ($0.03$ for \texttt{hes} to $1.91$ for \texttt{naive}). At $F=20$, \texttt{hes} and \texttt{mcjac} remain closest to the diagonal, while \texttt{mc}, \texttt{jac}, and especially \texttt{naive} show increasingly distorted Jacobian spectra. In the Frobenius-error distributions, \texttt{hes} is most concentrated near zero, \texttt{mcjac} retains a strong low-error peak with a broader secondary component, \texttt{mc} and \texttt{jac} shift to broad intermediate-error distributions, and \texttt{naive} has the largest errors overall.}
    \figlab{l96-jac}
\end{figure}

The next diagnostic examines the learned Jacobian.  
At $F=10$ (Figure~\figref{l96-jac-F10}), all methods recover the qualitative structure of the true Jacobian, but the mean Frobenius errors differ, ranging from $0.03$ for \texttt{hes} and $0.095$ for \texttt{mcjac} to $0.33$ for \texttt{jac}, $0.83$ for \texttt{mc}, and $1.91$ for \texttt{naive}. The same ordering is observed for $F=20$ (Figure~\figref{l96-jac-F20}), where the larger absolute errors reflect the increased Jacobian norm at higher forcing.
In both regimes, the smallest Jacobian errors are achieved not by \texttt{jac}, which explicitly supervises the Jacobian, but by the second-order methods \texttt{hes} and \texttt{mcjac}. The histograms sharpen this distinction: \texttt{hes} is strongly concentrated near zero error, \texttt{mcjac} retains a dominant low-error peak with a weaker secondary lobe, whereas \texttt{mc}, \texttt{jac}, and \texttt{naive} produce much broader distributions centered farther from zero. Thus, the methods incorporating second-order information yield larger fractions of trajectory states with near-exact Jacobian recovery than the first-order methods.

The curvature analysis above explains this pattern. First-order supervision constrains the learned vector field at the supervised states but leaves its spatial variation between those states underconstrained. By constraining the curvature, second-order supervision indirectly regularizes how the Jacobian varies across phase space, yielding a smoother and globally more faithful first-order structure. Another factor is that the local MLP architecture already encodes the translation equivariance and local interaction structure of the quadratic transport term, reducing the remaining ambiguity that supervision must resolve. As a result, in the moderately chaotic regime ($F=10$), even the lower-order methods reproduce many macroscopic statistical properties despite noticeable differences in local derivative accuracy. The distinction between supervision strategies becomes clearer in the more chaotic out-of-distribution regime at $F=20$.

\begin{figure}
    \centering    \includegraphics[width=0.9\linewidth]{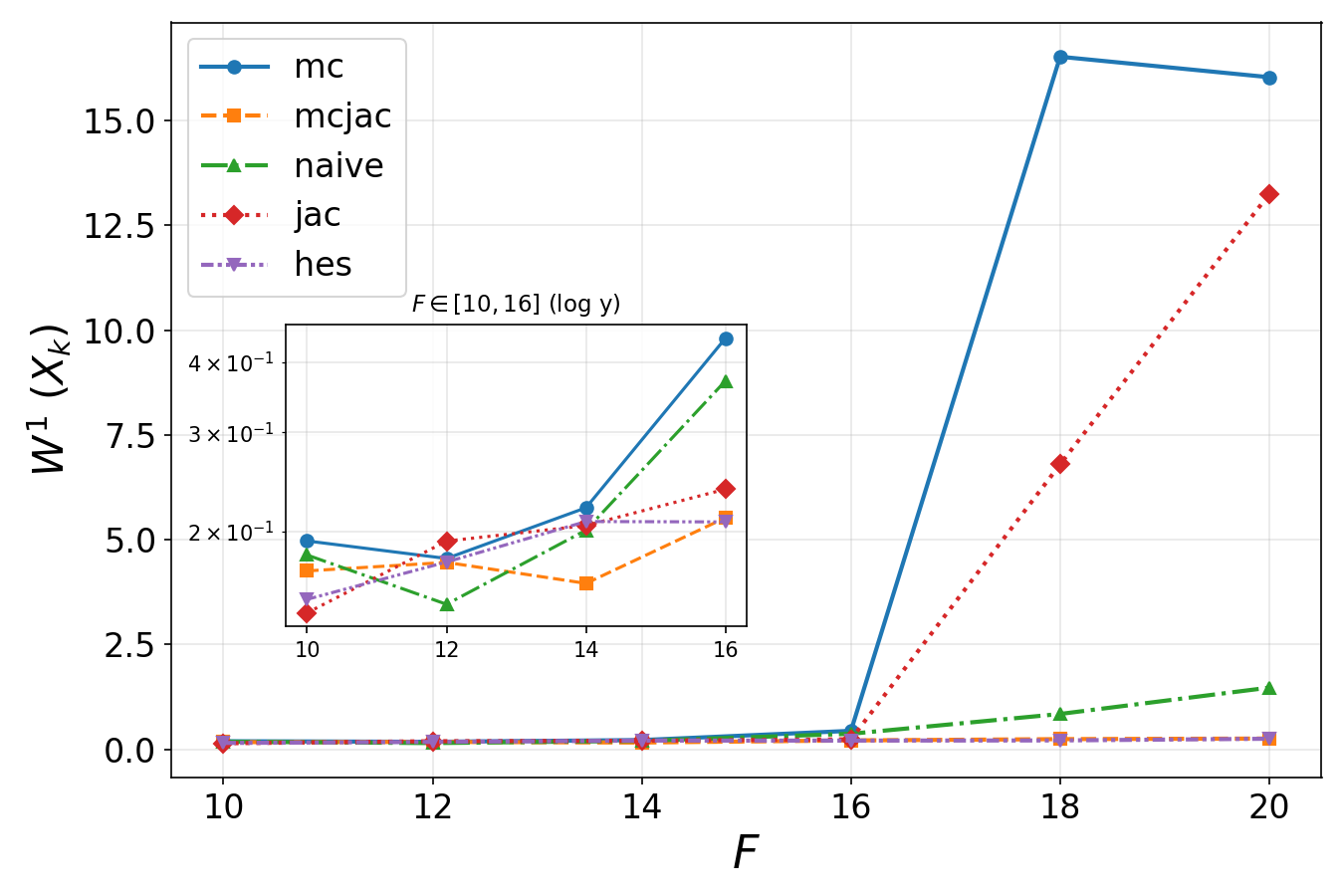}
    \caption{Site-averaged Wasserstein-1 distance between learned and true marginal distributions of $X_k$ as a function of forcing $F$. The inset shows the mild regime $F\in[10,16]$. For $F\leq 16$ all five methods produce $W^1\lesssim 0.4$ and remain essentially indistinguishable. Between $F=16$ and $F=18$, \texttt{mc} and \texttt{jac} undergo an abrupt increase of one to two orders of magnitude, \texttt{naive} degrades moderately, and \texttt{mcjac} and \texttt{hes} remain close to the true marginals across the entire sweep.}
    \figlab{l96-wd}
\end{figure}


The analysis presented so far compares the five methods at two representative forcing values, $F=10$ and $F=20$. To characterize the transition between these regimes, we now sweep $F$ continuously from $10$ to $20$ with all other settings fixed and track two complementary metrics: the site-averaged Wasserstein-1 distance $W^1$ between the learned and true marginals (Figure~\figref{l96-wd}), and the mean squared error of the top-20 Lyapunov exponents (Table~\tabref{l96-le-wc}).
For $F\le16$, all five methods are close to the reference, with $W^1\lesssim0.4$ and Lyapunov MSEs below $0.5$. A clear separation
appears between $F=16$ and $F=18$. In $W^1$, \texttt{mc} and \texttt{jac}
increase by one to two orders of magnitude, whereas \texttt{naive} degrades
more gradually. This is the only case in our experiments where \texttt{naive}
shows less deterioration in $W^1$ than derivative-supervised methods outside
the training regime, and this occurs only relative to \texttt{mc} and
\texttt{jac}. The Lyapunov MSE shows a different ordering: at $F=18$,
\texttt{naive} reaches $10.0$ and \texttt{jac} $2.7$, while \texttt{mcjac}
and \texttt{hes} remain more accurate. The relatively mild degradation of \texttt{naive} in $W^1$ is not observed in its other dynamical statistics.
At $F=20$, \texttt{jac}, \texttt{mcjac}, and \texttt{hes} achieve comparable spectral errors near $2\times10^{-1}$, while \texttt{mc} and \texttt{naive} remain one to two orders of magnitude worse. The marginal distributions give a different ordering: only \texttt{mcjac} and \texttt{hes} remain closely aligned with the reference throughout the forcing sweep, whereas \texttt{jac} matches the Lyapunov spectrum at $F=20$ while distorting the invariant measure.
We also observe that \texttt{mcjac} achieves Hessian-level accuracy at first-order cost (Table~\tabref{l96-le-wc}).

\begin{table}[ht]
\centering
\caption{
MSE of the top-20 Lyapunov exponents and training wall-clock time for each
method across forcing $F$. The errors are evaluated at
$F\in\{10,16,18,20\}$, while wall-clock times are reported in seconds.
Cell shading indicates relative magnitude within each column on a logarithmic
scale, with darker shading corresponding to smaller values. Blue shading is
used for Lyapunov-exponent errors and gray shading for training time.
}
\begin{tabular}{c|cccc|c}
\hline
method & $F=10$ & $F=16$ & $F=18$ & $F=20$ & wc (s) \\
\hline

\texttt{mc}
& \cellcolor{heat!70} {\bf 2.051E-02}
& \cellcolor{heat!22} 3.950E-01
& \cellcolor{heat!56} 4.733E-01
& \cellcolor{heat!44} 1.935E+00
& \cellcolor{timeheat!30} 147.6 \\

\texttt{mcjac}
& \cellcolor{heat!38} 3.634E-02
& \cellcolor{heat!15} 4.843E-01
& \cellcolor{heat!70} {\bf 1.751E-01}
& \cellcolor{heat!70} {\bf 2.192E-01}
& \cellcolor{timeheat!28} 157.6 \\

\texttt{naive}
& \cellcolor{heat!41} 3.485E-02
& \cellcolor{heat!48} 1.899E-01
& \cellcolor{heat!15} 1.001E+01
& \cellcolor{heat!15} 2.294E+01
& \cellcolor{timeheat!70} {\bf 30.3} \\

\texttt{jac}
& \cellcolor{heat!65} 2.247E-02
& \cellcolor{heat!70} {\bf 1.023E-01}
& \cellcolor{heat!33} 2.748E+00
& \cellcolor{heat!67} 2.807E-01
& \cellcolor{timeheat!30} 146.9 \\

\texttt{hes}
& \cellcolor{heat!15} 5.528E-02
& \cellcolor{heat!37} 2.593E-01
& \cellcolor{heat!51} 7.213E-01
& \cellcolor{heat!70} 2.217E-01
& \cellcolor{timeheat!15} 268.3 \\

\hline
\end{tabular}
\tablab{l96-le-wc}
\end{table}





\section{Conclusion}
\seclab{conclusion}

Trajectory and Jacobian matching leave the higher-order geometry of a learned vector field underconstrained. On both Lorenz~63 and coupled Lorenz~96, models trained with these objectives may reproduce short-time trajectories and local tangent behavior while developing spurious attractors, distorted invariant measures or Lyapunov spectra, and unreliable long-time dynamics.

We proposed randomized Jacobian matching, which extends model-constrained training to implicit second-order supervision by matching Jacobians at perturbed states. The resulting loss contains a leading-order Hessian-mismatch penalty without computing Hessian tensors, making second-order supervision feasible in higher dimensions.

On Lorenz~63 under minimal temporal supervision, adding second-order information reduced the invariant-measure error and Lyapunov-spectrum MSE in both paired comparisons and recovered the constant Hessian norm of the true bilinear field. Increasing the rollout length to $m=2$ made the first-order methods competitive in ensemble-averaged metrics. 

The five-seed experiment showed that curvature supervision at sampled states does not eliminate catastrophic outliers. Explicit Hessian matching produced Lyapunov outliers for four of five training seeds, whereas randomized Jacobian matching produced none among $5{,}000$ on-attractor rollouts. A representative outlier passed through the saddle region near the origin before entering a compact spurious attracting set. The directional capture threshold increased from \texttt{mc} to \texttt{mcjac} and from \texttt{jac} to \texttt{hes}. Randomized Jacobian matching had the largest threshold. Its advantage over explicit Hessian matching may reflect the neighborhood coverage provided by the perturbed inputs.

In coupled Lorenz~96, \texttt{mc} and \texttt{jac} entered spurious high-amplitude regimes under out-of-distribution forcing, whereas \texttt{mcjac} and \texttt{hes} agreed with the reference across the forcing sweep. At $F=20$, \texttt{jac} matched the Lyapunov spectrum but distorted the invariant measure, showing that the two diagnostics are complementary. Randomized Jacobian matching achieved long-time accuracy comparable to that of explicit Hessian matching at about $40\%$ lower training cost.

The coupled Lorenz~96 experiments used a structured universal differential equation setting in which only the quadratic advection term was learned. Whether the method extends to fully learned high-dimensional systems and spatiotemporally chaotic partial differential equations remains open. We selected the perturbation scale and loss weights separately for each benchmark but did not systematically study their dependence on the integration step and rollout length. The Lorenz~63 robustness analysis was limited to five training seeds, and the capture scan considered a single perturbation direction. Additional seeds and directions are needed to estimate failure rates and test whether neighborhood coverage explains the difference between randomized Jacobian and explicit Hessian matching. Deterministic, antithetic, and quasi-Monte Carlo perturbation ensembles may also reduce sampling variance.

\begin{acknowledgments}
This work was supported by the National Research Foundation of Korea (NRF) grant funded by the Korea government (MSIT) (RS-2026-25588771), and by Korea University Grants (Nos. K2411041, K2414071, K2425851, K2514791).
\end{acknowledgments}

\section*{Data Availability Statement}
A reference implementation of the proposed loss functions, together with a minimal demonstration on the Lorenz~63 system and the stored evaluation array underlying the five-seed comparison reported here, is available at \url{https://github.com/diffcl/mcjac}. The remaining scripts used to produce the results reported in this paper are available from the corresponding author upon reasonable request.

\appendix

\section{Hyperparameter Settings}

To ensure the reproducibility of the reported results, we provide the detailed hyperparameter configurations for \texttt{naive}, \texttt{mc}, \texttt{mcjac}, \texttt{jac}, and \texttt{hes}. Tables~\tabref{l63-hparams} and~\tabref{l96-hparams} summarize the specific configurations used in the Lorenz~63 and the coupled Lorenz~96 examples.

\begin{table}[ht]
\centering
\caption{Hyperparameters for the Lorenz~63 experiments ($m=1,2$).}
\tablab{l63-hparams}
\begin{ruledtabular}
\begin{tabular}{l|cccccc}
Method & $\alpha_{\mathrm{mc}}$ & $\alpha_{\mathrm{mcjac}}$ & $\alpha_{\mathrm{jac}}$ & $\alpha_{\mathrm{hes}}$ & $\eta$ & learning rate \\
\colrule
\texttt{mc}    & $100$ & --    & --    & --    & $1.0$ & $1\times 10^{-3}$ \\
\texttt{mcjac} & $100$ & $100$ & --    & --    & $0.5$ & $1\times 10^{-3}$ \\
\texttt{naive} & --    & --    & --    & --    & --    & $1\times 10^{-3}$ \\
\texttt{jac}   & --    & --    & $1.0$ & --    & --    & $1\times 10^{-3}$ \\
\texttt{hes}   & --    & --    & $1.0$ & $0.1$ & --    & $1\times 10^{-3}$ \\
\end{tabular}
\end{ruledtabular}
\end{table}

\begin{table}[ht]
\centering
\caption{Hyperparameters for the coupled Lorenz 96 experiments ($m=2$).}
\label{tab:l96-hparams}
\begin{ruledtabular}
\begin{tabular}{l|cccccc}
Method & $\alpha_{\mathrm{mc}}$ & $\alpha_{\mathrm{mcjac}}$ & $\alpha_{\mathrm{jac}}$ & $\alpha_{\mathrm{hes}}$ & $\eta$ & learning rate \\
\colrule
\texttt{mc}    & $100$  & --  & --    & --  & $0.9$ & $5\times 10^{-3}$ \\
\texttt{mcjac} & $0.01$ & $1$ & --    & --  & $0.9$ & $5\times 10^{-3}$ \\
\texttt{naive} & --     & --  & --    & --  & --    & $1\times 10^{-2}$ \\
\texttt{jac}   & --     & --  & $0.01$ & -- & --    & $1\times 10^{-3}$ \\
\texttt{hes}   & --     & --  & $0.01$ & $1$ & --    & $5\times 10^{-3}$ \\
\end{tabular}
\end{ruledtabular}
\end{table}

\section{Finite-window behavior of the curvature length-scale statistic}
\seclab{lcurv}

\begin{figure}
    \centering  
    \begin{subfigure}[b]{0.9\linewidth}
    \centering
    \includegraphics[width=\linewidth]{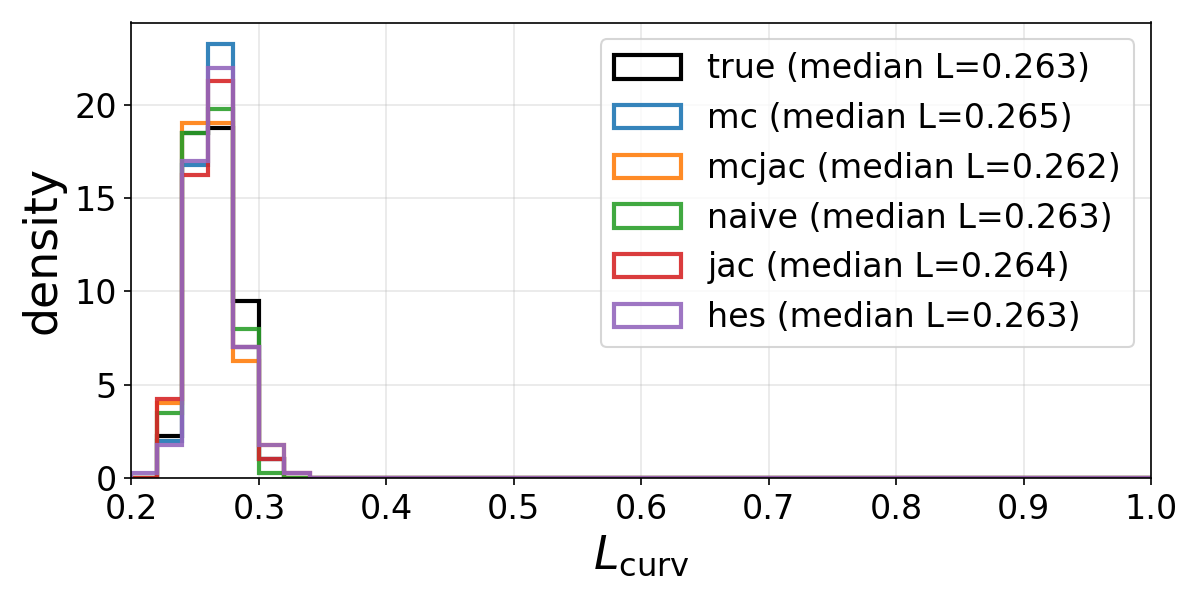}
        \caption{$F=10$}
        \figlab{l96-lcurv-F10}
    \end{subfigure}
    \hfill
    \begin{subfigure}[b]{0.9\linewidth}
    \centering
    \includegraphics[width=\linewidth]{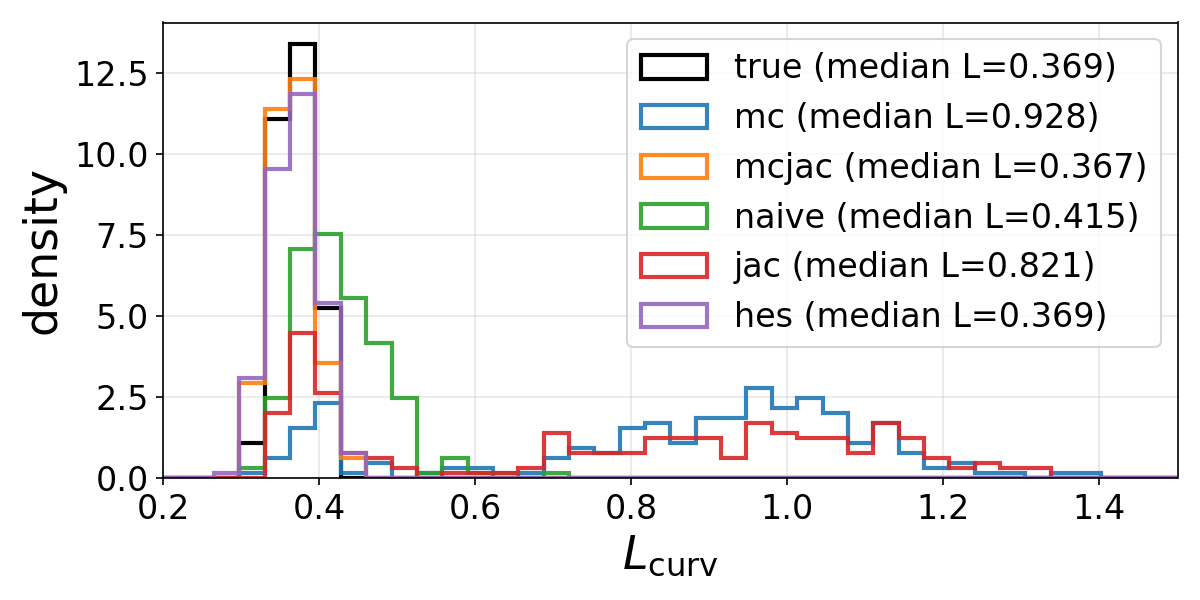}
        \caption{$F=20$}
        \figlab{l96-lcurv-F20}
    \end{subfigure}
    \caption{
    Finite-window distributions of the curvature length scale \(\Lcurv\) along the reference and learned trajectories.
    Each distribution is computed from 200 equally spaced states over \(t\in[10,500]\) after burn-in.
    (a) At \(F=10\), all methods give similar single-peaked distributions centered near \(\Lcurv\approx0.263\).
(b) At \(F=20\), \texttt{mcjac} and \texttt{hes} remain close to the reference distribution, whereas the \texttt{naive} distribution has a median of $0.415$. The \texttt{mc} and \texttt{jac} distributions are bimodal, with one mode near the reference value and a broader mode at larger
\(\Lcurv\).
}
    \figlab{l96-lcurv}
\end{figure}

Figure~\figref{l96-lcurv} shows the finite-window distributions of \(\Lcurv\) along the model-specific rollouts of the
Lorenz~96 system. 
Each learned field is evaluated on the states visited by its own rollouts. Thus, these distributions contain both field and state effects.
At $F=10$, all methods yield nearly identical distributions centered near $\Lcurv\approx0.263$ as shown in Figure~\figref{l96-lcurv-F10}. 
At $F=20$, \texttt{mcjac} and \texttt{hes} are centered near the true peak at $\Lcurv\approx0.369$, while the \texttt{naive} distribution has a median of $0.415$.
The \texttt{mc} and \texttt{jac} distributions each split into a main mode near the reference value and a
broader second lobe spanning approximately $0.7$--$1.4$, resulting in medians of $0.928$ and $0.821$, respectively.
  


\section{Seed Dependence of Lorenz~63 Statistical Biases}
\seclab{l63-seed-bias}

Since all \texttt{naive} rollouts leave the target attractor, we report the invariant-measure statistics of the other four methods evaluated over five random seeds each. 
The Lorenz~63 equations have a symmetry that exchanges the two lobes: if $(x(t),y(t),z(t))$ is a trajectory, then $(-x(t),-y(t),z(t))$ is also a
trajectory, which have equal mass under the physical invariant measure. We measure deviations from this symmetry by the signed lobe-mass imbalance.
\begin{equation*}
  \Delta p_{\mathrm{lobe}}^{(m,s)}
  :=\widehat\mu_{\widehat f_{m,s}}(x>0)
   -\widehat\mu_f^{(m,s)}(x>0)
\end{equation*}
where $\widehat\mu_{\widehat f_{m,s}}$ is the empirical measure generated by
method $m$ with seed $s$, and $\widehat\mu_f^{(m,s)}$ is the empirical measure
of the matching true rollouts.  Failed learned rollouts and their true
counterparts are excluded.  The symmetry predicts no preferred sign.

Table~\ref{tab:l63-lobe-bias} shows both signs for every method.  Across all $20$ combinations, the signed average is $-0.0024$ ($-0.24$ percentage points). Thus, individual results favor one lobe or the other, but the direction shows no systematic dependence on derivative supervision.
The mean row also reports the ensemble-averaged signed bias over the five seeds for each method.

\begin{table}[ht]
  \centering
  \caption{Signed lobe-mass bias $\Delta p_{\mathrm{lobe}}$ for each method and seed.}
  \label{tab:l63-lobe-bias}
  \begin{ruledtabular}
  \begin{tabular}{lcccc}
    Seed & \texttt{mc} & \texttt{mcjac} & \texttt{jac} & \texttt{hes} \\
    \colrule
    $32$ & $-0.0272$ & $+0.0143$ & $-0.0100$ & $+0.0011$ \\
    $33$ & $+0.0042$ & $-0.0234$ & $+0.0279$ & $-0.0390$ \\
    $34$ & $+0.0021$ & $+0.0138$ & $-0.0328$ & $+0.0161$ \\
    $35$ & $+0.0341$ & $+0.0078$ & $+0.0053$ & $-0.0243$ \\
    $36$ & $-0.0300$ & $-0.0041$ & $+0.0148$ & $+0.0020$ \\
    \colrule
    Mean
      & $-0.0034$ & $+0.0017$ & $+0.0010$ & $-0.0088$ \\
  \end{tabular}
  \end{ruledtabular}
\end{table}

\section{Reverse forcing experiment}
\seclab{l96-reverse}

To separate the effect of chaoticity and out-of-distribution distance, we repeated the Lorenz~96 experiment in reverse, training all methods
at $F=20$ and evaluating over $F\in\{10,12,14,16,18,20\}$.

For each forcing $F$, we compare the learned and exact Jacobian and Hessian of the local advection term \((X_{k+1}-X_{k-2})X_{k-1}\) on states sampled from the true trajectory. We use $20{,}000$ states for the Jacobian and the Hessian at a given $F$. We report the RMS errors of the learned Jacobian and Hessian, normalized by the RMS magnitude of the corresponding exact derivatives.

\begin{figure}[h!t!b!]
    \centering
    \begin{subfigure}[b]{0.48\linewidth}
        \centering
        \includegraphics[trim=0cm 0.0cm .0cm 0.0cm,clip=true,width=\linewidth]{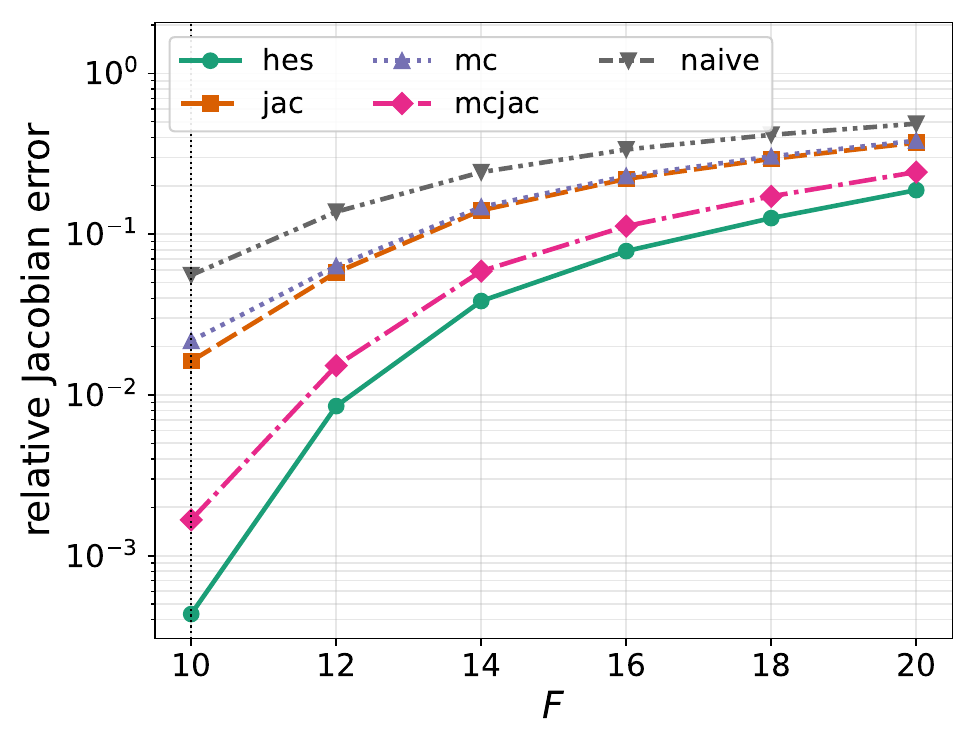}
        \caption{Jacobian error, trained at $F=10$}
        \figlab{f10-jacerr}
    \end{subfigure}
    \begin{subfigure}[b]{0.48\linewidth}
        \centering
        \includegraphics[trim=0cm .0cm .0cm .0cm,clip=true,width=\linewidth]{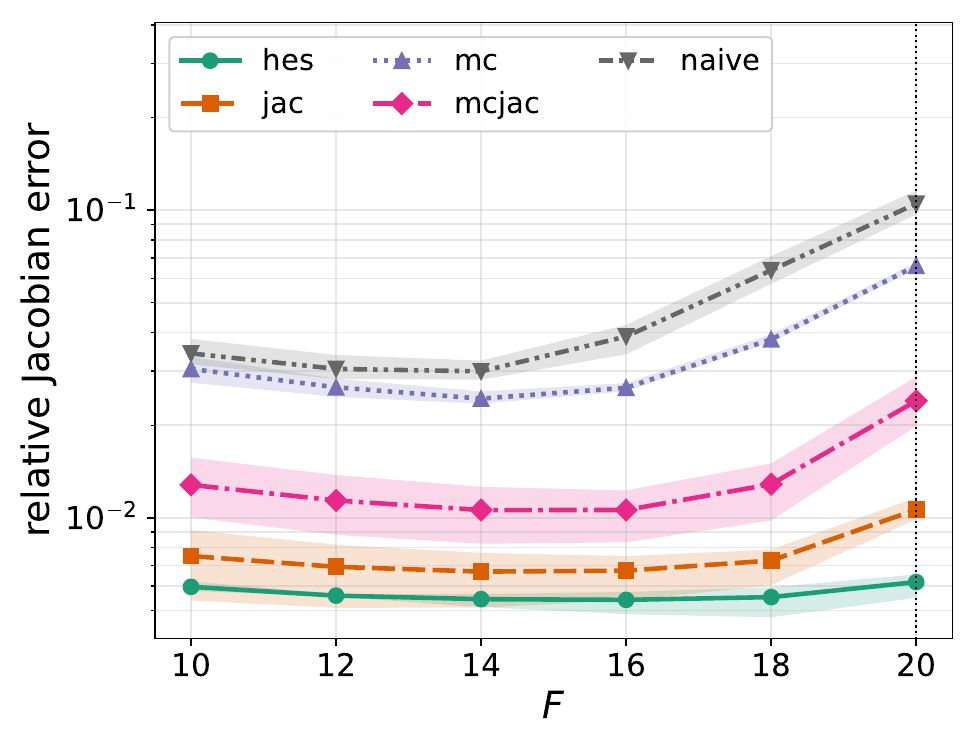}
        \caption{Jacobian error, trained at $F=20$}
        \figlab{f20-jacerr}
    \end{subfigure} \\
    \begin{subfigure}[b]{0.48\linewidth}
        \centering
        \includegraphics[trim=0cm 0.cm 0.cm .0cm,clip=true,width=\linewidth]{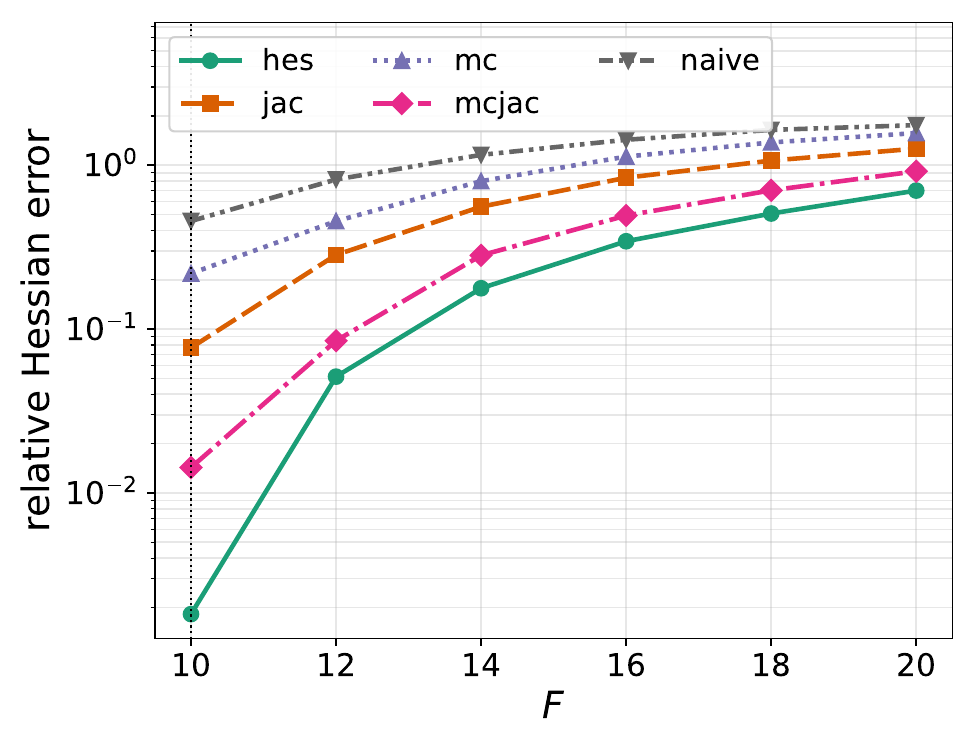}
        \caption{Hessian error, trained at $F=10$}
        \figlab{f10-heserr}
    \end{subfigure}
    \begin{subfigure}[b]{0.48\linewidth}
        \centering
        \includegraphics[trim=0cm 0.cm 0.cm 0.0cm,clip=true,width=\linewidth]{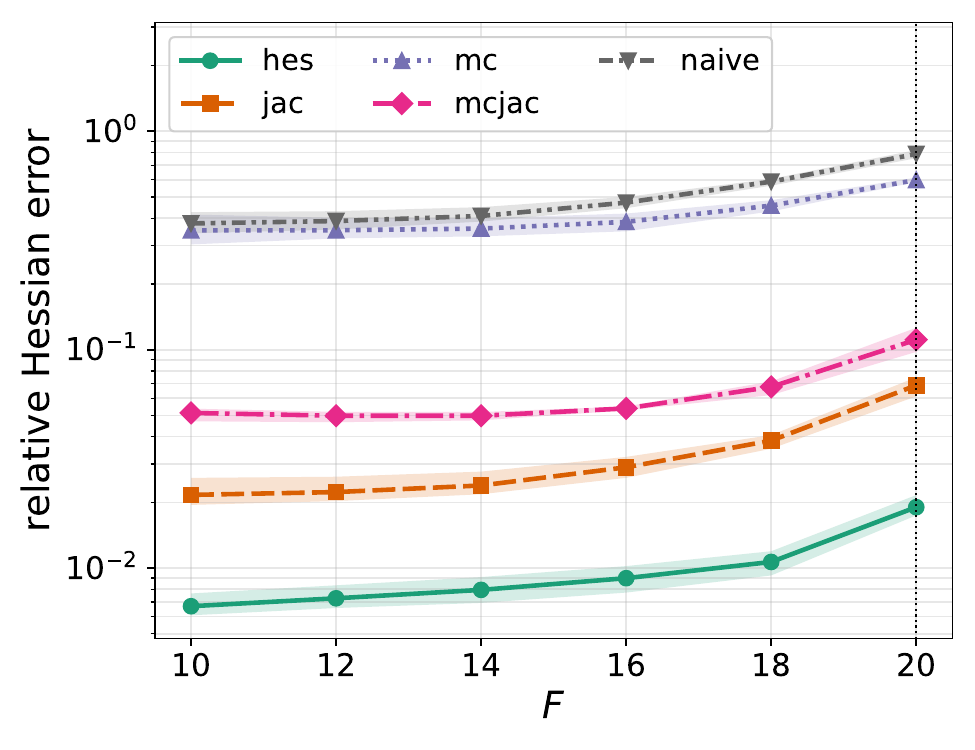}
        \caption{Hessian error, trained at $F=20$}
        \figlab{f20-heserr}
    \end{subfigure}
    \caption{Relative Jacobian and Hessian errors for models trained at $F=10$ (left) and $F=20$ (right). Dashed lines indicate the training forcing. Bands for the $F=20$ models show the min--max range over three training seeds, whereas the $F=10$ results are from a single training seed. In both training forcings, the derivative errors generally increase with $F$, rather than with the distance from the training forcing.
    }
    \figlab{l96-reverse-F}
\end{figure}

Figure~\figref{l96-reverse-F} shows that reversing the direction of
extrapolation does not reverse the direction of error growth. Whether trained at $F=10$ or $F=20$, the derivative errors generally increase toward larger $F$. Thus, the degradation follows the forcing level consistent with the larger Lyapunov exponent and stronger chaoticity at larger $F$. The reverse experiment does not, however, show that the relative advantage of second-order over first-order supervision
systematically increases with chaoticity.

\section{Field and Trajectory Effects on the Lyapunov Spectrum}
\seclab{l96-spectrum-decomposition}

The Lyapunov spectrum depends on both the local Jacobian field and the
trajectory along which those Jacobians are sampled. Because the temporal
ordering of the sampled Jacobians also matters, we write $\Lambda(g,X)$
for the Benettin spectrum obtained by evolving the tangent dynamics with
$\nabla g$ evaluated along the ordered trajectory $X$.
Let $X_f$ and $X_{\hat f}$ denote trajectories generated by the true and learned vector fields, respectively. $\Lambda(f,X_f)$ and $\Lambda(\hat f,X_{\hat f})$ are the Lyapunov spectra of the true and learned dynamical systems. $\Lambda(\hat f, X_f)$ evaluates the learned Jacobian along the true trajectory, whereas $\Lambda(f,X_{\hat f})$ evaluates the true Jacobian along the learned trajectory.

\subsection*{Lorenz 63: field--trajectory cross evaluation}
\label{app:l63-lyap-cross}

Table~\ref{tab:l63-lyap-cells-off-m1} reports the leading Lyapunov exponent for the Lorenz~63 models at $m=1$, using the same $1000$ off-attractor initial
conditions. All entries, including the true reference $\Lambda_1(f,X_f)=0.9046$, are medians over the $1000$ trajectories; the mean reference value 0.9041 is used in the main-text error metrics.

\begin{table}[t]
\centering
\caption{Cross evaluation of the leading Lyapunov exponent for Lorenz~63 with $m=1$. Entries are medians over the $1000$ evaluation trajectories. The true reference median is $\Lambda_1(f,X_f)=0.9046$.
For \texttt{naive}, all learned rollouts fail, so $\Lambda_1(\hat f,X_{\hat f})$ is undefined.}
\label{tab:l63-lyap-cells-off-m1}
\begin{tabular}{c|ccc}
\hline
method
& $\Lambda_1(\hat f,X_f)$
& $\Lambda_1(f,X_{\hat f})$
& $\Lambda_1(\hat f,X_{\hat f})$ \\
\hline
\texttt{mc}    & 0.9477 & 0.8123  & 0.8974 \\
\texttt{mcjac} & 0.8819 & 0.8991  & 0.9043 \\
\texttt{naive} & 0.9342 & 61.3286 & N/A \\
\texttt{jac}   & 0.8817 & 0.8914  & 0.9127 \\
\texttt{hes}   & 0.8242 & 0.9865  & 0.9067 \\
\hline
\end{tabular}
\end{table}

The comparison shows that the trajectory on which the learned Jacobian is sampled can change the resulting Lyapunov exponent. For \texttt{hes}, for example, the same learned Jacobian gives $\Lambda_1(\hat f,X_f)=0.8242$ along the true trajectory, but $\Lambda_1(\hat f,X_{\hat f})=0.9067$ along the learned trajectory, the latter being close to the reference value $0.9046$. A similar dependence is visible for \texttt{jac}, for which the corresponding values are $0.8817$ and $0.9127$. Thus, the Lyapunov exponent of the learned system is not determined by the accuracy of its Jacobian along the true trajectory alone. Even with the learned Jacobian field held fixed, changing from $X_f$ to $X_{\hat f}$ changes the sequence of states at which that Jacobian is sampled and can therefore change the resulting exponent.

\subsection*{Lorenz~96: full-spectrum trajectory effect at $F=20$}

We next examine whether the same trajectory effect appears across the leading
Lyapunov spectrum in the more strongly chaotic Lorenz~96 case at $F=20$.
For each learned vector field $\hat f$, we compute the top-20 spectrum using
the same learned Jacobian $\nabla\hat f$ along either its own rollout
$X_{\hat f}$ or the true trajectory $X_f$.
We define the top-20 spectral error as
$
\varepsilon_{\lambda}(Y)
=
\frac{1}{20}
\sum_{k=1}^{20}
\left(
\hat{\lambda}_k(Y)-\lambda_k^{\rm true}
\right)^2,
$
where $Y=X_{\hat f}$ or $X_f$ denotes the trajectory along which the learned
Jacobian is evaluated.

\begin{table}[h]
\centering
\caption{
Top-20 Lyapunov-spectrum MSE at $F=20$ using the learned Jacobian evaluated
along the learned rollout $X_{\hat f}$ or the true trajectory $X_f$.
The reduction is defined as
$1-\varepsilon_{\lambda}(X_f)/\varepsilon_{\lambda}(X_{\hat f})$.
}
\label{tab:l96-spectrum-decomposition}
\begin{tabular}{lccc}
\hline
method
& $\varepsilon_{\lambda}(X_{\hat f})$
& $\varepsilon_{\lambda}(X_f)$
& reduction \\
\hline
\texttt{mc}     & 1.935E+00 & 2.945E-02 & 98.5\% \\
\texttt{mcjac}  & 2.192E-01 & 1.774E-02 & 91.9\% \\
\texttt{naive}  & 2.294E+01 & 6.762E-02 & 99.7\% \\
\texttt{jac}    & 2.807E-01 & 3.194E-02 & 88.6\% \\
\texttt{hes}    & 2.217E-01 & 1.493E-02 & 93.3\% \\
\hline
\end{tabular}
\end{table}

Table~\ref{tab:l96-spectrum-decomposition} reports the corresponding spectral errors. Evaluating the learned Jacobian along the true trajectory reduces $\varepsilon_{\lambda}$ by $88.6$--$99.7\%$ across all five methods. When the same learned Jacobian is evaluated along the true trajectory instead of the learned rollout, the spectral error drops sharply for every method. This shows that a part of the Lyapunov-spectrum error comes from the learned trajectory visiting different states, not only from errors in the learned Jacobian.

\bibliography{references}

\end{document}